\documentclass[12pt]{amsart}
\usepackage{comment}

\def\GL{\mathrm{GL}}
\def\SL{\mathrm{SL}}
\def\SU{\mathrm{SU}}
\def\OC{\mathcal{O}}
\def\CZ{\mathcal{Z}}

\def\bM{{\mathbf{M}}}
\def\II{\mathrm{i}}

\def\br{\operatorname{br}\nolimits}
\def\Br{\operatorname{Br}\nolimits}
\def\Ho{\operatorname{Ho}\nolimits}
\def\redu{\operatorname{red}\nolimits}
\def\St{\operatorname{St}\nolimits}
\def\Irr{\operatorname{Irr}\nolimits}
\def\Aut{\operatorname{Aut}\nolimits}
\def\bbchi{\boldsymbol{\chi}}
\def\bblambda{\boldsymbol{\lambda}}

\usepackage[latin1]{inputenc}
\usepackage[T1]{fontenc}
\usepackage{dsfont}
\usepackage{amsxtra}
\usepackage{amsmath}
\usepackage{amscd}
\usepackage{amssymb}
\usepackage{amsfonts}
\usepackage{mathrsfs}
\usepackage{amsthm} 
\usepackage{xcolor}
\usepackage{upgreek }
\usepackage{mathdots}
\usepackage{pstricks-add}
\usepackage{lscape}
\usepackage{multirow}
\usepackage{rotating}
\usepackage{xypic}
\usepackage{enumitem}
\usepackage[colorlinks=true,linkcolor=magenta,citecolor=blue]{hyperref}

\xyoption{all}

\theoremstyle{plain}
\newtheorem{theorem}{Theorem}[section]
\newtheorem{lem}[theorem]{Lemma}

\newtheorem{prop}[theorem]{Proposition}

\newtheorem{cor}[theorem]{Corollary}
\numberwithin{equation}{section}

\theoremstyle{definition}
\newtheorem{definition}[theorem]{Definition}

\theoremstyle{remark}
\newtheorem{rmk}[theorem]{Remark}

\def\bbC{\mathbb{C}}
\def\bbF{\mathbb{F}}
\def\bbG{\mathbb{G}}
\def\bbL{\mathbb{L}}
\def\bbM{\mathbb{M}}
\def\bbQ{\mathbb{Q}}
\def\bbR{\mathbb{R}}
\def\bbZ{\mathbb{Z}}

\def\bB{\mathbf{B}}
\def\bH{\mathbf{H}}
\def\bG{\mathbf{G}}
\def\bP{\mathbf{P}}
\def\bT{\mathbf{T}}
\def\bL{\mathbf{L}}
\def\bU{\mathbf{U}}
\def\bV{\mathbf{V}}

\def\rX{\mathrm{X}}
\def\rY{\mathrm{Y}}

\def\simto{\overset{\sim}\to}
\def\Hc{\mathrm{H}_c}
\def\H{\mathrm{H}}

\def\ad{\operatorname{ad}\nolimits}
\def\Hom{\operatorname{Hom}\nolimits}
\def\RHom{\operatorname{RHom}\nolimits}
\def\Deg{\operatorname{Deg}\nolimits}
\def\End{\operatorname{End}\nolimits}
\def\Res{\operatorname{Res}\nolimits}

\def\Ind{\operatorname{Ind}\nolimits}
\def\Ext{\operatorname{Ext}\nolimits}
\def\Uch{\mathrm{Uch}}
\def\mstab{\operatorname{-\mathsf{stab}}\nolimits}
\def\mmod{\operatorname{-\mathsf{mod}}\nolimits}
\def\mproj{\operatorname{-\mathsf{proj}}\nolimits}

\def\Rgc{\mathrm{R}\Gamma_c}
\def\Rg{\mathrm{R}\Gamma}
\def\tRgc{\widetilde{\mathrm{R}}\Gamma_c}

\def\trait{-\hskip-2mm - \hskip-2mm -}

\newcommand{\Chevie}{{\sf Chevie}{}}

\begin{document}

\title[Brauer trees of unipotent blocks]{Brauer trees of unipotent blocks}
\date{\today}

\author{David A.~Craven}
\address{School of Mathematics, University of Birmingham, Edgbaston, Birmingham B15 2TT, UK.}
\email{d.a.craven@bham.ac.uk}

\author{Olivier Dudas}
\address{Universit\'e Paris Diderot, UFR de Math\'ematiques,
B\^atiment Sophie Germain, 5 rue Thomas Mann, 75205 Paris CEDEX 13, France.}
\email{olivier.dudas@imj-prg.fr}

\author{Rapha\"el Rouquier}
\address{UCLA Mathematics Department, Los Angeles, CA 90095-1555, USA.}
\email{rouquier@math.ucla.edu}

\thanks{The first author is supported by a Royal Society Research Fellowship. 
The second author gratefully acknowledges financial support by the ANR, Project
No ANR-16-CE40-0010-01. The third author is partly supported by the NSF
(grant DMS-1161999) and by a grant from the Simons Foundation (\#376202)}
  
\begin{abstract}In this paper we complete the determination of the Brauer trees
of unipotent blocks (with cyclic defect groups) of finite groups of Lie type.
These trees were conjectured by the first author in \cite{Cra3}. 
As a consequence, the Brauer trees of principal $\ell$-blocks of finite groups
are known for $\ell>71$.
\end{abstract}

\maketitle
\setcounter{tocdepth}{3}
\tableofcontents

\section{Introduction}

A basic problem in the modular representation theory of finite groups is
to determine decomposition matrices. The
theory of blocks with cyclic defect groups that originated with Brauer
\cite{Bra} and was completed by Dade \cite{Dade}, encodes the Morita
equivalence class of a block in a planar embedded tree. Its vertices
correspond to ordinary irreducible representations, its edges to modular
irreducible representations, and the edges containing a given vertex
correspond to the composition factors of a modular reduction of the
ordinary irreducible representation.

The prospect of determining all Brauer trees associated to finite groups is
a fundamental challenge in modular representation theory.
In 1984, Feit \cite[Theorem 1.1]{Fe2} proved that, up to \emph{unfolding}
--- broadly speaking,
taking a graph consisting of several copies of a given Brauer tree and then
 identifying all exceptional vertices -- the collection of Brauer trees of all
finite groups coincides with that of the quasisimple groups.

\smallskip

For alternating groups and their double covers, the Brauer trees are known \cite{Mul}, and for all but the two largest sporadic groups all Brauer trees are known (see \cite{HiLu}  for most of the trees). The remaining quasisimple groups, indeed the `majority' of quasisimple groups, are groups of Lie type $G(q)$: if $\ell$ is a prime dividing $|G(q)|$ then either $\ell\nmid q$ or $\ell\mid q$ --- in the latter case, for there to be an $\ell$-block with cyclic defect group we must have that $G/Z(G)=\mathrm{PSL}_2(\ell)$ and the Brauer tree is a line.

\smallskip

Thus the major outstanding problem is to determine the Brauer trees of $\ell$-blocks of groups of Lie type when $\ell\nmid q$. 
Conjecturally, all such blocks are Morita
equivalent to unipotent blocks ("Jordan decomposition of blocks").
It is known that every block is Morita equivalent to an isolated block
of a possibly non-connected reductive group
\cite{BDR}, and the case of isolated blocks with cyclic defect is currently
under investigation by the first author and Radha Kessar.

\smallskip

Here we complete the determination of the Brauer trees of
\emph{unipotent blocks} of $G(q)$. We determine in particular the trees
occurring in principal blocks. Our main theorem is the following.

\begin{theorem}\label{mainthm} Let $G$ be a finite group of Lie type and let
$\ell$ be a prime distinct from the defining characteristic.
If $B$ is a unipotent $\ell$-block of $G$ with cyclic defect groups then the planar-embedded Brauer tree of $B$ is known. Furthermore, 
the labelling of the vertices by unipotent characters in terms of Lusztig's parametrization is known. \end{theorem}

Theorem \ref{mainthm} has the following corollary.

\begin{cor}
Let $G$ be a finite group with cyclic Sylow $\ell$-subgroups.
If $\ell\neq 29,41,47,59,71$, then the (unparametrized)
Brauer tree of the principal $\ell$-block of $G$ is known.
\end{cor}

Note that a solution of the Jordan decomposition conjecture for isolated
blocks with cyclic defect would extend the previous corollary
to all blocks with cyclic defect groups of all finite groups
(for $\ell>71$ so that no sporadic groups are involved).

\bigskip
A basic method to determine decomposition matrices of
finite groups is to induce projective modules from proper subgroups.
In the case of modular representations of finite groups of Lie type in 
non-defining characteristic,
Harish-Chandra induction from standard Levi subgroups has similarly
been a very useful tool to produce projective modules. Here, we introduce a new method, based on the construction, via Deligne--Lusztig induction,
 of bounded complexes of projective
modules with few non-zero cohomology groups.
This is powerful enough to allow us to determine the decomposition matrices
of all unipotent blocks with cyclic defect groups of finite groups of Lie type.

\smallskip
In \cite{DR}, the second and third authors used Deligne--Lusztig
varieties associated to Coxeter elements to analyse representations modulo
$\ell$, where the order $d$ of $q$ modulo $\ell$ is the Coxeter number. Here,
we consider cases where that order is not the Coxeter number, but we use
nevertheless the geometry of Coxeter Deligne--Lusztig varieties, as they
are the best understood, and have certain
remarkable properties not shared by other Deligne--Lusztig varieties.

Our main result is the proof of the first author's conjecture \cite{Cra3},
in the case of blocks with cyclic defect groups. That conjecture is about
the existence of a perverse equivalence with a specific
perversity function. Using the
algorithm that determines the Brauer tree from the perversity function 
\cite{ChRou}, the first author had proposed conjectural Brauer trees and proved that his conjecture held in many cases. We complete here the proof of that conjecture.

\bigskip

\medskip

The methods we use for determining the Brauer trees are a combination of standard arguments and more recent methods developed in \cite{Du3,Du4,DR}. We start
with the subtrees corresponding to various Harish-Chandra series, giving a disjoint union of lines providing a first approximation of the tree. The difficulty lies in connecting those lines with edges labelled by cuspidal modules. Many possibilities can be ruled out by looking at the degrees of the characters and of some of their tensor products. These algebraic methods have proved to be efficient for determining most of the Brauer trees of unipotent blocks (see for instance \cite{HL,HLM}), but were not sufficient for groups of type $E_7$ and $E_8$. We overcome this problem by using the mod-$\ell$ cohomology of Deligne--Lusztig varieties and their smooth compactifications. This is done by analysing well-chosen Frobenius eigenspaces on the cohomology complexes of these varieties and extracting 
\begin{itemize}
 \item projective covers of cuspidal modules, giving the missing edges in the tree,
 \item Ext-spaces between simple modules, yielding the planar-embedded tree.
\end{itemize}
This strategy requires some control on the torsion part of the cohomology groups, and for that reason we must focus on small-dimensional Deligne--Lusztig varieties only (often associated with Coxeter elements). 

\smallskip
The simplest statement is obtained when the order of a Coxeter torus and
the order of proper Levi subgroups are
prime to $\ell$. In that case, we are able
to determine
part of the tree (Corollary \ref{co:Coxetertriv}). The most delicate part is
the last statement below. It involves the
planar embedding of the tree and unipotent representations 
corresponding to conjugate eigenvalues of the Frobenius. We show that

\begin{itemize}
\item there is a line $L$ starting with the trivial module $L_0=K$,
continuing with $r$(=$\mathbb{F}_q$-rank of the group)
principal series unipotent representations $L_1,\ldots,L_r$, the last of
which $L_r$ is the Steinberg representation $\mathrm{St}$.
\item $\mathrm{St}$ is connected to the
non-unipotent (usually exceptional) vertex by the edge corresponding to the
modular Steinberg module $\mathrm{St}_\ell$.
\item
If a vertex
not in $L$ is connected to $L$ by an edge, then it must be connected 
to the Steinberg representation or the non-unipotent vertex.
\item
The (irreducible) representation $V$ corresponding to the part of the
$r$-th cohomology group with compact support of the Coxeter Deligne-Lusztig
variety on which the Frobenius acts by an eigenvalue congruent to $q^r$
modulo $\ell$ is attached to $\mathrm{St}$ by an edge. That edge comes after
the edge connecting $\mathrm{St}$ to $L_{r-1}$ and before the edge connecting
$\mathrm{St}$ to the non-unipotent vertex, in the cyclic ordering of edges
around $\mathrm{St}$.
\end{itemize}
\begin{center}
 \begin{pspicture}(10,4.5)
  \psset{linewidth=1pt}

 \cnode[fillstyle=solid,fillcolor=black](0,2){5pt}{AA}
 \cnode(0,2){8pt}{A}
 \cnode(0.4,2.7){0pt}{A1}
 \cnode(0.4,1.3){0pt}{A2}
 \cnode(-0.6,2.7){0pt}{A3}
 \cnode(-0.6,1.3){0pt}{A4}
 \cnode(-0.9,2){0pt}{A5}
  \cnode[linestyle=none](1.5,2){8pt}{B}
 \cnode(2.1,2.7){0pt}{B1}
 \cnode(2.1,1.3){0pt}{B2}
 \cnode(0.9,2.7){0pt}{B3}
 \cnode(0.9,1.3){0pt}{B4}
 \cnode[linestyle=none](1.5,3.5){8pt}{BV}
 \cnode[linestyle=none](1.5,0.5){8pt}{BV2}
 \cnode(2.1,4.2){0pt}{V1}
 \cnode(0.9,4.2){0pt}{V2}
 \cnode(2.1,-0.2){0pt}{V3}
 \cnode(0.9,-0.2){0pt}{V4}
 \cnode(3,2){5pt}{C}
  \cnode(6.5,2){5pt}{C2}
 \cnode(8,2){5pt}{D}
 \cnode[linestyle=none](9.5,2){5pt}{E}

  \ncline{A}{B} \naput{$\mathrm{St}_\ell$} 
  \ncline{B}{C} \ncput[npos=-0.25]{$\mathrm{St}$}
  \ncline[linestyle=dashed]{C}{C2}
  \ncline{C2}{D}  
  \ncline{D}{E} \naput{$k$}   \ncput[npos=1.1]{$\vphantom{\Big(}1$} 
  \ncline{B}{BV}\ncput[npos=1.3]{$V$}
  \ncline{B}{BV2} \ncput[npos=1.3]{$V^*$}
  \ncline[linestyle=dashed]{BV}{V1}  
  \ncline[linestyle=dashed]{BV}{V2}  
  \ncline[linestyle=dashed]{BV2}{V3}  
  \ncline[linestyle=dashed]{BV2}{V4}  
  \ncline[linestyle=dashed]{A}{A1}  
  \ncline[linestyle=dashed]{A}{A2}  
  \ncline[linestyle=dashed]{A}{A3}  
  \ncline[linestyle=dashed]{A}{A4}
  \ncline[linestyle=dashed]{A}{A5}  
  \ncline[linestyle=dashed]{B}{B1}  
  \ncline[linestyle=dashed]{B}{B2}  
  \ncline[linestyle=dashed]{B}{B3}  
  \ncline[linestyle=dashed]{B}{B4}

\end{pspicture}
\end{center}

\bigskip

We now briefly describe the structure of the article. Section \ref{se:modular}
is devoted to general results on unipotent blocks of
modular representations of finite groups of
Lie type, using algebraic and geometrical methods. In Section \ref{se:unipcyclic},
we deal specifically with unipotent blocks with cyclic defect groups.
After recalling in \S\ref{se:blockscyclic} the basic theory of Brauer trees,
we consider in \S\ref{se:unipblocks} the local structure of the blocks.
In \S\ref{se:properties}, we establish general properties of the trees, and in particular we
relate properties of the complex of cohomology of Coxeter 
Deligne--Lusztig varieties with properties of the Brauer tree.
A key result is Lemma \ref{le:twononzerocoh} about certain
perfect complexes for blocks with cyclic defect groups with only two
non-zero rational cohomology groups.
In \S\ref{sec:determination} we complete the determination of the trees, which are collected in the appendix. The most complicated issues arise from
differentiating the cuspidal modules $E_8[\theta]$ and $E_8[\theta^2]$ when
$d=18$ (\S \ref{se:E8d18}) and ordering cuspidal edges around the Steinberg
vertex for $d=20$ (\S \ref{se:E8d20}).

\medskip

\noindent
{\bf Acknowledgements:} We thank C\'edric Bonnaf\'e, Frank L\"ubeck and Jean Michel for some useful discussions, and Gunter Malle for his comments on
a preliminary version of the manuscript.

\section{Notation}

Let $R$ be a commutative ring. Given two elements $a$ and $r$ of $R$ with $r$ prime, we denote by $a_r$ the largest power of $r$ that divides $a$.
If $M$ is an $R$-module and $R'$ is a commutative
$R$-algebra, we write $R'M=R'\otimes_R M$.

\smallskip
Let $\ell$ be a prime number, $\OC$ the ring of integers of a finite
extension $K$ of $\bbQ_\ell$ and $k$ its residue field. We assume that $K$ is
large enough so that the representations of finite groups considered 
are absolutely irreducible over $K$, and the Frobenius eigenvalues
on the cohomology groups over $K$ considered are in $K$.

Given a ring $A$, we denote by $A\mmod$ the category of finitely
generated $A$-modules, by $A\mproj$ the category of finitely generated
projective $A$-modules and by $\Irr(A)$ the set of isomorphism
classes of simple $A$-modules. When $A$ is a finite-dimensional
algebra over a field, we identify $K_0(A\mmod)$ with
$\bbZ\Irr(A)$ and we denote by $[M]$ the class of an $A$-module $M$.
Given two complexes $C$ and $C'$ of $A$-modules, we denote by 
$\Hom^\bullet_A(C,C')=\bigoplus_{i,j}\Hom_A(C^i,C^{\prime j})$
the total $\Hom$-complex.

\smallskip
Let $\Lambda$ be either $k$, $\OC$ or $K$ and let $A$ be a symmetric $\Lambda$-algebra:
$A$ is finitely generated and free as a $\Lambda$-module and
$A^*$ is isomorphic to $A$ as an $(A,A)$-bimodule.
An $A$-lattice is an $A$-module that is free of
finite rank as a $\Lambda$-module.

Given $M\in A\mmod$, we denote by $P_M$ a projective cover of $M$.
We denote by $\Omega(M)$ the kernel of
a surjective map $P_M\twoheadrightarrow M$ and we define inductively 
$\Omega^i(M)=\Omega(\Omega^{i-1}(M))$ for $i\ge 1$, where 
$\Omega^0(M)$ is a minimal submodule of $M$ such that $M/\Omega^0(M)$ is
projective. Note that $\Omega^i(M)$ is unique up to isomorphism.
When $M$ is an $A$-lattice, we define $\Omega^{-i}(M)$ as $(\Omega^i(M^*))^*$,
using the right $A$-module structure on $M^*=\Hom_\Lambda(M,\Lambda)$.

We denote by $\Ho^b(A)$ and $D^b(A)$ the homotopy and derived categories of bounded complexes of
finitely generated $A$-modules. Given a bounded complex of
finitely generated $A$-modules $C$, there is a complex $C^{\redu}$ of
$A$-modules, unique up to (non-unique) isomorphism, such 
that $C$ is homotopy equivalent to $C^{\redu}$ and $C^{\redu}$
has no non-zero direct summand that is homotopy equivalent to $0$.

Suppose that $\Lambda=k$. We denote by
$A\mstab$ the stable category of $A\mmod$, i.e., the
additive quotient by the full subcategory of finitely generated
projective $A$-modules. Note that the canonical functor $A\mmod\to
D^b(A)$ induces an equivalence from $A\mstab$ to the quotient of
$D^b(A)$ by the thick subcategory of perfect complexes of $A$-modules, making
$A\mstab$ into a triangulated category with translation functor~$\Omega^{-1}$.

\smallskip

Suppose that $\Lambda=\OC$. 
We denote by $d:K_0(KA)\to K_0(kA)$ the decomposition map. It is
characterized by the property $d([KM])=[kM]$ for an $A$-lattice $M$.

\smallskip
Let $G$ be a finite group and $A=KG$. We identify $\Irr(A)$ with the set
of $K$-valued irreducible characters of $G$.
Given $\chi\in\Irr(KG)$, we denote by $b_\chi$ the block idempotent
of $\OC G$ that is not in the kernel of $\chi$.
We put $e_G=\frac{1}{|G|}\sum_{g\in G}g$.

\smallskip
Let $Q$ be an $\ell$-subgroup of $G$. We denote by
$\Br_Q:\OC G\mmod\to kN_G(Q)\mmod$ the Brauer functor: $\Br_Q(M)$
is the image of $M^Q$ in the coinvariants $(kM)_Q := k\otimes_{\OC Q} M$.
We denote by $\br_Q:(\OC G)^Q\to kC_G(Q)$ the algebra morphism
that is the resriction of the linear map defined by $g\mapsto
\delta_{g\in C_G(Q)}g$, where $\delta_{g \in H}$ equals $1$ if
$g \in H$ and $0$ otherwise.

\section{Modular representations and geometry}
\label{se:modular}

\subsection{Deligne--Lusztig varieties}
\subsubsection{Unipotent blocks}
\label{se:unipotentblocks}

Let $\bG$ be a connected reductive algebraic group
defined over an algebraic closure
of a finite field of characteristic $p$, together with an endomorphism $F$,
a power of which is a Frobenius endomorphism. In other words, there exists a positive integer $\delta$ 
such that $F^\delta$ defines a split $\bbF_{q^\delta}$-structure on $\bG$ for a certain power $q^\delta$ of $p$, where $q \in \bbR_{>0}$. We will assume that $\delta$ is minimal for this property. Given an $F$-stable closed
subgroup $\bH$ of $\bG$, we will denote by $H$ the finite group of fixed points $\bH^F$. The group $G$ is a finite group of Lie type. We are interested in the modular representation theory of $G$ in non-defining characteristic, so that we shall always work under the assumption $\ell \neq p$. 

\smallskip
Let $\bT \subset \bB$ be a maximal torus contained in a Borel subgroup of
$\bG$, both of which are assumed to be $F$-stable. Let $W = N_\bG(\bT)/\bT$ be the Weyl group of $\bG$ and $S$ be the set of simple reflections of $W$
associated to $\bB$.
We denote by $r=r_G$ the $F$-semisimple rank of $(\bG,F)$, i.e., the
number of $F$-orbits on $S$.

\smallskip
Given $w \in W$, the \emph{Deligne--Lusztig variety} associated to $w$ is 
\[ \rX_\bG(w) = \rX(w) = \left\{ g\bB \in \bG /\bB \mid g^{-1} F(g) \in \bB w \bB\right\}.\]
It is a smooth
quasi-projective variety endowed with a left action of $G$ by left multiplication. 

Let $\Lambda$ be either $K$ or $k$. Recall that a simple $\Lambda G$-module is {\em unipotent} if it is a composition
factor of $\Hc^i(\rX(w),\Lambda)$ for some $w\in W$ and $i\ge 0$. We denote by
$\Uch(G)\subset\Irr(KG)$ the set of unipotent irreducible $KG$-modules (up
to isomorphism).

A \emph{unipotent block} of $\OC G$
is a block containing at least one unipotent character.

\smallskip
Given $\bP$ a parabolic subgroup of $\bG$ with unipotent radical $\bU$
and an $F$-stable Levi
complement $\bL$, we have a {\em Deligne--Lusztig variety}
$$\rY_\bG(\bL\subset\bP)=\{g\bU\in \bG/\bU \mid g^{-1}F(g)\in \bU\cdot F(\bU)\},$$
a variety with a left action of $G$ and a free right action of $L$ by multiplication.
The {\em Deligne--Lusztig induction} is defined by
\[R_{\bL\subset\bP}^{\bG}:\bbZ\Irr(KL)\to\bbZ\Irr(KG),
\ [M]\mapsto\sum_{i\ge 0}(-1)^i [\Hc^i(\rY_\bG(\bL\subset\bP))\otimes_{K\bL}
M].\]
We also write $R_L^G=R_{\bL\subset\bP}^{\bG}$.
We denote by ${^*R}_{\bL\subset\bP}^{\bG}:\bbZ\Irr(KG)\to\bbZ\Irr(KL)$
the adjoint map.
We have $R_{\bL\subset\bP}^{\bG}(\Uch(L))\subset\bbZ\Uch(G)$
and ${^*R}_{\bL\subset\bP}^{\bG}(\Uch(G))\subset\bbZ\Uch(L)$.

\medskip
Let $w\in W$ and let $h\in\bG$ such that $h^{-1}F(h)\bT=w$.
The maximal torus $\bL=h\bT h^{-1}$ is $F$-stable. It is contained in
the Borel subgroup $\bP=h\bB h^{-1}$ with unipotent radical $\bU$.
In that case, the map $g\bU \mapsto g\bU h = gh (h^{-1}\bU h)$
identifies $\rY(\bL\subset\bP)$ with the variety
$$\rY_\bG(w) = \rY(w) = \left\{ g\bV \in \bG /\bV \mid g^{-1} F(g) \in \bV \dot w \bV\right\}$$
where $\bV = h^{-1} \bU h$ is the unipotent radical of $\bB$ and
$\dot w = h^{-1}F(h) \in N_\bG(\bT)$. Furthermore, there is a morphism of varieties
$$\rY(w)\to \rX(w),\ g\bV\mapsto g\bB$$
corresponding to the quotient by $\bT^{wF} \simeq L$.

\subsubsection{Harish-Chandra induction and restriction}
Given an $F$-stable subset $I$ of $S$, we denote by $W_I$ the subgroup
of $W$ generated by $I$ and by $\bP_I$ and $\bL_I$ the standard parabolic
subgroup and standard Levi subgroup respectively of $\bG$ corresponding to $I$. 
In that case, the maps $R_{L_I}^{G}$ and ${^*R}_{L_I}^{G}$ are induced by the usual
Harish-Chandra induction and restriction functors.
A $\Lambda G$-module $V$ is {\em cuspidal} if ${^*R}_{L_I}^G(V)=0$
for all proper $F$-stable subsets $I$ of $S$.

\smallskip
The following result is due to Lusztig when $\bL$ is a torus
\cite[Corollary 2.19]{LuCBMS}. The same proof applies, using 
Mackey's formula for the Deligne--Lusztig restriction to a torus.

\begin{prop} 
\label{pr:inducedcuspidal}
Let $\bL$ be an $F$-stable Levi subgroup of $\bG$ and $\psi$
an irreducible character of $L$ such that
$(-1)^{r_G+r_L}R_L^G(\psi)$ is an irreducible character of $G$.

If $\psi$ is cuspidal and $\bL$ is not contained in a proper
$F$-stable parabolic subgroup of $\bG$, then 
$(-1)^{r_G+r_L}R_L^G(\psi)$ is cuspidal.
\end{prop}

\begin{proof}
Let $\bT$ be an $F$-stable maximal torus contained in a proper $F$-stable parabolic
subgroup $\bP$ of $\bG$. The Mackey formula (see \cite[7.1]{DeLu83}) provides
a decomposition
$${}^*R_T^G R_L^G(\psi) = \frac{1}{|L|}\sum_{\begin{subarray}{c} x \in G \\
\bT \subset {}^x\bL \end{subarray}} {}^* R_{T}^{{}^xL} ({}^x\psi)$$
where ${}^x \psi := \psi \circ \mathrm{ad} x^{-1}$.
Let $x\in G$ with $\bT \subset {}^x\bL$. By assumption, ${}^x \bL \nsubseteq \bP$,
hence $\bT$ lies in the proper $F$-stable parabolic subgroup ${}^x \bL \cap \bP$ of
${}^x\bL$. Since $\psi$ is cuspidal, $\psi^x$ is a cuspidal character of 
${}^x \bL$, hence ${}^*R_{T}^{{}^xL} ({}^x\psi) = 0$ by 
\cite[Proposition 2.18]{LuCBMS}. It follows that 
${}^*R_T^G((-1)^{r_G+r_L}R_L^G(\psi))=0$, hence
$(-1)^{r_G+r_L}R_L^G(\psi)$ is cuspidal by \cite[Proposition 2.18]{LuCBMS}.
\end{proof}

%
%

\smallskip
Let $A=\OC Gb$ be a block of $\OC G$.
Let $\bP$ be an $F$-stable parabolic subgroup of $\bG$ with unipotent
radical $\bU$ and an $F$-stable Levi complement $\bL$.
Let $A'=\OC L b'$ be a block of $\OC L$.
We say that $A$ is {\em relatively Harish-Chandra $A'$-projective} if
the multiplication map
$b\OC Ge_Ub'\otimes_{\OC L}e_Ub'\OC G b\to \OC Gb$ is a split
surjection as a morphism of $(A,A)$-bimodules. This implies
in particular that any projective $A$-module is a direct summand
of the Harish-Chandra induction of a projective $A'$-module.

\smallskip
The first part of the following lemma follows from
\cite[Proposition 1.11]{DiMinc} (see \cite[Proposition 3.4.(b)]{BDR}
for the general case of a $p'$-solvable group), while the second part is immediate.

\begin{lem}
\label{le:Brauer}
Let $\bP$ be an $F$-stable parabolic subgroup of $\bG$ with unipotent
radical $\bU$ and an $F$-stable Levi complement $\bL$.
Let $Q$ be an $\ell$-subgroup of $L$. 

Then $\bP\cap C_\bG(Q)^\circ$ is a parabolic subgroup of
$C_\bG(Q)^\circ$ with unipotent radical $\bV=\bU\cap C_\bG(Q)$ and
Levi complement $\bL\cap C_\bG(Q)^\circ$.

Given $b$ and $b'$, block idempotents of $\OC G$ 
and $\OC L$ respectively, we have an isomorphism of
$(kC_G(Q),kC_L(Q))$-bimodules
$\Br_{\Delta Q}(b\OC G e_U b')\simeq \br_Q(b)k C_G(Q)e_V \br_Q(b')$.
\end{lem}

Let $D$ be a defect group of $A$ and let $\bH=C_\bG^\circ(D)$.
Assume that $H=C_G(D)$. Let $\lambda$ be a character of $H$ that is trivial
on $Z(D)$ and such that
$\br_D(b)b_\lambda=b_\lambda$.


\smallskip
The following lemma is a variation on \cite[Proposition 4.2]{Li2}.

\begin{lem}
\label{le:relativeproj}
Let $\bP$ be an $F$-stable parabolic subgroup of $\bG$ with unipotent
radical $\bU$ and an $F$-stable Levi complement $\bL$.

Assume that $D\le L$ and let $\lambda'$ be a character
of $C_L(D)$ such that
$\langle {^*R}_{H\cap L}^H(\lambda),\lambda' \rangle\neq 0$ and such that $\lambda'$ is the lift to $C_L(D)$ of a defect zero character of $C_L(D)/Z(D)$.
Let $A' = \OC b'$ be the block of $\OC L$ of defect group $D$ such that $\br_D(b')b_{\lambda'}=b_{\lambda'}$.

Then the block $A$ is relatively Harish-Chandra $A'$-projective.
\end{lem}

\begin{proof}
Let $\bV$ be the unipotent radical of $\bH\cap\bP$ and let
$\bM=\bH\cap\bL$, a Levi complement of $\bV$ in $\bH\cap\bP$.
Note that $D\subset M=C_L(D)$.

Recall that $H = C_G(D)$. 
The condition $\langle {^*R}_{H\cap L}^H(\lambda),\lambda' \rangle\neq 0$
implies that the multiplication map
$$b_\lambda k(H/Z(D)) e_V b_{\lambda'}\otimes_{kM/Z(D)}
e_V b_{\lambda'}k(H/Z(D)) b_\lambda \to k(H/Z(D)) b_\lambda$$
is surjective. It follows from Nakayama's Lemma that the multiplication map
$$b_\lambda kH b_{\lambda'}\otimes_{kM}kH b_{\lambda'}b_\lambda  \to
kH b_\lambda$$
is also surjective.

Since $\br_D(e_U)=e_V$, the commutativity of the diagram 
$$\xymatrix{
kH\otimes_{kM}kH\ar@{^{(}->}[r]^-{\mathrm{can}}
\ar@/_3pc/[ddrrr]_{\mathrm{mult}} &
(kG\otimes_{kL}kG)^{\Delta D} \ar[rr]^-{\mathrm{mult}} 
\ar@{->>}[d]^{\mathrm{can}} &&
 (kG)^{\Delta D} \ar@{->>}[d]_{\mathrm{can}} \\
& \Br_{\Delta D}(kG\otimes_{kL}kG) \ar[rr]_-{\Br_{\Delta D}(\mathrm{mult})}
 && \Br_{\Delta D}(kG) \\
&&& kH \ar[u]^\sim \ar@/_3pc/[uu]_{\mathrm{can}}
}$$
together with Lemma \ref{le:Brauer} shows that the multiplication map induces a surjection
$$\Br_{\Delta D}(bkG e_U b'\otimes_{kL}b'e_U kGb)\twoheadrightarrow
\Br_{\Delta D}(bkG).$$
We deduce from \cite[Lemma A.1]{BDR} that the multiplication map
gives a split surjective morphism of $(\OC Gb,\OC Gb)$-bimodules
$b\OC G e_U b'\otimes_{\OC L}b'e_U \OC Gb\twoheadrightarrow b\OC G$.
\end{proof}

\subsubsection{Complex of cohomology and Frobenius action}
\label{subsubsec:complex}

Following \cite[Theorem 1.14]{DR}, given a variety $X$ defined over
$\bbF_{q^\delta}$ with the action of a finite group $H$, there is a
bounded complex $\tRgc(X,\OC)$ of 
$\OC(H \times \langle F^\delta \rangle)$-modules with
the following properties:
\begin{itemize}
\item $\tRgc(X,\OC)$ is unique up to isomorphism in the quotient of the homotopy category of complexes of $\OC(H \times \langle F^\delta \rangle)$-modules by the thick subcategory of complexes whose restriction to $\OC H$ is homotopic to $0$;
\item the terms of $\mathrm{Res}_{\OC H} \, \tRgc(X,\OC)$ are direct summands of finite direct sums of modules of the form $\OC(H/L)$, where $L$ is the stabilizer in $H$ of a point of $X$;
\item the image of $\tRgc(X,\OC)$ in the derived category of $\OC (H \times \langle F^\delta\rangle)$ is the usual complex $\Rgc(X,\OC)$.
\end{itemize}

\smallskip
Note that in \cite{DR} such a complex was constructed over $k$ instead of
$\OC$, but the same methods lead to a complex over $\OC$. Indeed,
note first that there is a bounded
complex of $\OC (H \times \langle F^\delta\rangle)$-modules $C$ constructed in \cite[\S 2.5.2]{Rou}, whose restriction to
$\OC H$ has terms that are direct summands of possibly infinite direct
sums of modules of the form $\OC(H/L)$, where $L$ is the stabilizer
in $L$ of a point of $X$. Furthermore, that restriction is homotopy
equivalent to a bounded complex whose terms are direct summands of 
finite direct
sums of modules of the form $\OC(H/L)$, where $L$ is the stabilizer
in $H$ of a point of $X$. One can then proceed as in \cite{DR} to
construct $\tRgc(X,\OC)$.

\medskip

Given $\lambda \in k^\times$, we denote by $L(\lambda)$ the inverse
image of $\lambda$ in $\OC$. Given an
$\OC \langle F^\delta \rangle$-module $M$ that is finitely generated as
an $\OC$-module, we denote by
$$M_{(\lambda)}=\{m\in M \mid \exists \lambda_1,\ldots,\lambda_N\in L(\lambda)
\text{ such that }
(F^\delta-\lambda_1)\cdots (F^\delta-\lambda_N)(m)=0\}$$
the `generalized $\lambda$-eigenspace mod $\ell$' of $F^\delta$.

\smallskip
 The image
of $\tRgc(X,k)_{(\lambda)}$ in $D^b(kH)$ will be denoted by
$\Rgc(X,k)_{(\lambda)}$ and we will refer to it as the generalized
$\lambda$-eigenspace of $F^\delta$ on the cohomology complex of $X$.

\medskip
When $\ell\nmid|\bT^{wF}|$, the stabilizers of points of $\rX(w)$
under the action of $G$ are $\ell'$-groups and the terms of the complex
of $\OC G$-modules $\tRgc(\rX(w),\OC)$ are projective.

\begin{lem}
\label{le:shiftFrob}
Let $w\in W$ such that $\bT^{wF}$ has cyclic Sylow $\ell$-subgroups.
Given $\zeta\in k^\times$, we have
$$\Rgc(\rX(w),k)_{(q^{-\delta}\zeta)}\simeq \Rgc(\rX(w),k)_{(\zeta)}[2]\quad
\text{in }kG\mstab.$$
\end{lem}

\begin{proof}
Recall (\S \ref{se:unipotentblocks}) that there is a variety $\rY(w)$ acted on by
$G = \bG^F$ on the left and
acted on freely by $\bT^{wF}$ on the right such that
$\rY(w)/\bT^{wF}\simeq\rX(w)$. Consider the automorphism $\varphi$ of
	$\bT^{wF}$ given by the action of $F^{-\delta}$.
We have a right action of $\bT^{wF}\rtimes \langle\varphi\rangle$ on
$\rY(w)$ where $\varphi$ acts as $F^\delta$.
We have
$\Rgc(\rY(w),k)\otimes^\bbL_{k\bT^{wF}}k\simeq \Rgc(\rX(w),k)$.

Let $t$ be a generator of the Sylow $\ell$-subgroup $D$ of $\bT^{wF}$ and
let $I=(t-1)\cdot kD$.
	We have $\varphi(t-1)=q^{-\delta}(t-1)\pmod{I^2}$, hence
there is an exact sequence 
of $k(D\rtimes\langle \varphi\rangle)$-modules
$$0\to \ker f \to kD\otimes k_{-q^\delta}\xrightarrow{f} kD\to k\to 0$$
where $k_{-q^\delta}$ is the one-dimensional module with trivial
	$D$-action and where $\varphi$ acts by multiplication by $q^{-\delta}$.
The kernel of $f$ is the socle $k((1+t+\cdots +t^{|D|-1})\otimes 1)$
	of $kD\otimes k_{q^{-\delta}}$. Since $\varphi$ acts on that line by
	multiplication by $q^{-\delta}$, the exact sequence above gives 
a $\varphi$-equivariant distinguished triangle 
	$k\to k_{q^{-\delta}}[2]\to C\rightsquigarrow$
in $D^b(k\bT^{wF})$, where $C$ is perfect.

Applying $\Rgc(\rY(w),k)\otimes^\bbL_{k\bT^{wF}}-$, we obtain
a distinguished triangle in $D^b(kG)$, equivariant for the action of
$F^\delta$
$$\Rgc(\rX(w),k)\to\Rgc(\rX(w),k)\otimes k_{q^{-\delta}}[2]\to
C'\rightsquigarrow,$$
where $C'$ is perfect. 
The lemma follows by taking generalized $q^{-\delta}\zeta$-eigenspaces.
\end{proof}

\subsubsection{Simple modules in the cohomology of Deligne--Lusztig varieties}
By definition, every simple unipotent $kG$-module occurs in the cohomology of some Deligne--Lusztig variety $\rX(w)$.
If $w$ is minimal for the Bruhat order, this module only occurs in middle
degree.
This will be an
important property to compute the cohomology of $\rX(w)$ over $\mathcal{O}$
from the cohomology over $K$. 
Let us now recall the precise result \cite[Propositions 8.10 and 8.12]{BR1}.
We adapt the result to the varieties $\rX(w)$.

\smallskip
Recall that there is a pairing $K_0(kG\mproj)\times K_0(kG\mmod)\to\bbZ$
defined by 
$$\langle [P];[M] \rangle = \dim_k \mathrm{Hom}_{kG}(P,M)$$
for $P \in kG\mproj$ and $M \in kG\mmod$.
The Cartan map $K_0(kG\mproj)\to K_0(kG\mmod)$ is injective and we
identify $K_0(kG\mproj)$ with its image. It is a submodule of finite index.
In other words, for any $f \in K_0(kG\mproj)$, there is a positive integer
$N$ such  that $Nf\in K_0(kG\mproj)$. Consequently the pairing above can
be extended to a pairing $K_0(kG\mmod)\times K_0(kG\mmod)\to\bbQ$.

\begin{prop}\label{prop:middledegree}
  Let $M$ be a simple unipotent $kG$-module and let $w \in W$.
The following properties are equivalent: 
  \begin{itemize}
    \item[(a)] $w$ is minimal such that
    $\RHom^\bullet_{kG}(\Rgc(\rX(w),k),M) \neq 0$;
    \item[(b)] $w$ is minimal such that
    $\RHom^\bullet_{kG}(M,\Rgc(\rX(w),k)) \neq 0$; 
    \item[(c)] $w$ is minimal such that
    $\langle [\Rgc(\rX(w),k)],[M]\rangle \neq 0$.
  \end{itemize}
Assume that $w$ is such a minimal element. We have
$\Hom_{kG}(M,\Hc^{\ell(w)}(\rX(w),k))\neq 0$. If
$\ell \nmid |\bT^{wF}|$, then 
$\Hom_{D^b(kG)}(\Rgc(\rX(w),k),M[-i])=
\Hom_{D^b(kG)}(M,\Rgc(\rX(w),k)[i])=0$ for $i\neq \ell(w)$.
\end{prop}

\begin{proof}
We use the variety $\rY(w)$ as in the proof of Lemma \ref{le:shiftFrob}.
Since the stabilizers for the action of $G$ on $\rY(w)$ are $p$-groups
(and hence $\ell'$-groups), the complex $\Rgc(\rY(w),k)$ is perfect
(see \S\ref{subsubsec:complex}) and therefore $[\Rgc(\rY(w),k)] \in  K_0(kG\mproj)$. 

Let  $\bT^{wF}_{\ell'}$ be the subgroup of elements of $\bT^{wF}$ of order
prime to $\ell$ and
$$b_w=\frac{1}{|\bT^{wF}_{\ell'}|}\sum_{t\in \bT^{wF}_{\ell'}}t$$
be the principal block idempotent of $k\bT^{wF}$.

All composition factors of $b_wk\bT^{wF}$ are trivial, hence
$\Rgc(\rY(w),k)b_w$ is an extension of $N=|\bT^{wF}|_\ell$ copies of
$\Rgc(\rX(w),k)$.
As a consequence, $[\Rgc(\rY(w),k)b_w]=N\cdot [\Rgc(\rX(w),k)]$.

We deduce that $\langle [\Rgc(\rX(w),k)],[M]\rangle \neq 0$ if and only if
$\langle [\Rgc(\rY(w),k)b_w],[M]\rangle \neq 0$.
It follows also that an integer $r$ is minimal such that 
$\Hom_{D^b(kG)}(M,\Rgc(\rX(w),k)[r])\neq 0$ if and only if it is
minimal such that $\Hom_{D^b(kG)}(M,\Rgc(\rY(w),k)b_w[r])\neq 0$.
It follows that $w$ is minimal such that
$\RHom^\bullet_{kG}(M,\Rgc(\rY(w),k)b_w)\neq 0$ if and only if
(b) holds. Similarly,
$w$ is minimal such that
$\RHom^\bullet_{kG}(\Rgc(\rY(w),k)b_w,M)\neq 0$ if and only if
(a) holds. 

Note also that the statements above with $\Rgc(\rY(w),k)b_w$ are
equivalent to the same statements with $\Rgc(\rY(w),k)$ since $M$ is
unipotent. The equivalence between (a), (b) and (c) follows now
from \cite[Proposition 8.12]{BR1}.

Suppose that $w$ is minimal with the equivalent properties (a), (b) and (c).
It follows from \cite[Proposition 8.10]{BR1} that
the cohomology of $\RHom^\bullet_{kG}(M,\Rgc(\rY(w),k)b_w)$ is concentrated in
degree $\ell(w)$. The last assertions of the lemma follow.
\end{proof}

Proposition \ref{prop:middledegree} shows that for a minimal $w$, if
$\ell\nmid|\bT^{wF}|$, then the complex of $kG$-modules
$\tRgc(\rX(w),k)^{\redu}$ is isomorphic to a bounded
complex of projective modules such that a
projective cover $P_M$ of $M$ appears only in degree $\ell(w)$ as a direct
summand of a term of this complex.

\subsection{Compactifications}\label{sec:compact}

Let $\underline{S}$ be a set together with a bijection $S\xrightarrow{\sim}
\underline{S},\ s \longmapsto \underline{s}$. Given
$\underline{s} \in \underline{S}$, we put $\bB \underline s \bB = \bB s \bB \cup \bB$. The \emph{generalized Deligne--Lusztig variety} associated to a sequence $(t_1, \ldots ,t_d)$ of elements of $S \cup \underline{S}$ is

\[ \rX(t_1, \ldots,t_d) \, = \, \left\{ (g_0 \bB, \ldots, g_d \bB) \in (\bG/\bB)^{d+1} \, \left| \, \begin{array}{l} g_i^{-1} g_{i+1} \in \bB t_i \bB \, \text{ for } \, i =0, \ldots d-1 \\  g_d^{-1} F(g_0) \in \bB t_d \bB \\\end{array}\right.\right\}.\]
If $w = s_1 \cdots s_d$ is a reduced expression of $w \in W$ then $\rX(s_1, \ldots, s_d)$ is isomorphic to $\rX(w)$ and $\rX(\underline{s}_1, \ldots, \underline{s}_d)$ is a smooth compactification of $\rX(w)$. It will be denoted by $\overline{\rX}(w)$ (even though it depends on the choice of a
reduced expression of $w$). 

\begin{rmk}\label{rmk:xbar}
  Proposition \ref{prop:middledegree} also holds for $\rX(w)$ replaced
by $\overline{\rX}(w)$ (and the assertions for $\rX(w)$ are equivalent
to the ones for $\overline{\rX}(w)$), with the assumption `$\ell\nmid |\bT^{wF}|$' replaced by 
`$\ell\nmid |\bT^{vF}|$ for all $v\le w$' for the last statement.
This follows from the fact that 
  $$\RHom^\bullet_{kG}(\Rgc({\overline\rX}(w),k),M) \simeq
 \RHom^\bullet_{kG}(\Rgc(\rX(w),k),M)$$
  whenever $\RHom^\bullet_{kG}(\Rgc(\rX(v),k),M) = 0$ for all $v<w$. 
\end{rmk}

\smallskip

The cohomology of $\overline{\rX}(w)$ over $K$ is
known \cite{DMR}. We provide here some partial information in the
modular setting. Recall that two elements $w,w'\in W$ are
$F$-conjugate if there exists $v\in W$ such that $w'=v^{-1}wF(v)$.

\begin{prop}\label{xbarxbar}Let $w,w' \in W$. If $\ell\nmid |\bT^{vF}|$
for all $v\leq w$ or for all $v\leq w'$, then 
$\Hc^*(\overline{\rX}(w) \times_{G} \overline{\rX}(w'),\OC)$ is torsion-free.
\end{prop}

\begin{proof}
Given $w,w' \in W$, Lusztig defined in \cite{Lu4} a decomposition of $\overline{\rX}(w) \times \overline{\rX}(w')$ as a disjoint union of
locally closed subvarieties $Z_\mathbf{a}$ stable under the diagonal action of $G$. The quotient by $G$ of each of these varieties has the same cohomology as an affine space. More precisely, given $\mathbf{a}$, there exists:

\begin{itemize}
\item a finite group $\mathcal{T}$, isomorphic to $\bT^{vF}$ for some
$v \leq w$ and to $\bT^{v'F}$ for some $v' \leq w'$ ($F$-conjugate to $v$),
and a quasi-projective variety $Z_0$ acted on by  $G \times \mathcal{T}$,
where $\mathcal{T}$ acts freely, together with a $G$-equivariant isomorphism
$Z_0 / \mathcal{T} \xrightarrow{\sim} Z_\mathbf{a}$;

\item a quasi-projective variety $Z_1$ acted on freely by $G$ and $\mathcal{T}$, such that \linebreak $\Rgc(G\backslash Z_1,\mathcal{O})[2\dim Z_1] \simeq \mathcal{O}$;

\item a $(G\times \mathcal{T})$-equivariant quasi-isomorphism
 $$\Rgc(Z_0,\mathcal{O})[2\dim Z_\mathbf{a}] \xrightarrow{\sim}
\Rgc(Z_1,\mathcal{O}) [2\dim Z_1].$$
\end{itemize}
From these properties and \cite[Lemma 3.2]{BR1} we deduce that if $\mathcal{T}$ is an $\ell'$-group then 
\begin{align*}\Rgc(G\backslash Z_\mathbf{a},\mathcal{O})&\simeq \Rgc(G \backslash Z_0, \mathcal{O}) \, \otimes_{\mathcal{O T}} \, \mathcal{O}
\\&\simeq\Rgc(G \backslash Z_1, \mathcal{O}) \, \otimes_{\mathcal{O T}} \, \mathcal{O} [2 \dim Z_1-2\dim Z_\mathbf{a}]
\\&\simeq\mathcal{O}[-2\dim Z_\mathbf{a}].\end{align*}

As a consequence, the cohomology groups of $\overline{\rX}(w) \times_{G} \overline{\rX}(w')$ are the direct sums of the cohomology groups  of the varieties $G\backslash Z_\mathbf{a}$ and the proposition follows.
\end{proof}

\begin{prop}\label{xbartorsion}
Let $I$ be an $F$-stable
 subset of $S$ such that $\ell \nmid |L_I|$. If $M$ is a simple
$kG$-module such that ${^*R}_{L_I}^G(M)\not=0$, then
$M$ is not a composition factor in the torsion of $\Hc^*(\overline{\rX}(w),\mathcal{O})$ for any $w\in W$.
\end{prop}

\begin{proof} 
Let $V$ be a simple $kL_I$-module such that $\Hom_{kL_I}(V,
{^*R}_{L_I}^G(M))\not=0$.
Let $v \in W_I$ be minimal such that 
$V^*$ occurs as a composition factor, or equivalently as a direct
summand, of $\Hc^*(\rX_{\bL_I}(v),k)$. By Remark \ref{rmk:xbar}, it follows that
$V^*$ occurs only in $\Hc^{\ell(v)}(\overline{\rX}_{\bL_I}(v),k)$. 
Since $\mathcal{O}L_I\mmod$ is a hereditary category, it follows that
there is a projective
$\OC L_I$-module $V'$ such that $V'\otimes_{\OC}k\simeq V^*$ and
$V'$  is a direct summand of $\Rgc(\overline{\rX}_{\bL_I}(v),\mathcal{O})$.

Since a projective cover $P_{M^*}$ of $M^*$ occurs as a direct summand of
$R_{L_I}^G(V^*)$, we deduce that it occurs as a direct summand of
$R_{L_I}^G(\Rgc(\overline{\rX}_{\bL_I}(v),\mathcal{O}))\simeq
\Rgc(\overline{\rX}(v),\mathcal{O})$.
It follows from the K\"unneth formula that the complex 
$P_M^*\otimes_{\mathcal{O}G}\Rgc(\overline{\rX}(w),\mathcal{O})$ is 
a direct summand of the complex $\Rgc(\overline{\rX}(v) \times_G \overline{\rX}(w),\OC)$.
By Proposition \ref{xbarxbar} applied to 
$\overline{\rX}(v) \times_G \overline{\rX}(w)$, we deduce that the cohomology
of $P_M^*\otimes_{\mathcal{O}G}\Rgc(\overline{\rX}(w),\mathcal{O})$ is
torsion-free, and hence $M$ does not appear as a composition factor of the
torsion of $\Hc^*(\overline{\rX}(w),\mathcal{O})$.
\end{proof}

\begin{rmk}\label{rmk:torsion}
Note two particular cases of the previous proposition:
\begin{itemize}
\item if $G$ is an $\ell'$-group (\emph{i.e.} if $\ell \nmid |G|$) then so is every subgroup,
therefore $\Hc^*(\overline{\rX}(w),\OC)$ is torsion-free;
\item if $\ell \nmid |L_I|$ for all $F$-stable $I\subsetneq S$,
then the torsion in $\Hc^*(\overline{\rX}(w),\OC)$ is cuspidal.
\end{itemize}
\end{rmk}

\begin{lem}
\label{le:vanishingonminimal}
Let $\Lambda$ be one of $k$, $\OC$ and $K$.
Let $J$ be a subset of $W$ such that if $w\in J$ and $w'<w$, then $w'\in J$
and such that given $w\in W$ and $s\in S$ with $l(sw)>l(w)$ and $l(wF(s))>l(w)$,
then $sw\in J$ if and only if $wF(s)\in J$.

Let $\CZ$ be a thick subcategory of $D^b(\Lambda G)$ such that
$\Rgc(\rX(v),\Lambda)\in\CZ$ for
all elements $v\in J$ that are of minimal length in their $F$-conjugacy
class.

We have $\Rgc(\rX(v),\Lambda)\in\CZ$ for all $v\in J$ and
$\Rgc(\overline{\rX}(w),\Lambda)\in\CZ$ for all $w\in J$.
\end{lem}

\begin{proof}
Consider $s\in S$ and $v,v'\in W$ with $v=sv'F(s)$ and $v\neq v'$.

Assume that $\ell(v)=\ell(v')$, and furthermore that $\ell(sv)<\ell(v)$. We have
$v=sv''$ where $\ell(v)=\ell(v'')+1$ and $v'=v''F(s)$. 
The $G$-varieties $\rX(v)$ and $\rX(v')$ have the same \'etale site,
hence isomorphic complexes of cohomology \cite[Theorem 1.6]{DeLu83}.
If $\ell(sv)>\ell(v)$, then $\ell(vF(s))<\ell(v)$ \cite[Lemma 7.2]{Hu}
and $v=v''F(s)$ with $\ell(v)=\ell(v'')+1$ and $v'=sv''$. We conclude
as above.

Assume now that $\ell(v)=\ell(v')+2$. It follows
from
\cite[Proposition 3.2.10]{DMR} that there is a distinguished triangle
$$\Rgc(\rX(sv'),\Lambda)[-2]\oplus \Rgc(\rX(sv'),\Lambda)[-1]\to
\Rgc(\rX(v),\Lambda) \to \Rgc(\rX(v'),\Lambda)[-2]\rightsquigarrow.$$
So, if $\Rgc(\rX(v'),\Lambda)\in\CZ$ and $\Rgc(\rX(sv'),\Lambda)\in\CZ$,
then $\Rgc(\rX(v),\Lambda)\in\CZ$.

By \cite{GePf0,GePfKi}, any element $v\in W$ 
can be reduced to an element of minimal length in its $F$-conjugacy class by
applying one of the transformations $v\to v'$ above. Note that if $v\in J$,
then $v'\in J$. The lemma follows from the discussion above.
\end{proof}

\subsection{Steinberg representation}\label{sect:ggraev}

We denote by $\bU$ the unipotent radical of the Borel subgroup $\bB$. Let $\psi$ be a \emph{regular} character of $U$ (see \cite[\S 2.1]{BR2}
), $e_\psi$ be the corresponding central idempotent in $\mathcal{O} U$ and
$\Gamma_\psi = \mathrm{Ind}_U^G (e_\psi \mathcal{O} U)$ be the 
Gelfand-Graev module attached to $\psi$. It is a projective $\OC G$-module.
Since $K\Gamma_\psi$ has only one unipotent constituent (namely
the Steinberg character, which we denote by $\mathrm{St}$), the projection of $\Gamma_\psi$
onto the sum of unipotent blocks is indecomposable and does
not depend on $\psi$. Indeed, it is proved in the proof of 
\cite[Theorem 3.2]{Hi} that any projective module in a unipotent
block has a unipotent constituent in its character (this does
not use the connectedness of the center of $\bG$).
Consequently, $\Gamma_\psi$ has a unique unipotent simple quotient
$\mathrm{St}_\ell$. It is called the {\em modular Steinberg representation}.
It is cuspidal if $\ell\nmid |L_I|$ for all $F$-stable $I\subsetneq S$
\cite[Theorem 4.2]{GeHiMa94}.

\medskip
Statement (i) of the proposition below is a result of \cite{Du1}.

\begin{prop}\label{gelfandgraevxbar}Let $t_1, \ldots, t_d$ be elements of $S \cup \underline{S}$.
\begin{itemize}
\item[$\mathrm{(i)}$]
 If $t_i\in S$ for all $i$, then
 $\Hom_{\mathcal{O}G}^\bullet \big(\Gamma_\psi, \Rgc(\rX(t_1,\ldots,
t_d),\mathcal{O})\big) \, \simeq \,
  \mathcal{O}[-\ell(w)]$
in $D^b(\mathcal{O})$, and hence
$\mathrm{St}_\ell$ does not occur as a composition factor of
$\Hc^i(\rX(t_1,\ldots,t_d),k)$ for $i\neq \ell(w)$.
\item[$\mathrm{(ii)}$]
If $t_i{\not\in}S$ for some $i$, then
 $\Hom^\bullet_{\mathcal{O}G} \big(\Gamma_\psi, \Rgc(\rX(t_1,\ldots,t_d),\mathcal{O})\big)$ is acyclic, and hence
$\mathrm{St}_\ell$ does not occur as a composition factor of
$\Hc^*(\rX(t_1,\ldots,t_d),k)$.
\item[$\mathrm{(iii)}$] $\St_\ell$ does not occur as a composition factor of the torsion
part of $\Hc^*(\rX(t_1,\ldots,t_d),\OC)$.
\end{itemize}
\end{prop}

\begin{proof} (i) follows from \cite{Du1} when
$t_1\cdots t_d$ is reduced, and the general case follows by
changing $\bG$ and $F$ as in \cite[Proposition 2.3.3]{DMR}.

Assume now that $t_i \in \underline{S}$ for all $i$.
 Using the decomposition of $\rX(t_1,\ldots,t_d)$ into
Deligne--Lusztig varieties associated to sequences of elements of $S$, we
deduce from the first part of the proposition that  the cohomology of 
$\Hom_{kG}^\bullet \big(k\Gamma_\psi, \Rgc(\rX(t_1,\ldots,t_d),k)\big)$ is zero outside degrees $0,\ldots,d$. Since  $\rX(t_1,\ldots,t_d)$ is a smooth projective variety and 
$(\Gamma_\psi)^* = \Gamma_{\psi^*}$, we deduce that the cohomology is also zero outside
the degrees $d,\ldots,2d$ and therefore it is concentrated in degree $d$.
As a consequence, the cohomology of
$\Hom_{\mathcal{O}G}^\bullet \big(\Gamma_\psi, \Rgc(\rX(t_1,\ldots,t_d),\mathcal{O})\big)$ is free over $\OC$ and concentrated in degree $d$.
By \cite[Proposition 3.3.15]{DMR}, we have
$\Hom_{KG}^\bullet \big(K\Gamma_\psi, \Rgc(\rX(t_1,\ldots,t_d),K)\big)=0$, and hence
$\Hom_{\mathcal{O}G}^\bullet \big(\mathcal{O}\Gamma_\psi, \Rgc(\rX(t_1,\ldots,t_d),\mathcal{O})\big)=0$.
 
\smallskip

(ii) follows now by induction on the number of $i$ such that $t_i$ is in $S$: if one of the $t_i$ is in $S$, say $t_1$, we use the distinguished triangle 
\[ \Rgc(\rX(t_1,t_2,\ldots,t_d),\mathcal{O}) \longrightarrow \Rgc(\rX(\underline{t_1},t_2,\ldots,t_d),
  \mathcal{O}) \longrightarrow \Rgc(\rX(t_2,\ldots,t_d),\mathcal{O}) \rightsquigarrow\]
and use induction. Note that the assumption that one of the $t_i$ is in $\underline{S}$ ensures that we never reach $\rX(1) = G/B$. 

\smallskip
Note finally that (iii) follows from (i) and (ii).
\end{proof}

\begin{prop}
\label{pr:solomontits}
If $\ell\nmid |L_I|$ for all $F$-stable $I\subsetneq S$,
then $K\otimes_{\OC}\Omega^r\OC\simeq\mathrm{St}$.
\end{prop}

\begin{proof}
Given $i\in\{1,\ldots,r\}$, let $M_i=\bigoplus_I R_{L_I}^G \circ {}^*R_{L_I}^G(\OC)$, where $I$ runs over $F$-stable subsets of $S$ such that
$|I/F|=i$. By the Solomon--Tits Theorem \cite[Theorem 66.33]{CuRei}, there is an exact sequence of $\OC G$-modules
\[0\to M\to M^0\to\cdots\to M^r\to 0,\]
where $KM\simeq\mathrm{St}$.

By assumption, $M^i$ is projective for $i\neq r$, while $M^r=\OC$.
We deduce that $M\simeq\Omega^r\OC$.
\end{proof}

\subsection{Coxeter orbits}
Let $s_1,\ldots,s_r$ be a set of representatives of $F$-orbits of simple reflections.
The product $c= s_1 \cdots s_r$ is a \emph{Coxeter element} of $(W,F)$. Throughout this section and \S\ref{secoxeter},
we will assume that $\ell \nmid |\bT^{cF}|$, and hence
$\tRgc(\rX(c),\OC)$ is a bounded complex of finitely generated projective
$\OC G$-modules.

\smallskip
If $v \in W$ satisfies $\ell(v) < \ell(c)$ then $v$ lies in a proper 
$F$-stable parabolic subgroup, 
forcing $\Hom_{kG}^\bullet(\Rgc(\rX(v),k),M)$ to be zero for every cuspidal module $kG$-module $M$. 
Therefore Proposition \ref{prop:middledegree} has the following corollary for Coxeter elements.

\begin{cor}\label{cor:middledegree}
Let $c$ be a Coxeter element and $M$ be a cuspidal $kG$-module. If $\ell \nmid |\bT^{cF}|$, then the cohomology of $\RHom_{kG}^\bullet(\Rgc(\rX(c),k),M)$ and
of $\RHom^\bullet_{kG}(M, \Rgc(\rX(c),k))$ vanishes outside degree $r$.
\end{cor} 

\begin{lem}\label{lem:onedegree} Assume that $\ell \nmid |\bT^{cF}|$. Let $C$ be a direct summand of
$\tRgc(\rX(c),\mathcal{O})$ in $\Ho^b(\mathcal{O}G\mmod)$ such that
  \begin{itemize}
    \item[$\mathrm{(i)}$] the torsion part of $\H^*(C)$ is cuspidal, and
    \item[$\mathrm{(ii)}$] $\H^i(KC)  =0$ for $i\neq r$.
  \end{itemize}
  Then $\H^r(C)$ is a projective $\OC G$-module and $\H^i(C)=0$ for $i\neq r$.
\end{lem}

\begin{proof} Since $r= \ell(c)$, the complex $C$ can be chosen, up to 
isomorphism in $\Ho^b(\mathcal{O}G\mmod)$,
to be a complex with projective terms in degrees $r,\ldots,2r$ and zero terms
outside those degrees.
Let $i$ be maximal such that $\H^i(C) \neq 0$ (or equivalently such that $\H^i(kC) \neq 0$). There is a non-zero map $kC \longrightarrow \H^i(kC)[-i]$ 
in $D^b(kG)$. From Corollary \ref{cor:middledegree} and the assumption (i) we deduce that $i=r$. It follows that the cohomology
of $C$ is concentrated in degree $r$. Since $C$ is a bounded complex
of projective modules, it follows that $\H^r(C)$ is projective.
\end{proof}

\begin{prop}
\label{pr:torsionCoxeter}
Let $I\subset S$ be an $F$-stable subset and let $c_I$ be a Coxeter element
of $W_I$. 
	\begin{itemize}
		\item[$\mathrm{(i)}$] If $\ell\nmid |L_I|$, then $\H^*_c(\rX(c_I),\OC)$ is torsion-free.
		\item[$\mathrm{(ii)}$] If $\H^*_c(\rX(c_I),\OC)$ is torsion-free, then the torsion of
		$\H^*_c(\rX(c),\OC)$ is killed by~${^*R}_{L_I}^G$.
	\end{itemize}
\end{prop}

\begin{proof}
The first statement follows from \cite[Corollary 3.3]{Du4} using $\H^*_c(\rX(c_I),\OC) = R_{L_I}^G(\H^*_c(\rX_{\bL_I}(c_I),\OC))$. 

The image by ${^*R}_{L_I}^G$ of the torsion of $\H^*_c(\rX(c),\OC)$ is the torsion
of $\H^*_c(U_I\setminus \rX(c),\OC)$.
By \cite[Corollary 2.10]{Lu}, the variety $U_I\setminus \rX(c)$ is isomorphic
to $(\bG_m)^{r-r_I}\times \rX_{\bL_I}(c_I)$. The second statement follows.
\end{proof}

\subsection{Generic theory}
\label{se:generictheory}
We recall here constructions of \cite{BrMa,BMM,Spets}, the representation
theory part being based on Lusztig's theory.

\subsubsection{Reflection data}
Let $K=\bbQ(q)$ and
$V=K\otimes_\bbZ Y$, where $Y$ is the cocharacter group of $\bT$. We
denote by $\varphi$ the finite order automorphism of $V$ induced by the action
of $q^{-1}F$.

We denote by $|(W,\varphi)|=x^N\prod_{i=1}^{\dim V}(x^{d_j}-\zeta_j)$ the
{\em polynomial order} of $(W,\varphi)$. Here, $N$ is the number of
reflections of $W$ and we have fixed a
decomposition into a direct sum of $(\bG_m\times\langle\varphi\rangle)$-stable
lines
$L_1\oplus\cdots\oplus L_{\dim V}$ of the tangent
space at $0$ of $V/W$, so that
$d_j$ is the weight of the action of $\bG_m$ on $L_j$ and $\zeta_j$ is
the eigenvalue of $\varphi$ on~$L_j$.

\smallskip
Recall that there is some combinatorial data
associated with $W$ (viewed as a reflection group on $V$) and $\varphi$:
\begin{itemize}
\item a finite set $\Uch(W,\varphi)$;
\item a map $\Deg:\Uch(W,\varphi)\to\bbQ[x]$.
\end{itemize}
We endow $\bbZ\Uch(W,\varphi)$ with a symmetric bilinear form
making $\Uch(W,\varphi)$ an orthonormal basis.

In addition, given $W'$ a parabolic subgroup of $W$ and $w\in W$ such that
$\ad(w)\varphi(W')=W'$, there is a linear map
 $R_{W',\ad(w)\varphi}^{W,\varphi}:
\bbZ\Uch(W',\ad(w)\varphi)\to\bbZ\Uch(W,\varphi)$.

We will denote by ${^*R}_{W',\ad(w)\varphi}^{W,\varphi}$ the adjoint map
to $R_{W',\ad(w)\varphi}^{W,\varphi}$.

\medskip
The data associated with $W$ and $\varphi$ depends only on the
class of $\varphi$ in $\GL(V)/W$. The corresponding pair
$\bbG=(W,W\varphi)$ is called a {\em reflection datum}.

\smallskip
A pair $\bbL=(W',\ad(w)\varphi)$ as above is called a {\em Levi
subdatum} of $(W,\varphi)$. We put $W_\bbL=W'$.

\medskip
There is a bijection 
$$\Uch(\bbG)\xrightarrow{\sim}\Uch(G),\ \bbchi\mapsto \bbchi_q$$
such that
$\Deg(\bbchi)(q)=\bbchi_q(1)$. 

\smallskip
There is a bijection from the set of $W$-conjugacy classes of Levi subdata of
$\bbG$ to the set of $G$-conjugacy classes of $F$-stable Levi subgroups
of $\bG$.

Those bijections have the property that given $\bL$ an $F$-stable
Levi subgroup of $\bG$ with associated Levi subdatum $\bbL=(W',\ad(w)\varphi)$,
we have
$(R_\bbL^\bbG(\bbchi))_q=R_L^G(\bbchi_q)$ for all 
$\bbchi\in\Uch(W,\varphi)$ (assuming $q>2$ if $(\bG,F)$ has a component
of type ${^2E}_6$, $E_7$ or $E_8$, in order for the Mackey formula to be known
to hold \cite{BoMi}).

\subsubsection{$d$-Harish-Chandra theory}
Let $\Phi$ be a cyclotomic polynomial over $K$, i.e., a prime
divisor of $X^n-1$ in $K[X]$ for some $n\ge 1$. Let $V'$ be a subspace
of $V$ and let $w\in W$ such that $w\varphi$ stabilizes $V'$ and
the characteristic polynomial of
$w\varphi$ acting on $V'$ is a power of $\Phi$. Let $W'=C_W(V')$.
Then $(W',\ad(w)\varphi)$ is called a {\em $\Phi$-split
Levi subdatum} of $(W,\varphi)$.

An element $\bbchi\in\Uch(W,\varphi)$ is {\em $\Phi$-cuspidal} if
${^*R}_\bbL^\bbG(\bbchi)=0$ for all proper
$\Phi$-split Levi subdata $\bbL$ of $\bbG$ (when $\bG$ is semisimple,
this is equivalent to the requirement that $\Deg(\bbchi)_\Phi=|\bbG|_\Phi$).

\smallskip

A pair $(\bbL,\bblambda)$ is a \emph{$\Phi$-cuspidal pair} of $\bbG$ if 
$\bbL = (W',\ad(w)\varphi)$ is a $\Phi$-split Levi subdata of $\bbG$ and
$\bblambda \in \Uch(W',\ad(w)\varphi)$ is $\Phi$-cuspidal.
Given such a pair $(\bbL,\bblambda)$, we denote by
$\Uch(\bbG,(\bbL,\bblambda))$ the set of $\bbchi\in\Uch(\bbG)$ such that
$\langle R_{\bbL}^{\bbG}(\bblambda),\bbchi\rangle\neq 0$.
We denote by $W_\bbG(\bbL,\bblambda)=N_W(W_\bbL)/W_\bbL$
the relative Weyl group.

\smallskip

The $\Phi$-Harish-Chandra theory states that:
\begin{itemize}
\item $\Uch(\bbG)$ is the disjoint union of the sets $\Uch(\bbG,(\bbL,\bblambda))$,
where $(\bbL,\bblambda)$ runs over $W$-conjugacy classes of
$\Phi$-cuspidal pairs;
\item there is an isometry 
$$I_{(\bbL,\bblambda)}^\bbG:\bbZ\Irr(W_\bbG(\bbL,\bblambda))\xrightarrow{\sim}
\bbZ\Uch(\bbG,(\bbL,\bblambda));$$
\item those isometries have the property that
$R_\bbM^\bbG I_{(\bbL,\bblambda)}^\bbM=I_{(\bbL,\bblambda)}^\bbG
\Ind_{W_\bbM(\bbL,\bblambda)}^{W_\bbG(\bbL,\bblambda)}$ for all
$\Phi$-split Levi subdata $\bbM$ of $\bbG$ containing $\bbL$.
\end{itemize}

\noindent The sets $\Uch(\bbG,(\bbL,\bblambda))$ are called the 
{\em $\Phi$-blocks} of $\bbG$. 
The {\em defect} of the $\Phi$-block $ \Uch(\bbG,(\bbL,\bblambda))$
is the integer $i\ge 0$ such that the common value of $\Deg(\bbchi)_\Phi$
for $\bbchi\in \Uch(\bbG,(\bbL,\bblambda))$ is $\Phi^i$
(see \S\ref{se:CaEn} for the relation with unipotent $\ell$-blocks).

\section{Unipotent blocks with cyclic defect groups}
\label{se:unipcyclic}

\subsection{Blocks with cyclic defect groups}
\label{se:blockscyclic}
We recall some basic facts on blocks with cyclic defect groups (cf.\ \cite{Feitbook} and \cite{Fe2} for the folding).

\subsubsection{Brauer trees and folding}

\begin{definition}
A {\em Brauer tree} is a planar tree $T$ with at least one edge
together with a positive integer $m$ (the `multiplicity') and,
if $m\ge 2$, the data of a vertex $v_x$, the `exceptional vertex'.
\end{definition}

Note that the data of an isomorphism class of planar trees is the same as
the data of a tree together with a cyclic ordering of the vertices
containing a given vertex.

\smallskip
Let $d>1$ be a divisor of $m$. We define a new Brauer tree $\wedge^d T$.
It has a vertex $\tilde{v}_x$, and the oriented graph $(\wedge^d T)\setminus\{
\tilde{v}_x\}$ is the disjoint union $(T\setminus\{v_x\})\times
\bbZ/d$ of $d$ copies of $T\setminus\{v_x\}$.
Let $l_1,\ldots,l_r$ be the edges of $T$ containing $v_x$, in the cyclic
ordering. The edges of the tree $\wedge^d T$ containing $\tilde{v}_x$ are,
in the cyclic ordering, $(l_1,0),\ldots,(l_1,d-1),(l_2,0),\ldots,(l_2,d-1),
\ldots,(l_r,0),\ldots,(l_r,d-1)$. Finally, for every $i\in\bbZ/d$,
we have an embedding of oriented trees of $T$ in $\wedge^d T$ given on edges by
$l\mapsto (l,i)$, on non-exceptional vertices by $v\mapsto (v,i)$ and
finally $v_x\mapsto \tilde{v}_x$. The multiplicity of $\wedge^d T$ is
$m/d$. When $m{\not=}d$, the exceptional vertex of $\wedge^d T$ is 
$\tilde{v}_x$.

There is an automorphism $\sigma$ of $\wedge^d T$ given by
$\sigma(\tilde{v}_x)=\tilde{v}_x$ and $\sigma(v,i)=(v,i+1)$ for
$v\in T\setminus\{v_x\}$. Let $X$ be the group of automorphisms of
$\wedge^d T$ generated by $\sigma$. There is an isomorphism of
planar trees
$$\kappa:(\wedge^d T)/X\xrightarrow{\sim}T,\ \tilde{v}_x\mapsto v_x,\
X\cdot(v,i)\mapsto v \text{ for }v\in T\setminus\{v_x\}.$$

\smallskip
In particular the Brauer tree $T'=\wedge^d T$ together with the
automorphism group $X$ determine $T$.

\begin{rmk}
\label{re:planarlift}
Given another planar embedding $T'$ of $\Lambda^dT$ compatible
with the automorphism $\sigma$ above and such that $\kappa$ induces an
isomorphism of planar trees
$T'/X\xrightarrow{\sim}T$, then there is an isomorphism of planar trees
$T'\xrightarrow{\sim}\Lambda^d T$ compatible with $\sigma$.
\end{rmk}

\subsubsection{Brauer tree of a block with cyclic defect}
\label{se:Brauertreesblocks}
Let $H$ be a finite group and $A=b\OC H$ be a block of $\OC H$. Let $D$ be a defect group of $A$ and let
$b_D$ be a block idempotent of the Brauer correspondent of $b$ in $\OC N_H(D)$. We
assume $D$ is cyclic and non-trivial.
Let $E=N_H(D,b_D)/C_H(D)$, a cyclic subgroup of $\Aut(D)$
of order $e$ dividing $\ell-1$.

When $e=1$, the block $A$ is Morita equivalent to $\OC D$. We will be
discussing Brauer trees only when $e>1$, an assumption we make for the
remainder of \S\ref{se:blockscyclic}.

\smallskip
We define a Brauer tree $T$ associated to $A$. We put $m=(|D|-1)/e$.
An irreducible character $\chi$ of $KA$ is called {\em non-exceptional}
if $d(\chi){\not=}d(\chi')$ for all
$\chi'\in\Irr(KA)$ for $\chi'{\not=}\chi$ (here $d$ is the decomposition map).
When $m>1$, we denote by $\chi_x$ the sum of the exceptional irreducible
characters
of $KA$ (those that are not non-exceptional). We define the set of vertices of $T$ as the union of the non-exceptional
characters together, when $m>1$, with an exceptional vertex corresponding to $\chi_x$. The set of edges is
defined to be $\Irr(kA)$. An edge $\phi$ has vertices
$\chi$ and $\chi'$ if $\chi+\chi'$ is the character of the projective cover of
the simple $A$-module with Brauer character $\phi$.
Note that the tree $T$ has $e$ edges.

The cyclic ordering of the edges containing a given vertex is defined
as follows: the edge $\phi_2$ comes immediately after the edge $\phi_1$ if
$\Ext^1_A(L_1,L_2){\not=}0$, where $L_i$ is the simple $A$-module
with Brauer character $\phi_i$.


\smallskip
Recall that
the full subgraph of $T$ with vertices the real-valued non-exceptional
irreducible characters and the exceptional vertex if $m>1$ is
a line (the `real stem' of the tree).
There is an embedding of the tree $T$ in $\bbC$ where 
the intersection of $T$ with the real line is the real stem and taking duals
of irreducible characters corresponds to reflection with respect to the
real line.


\subsubsection{Folding}\label{se:folding}
Let $H'$ be a finite group containing $H$ as a normal subgroup and let $b'$ be a block idempotent of $\OC H'$ such that
$bb'\not=0$. We put $A'=b'\OC H'$ and we denote by $T'$ the Brauer tree
of $A'$, with multiplicity $m'$.
We assume $D$ is a defect group of $b'$. Let $b'_{D}$ be the
block idempotent of $\OC N_{H'}(D)$ that
is the Brauer correspondent of $b'$ and let $E'=N_{H'}(D,b'_D)/C_{H'}(D)$,
an $\ell'$-subgroup of $\Aut(D)$. Note that $E$ is a subgroup of $E'$.
Let $H'_b$ be the stabilizer of $b$ in $H'$.
We have $H'_b=HN_{H'}(D,b'_D)$ and there is a Morita equivalence between
$b\OC H'_b$ and $b'\OC H'$ induced by the bimodule $b\OC H'b'$.

Suppose that $E'\neq E$, i.e., $m'\neq m$, since $[E':E]=m/m'$.
 The group $X$ of $1$-dimensional characters of $E'/E\simeq H'_b/H$ acts on
$\Irr(KA')$ and on $\Irr(kA')$ and this induces an action on $T'$, the Brauer
tree associated to $A'$.

\medskip
The result below is a consequence of \cite[proof of Lemma 4.3]{Fe2}
(the planar embedding part follows from Remark \ref{re:planarlift}).

\begin{prop}
There is an isomorphism of Brauer trees $\wedge^d T\xrightarrow{\sim}T'$
such that $(\chi,i)$ maps to a lift of $\chi$, for $\chi$ a
non-exceptional vertex.
\end{prop}

The previous proposition shows that the data of $T'$ and
of the action of $X$ on $T$ determine the tree $T$ (up to parametrization).

%
%
%

\subsection{Structure of unipotent blocks with cyclic defect groups}
\label{se:unipblocks}
We assume in \S\ref{se:unipblocks} that the simple factors of $[\bG,\bG]$ are
$F$-stable. Note that every finite reductive group can be realized as
$\bG^F$ for such a $\bG$.

From now on, we assume that $\ell$ is an odd prime.

\subsubsection{Centre}
We show here that a Brauer tree of a unipotent block of a
finite reductive group (in non-describing characteristic) is 
isomorphic to one coming from a
simple simply connected algebraic group.

\begin{lem}
\label{le:centre}
Assume that $\bG$ is simple and simply connected.
Let $A$ be a unipotent block of $k\bG^F$ whose image
in $k(\bG^F/(Z(\bG)^F)_\ell)$ has cyclic defect. Then, $A$ has
cyclic defect and $Z(\bG)^F_\ell=1$.
\end{lem}

\begin{proof}
Since $\ell$ is odd, it divides $|Z(\bG)^F|$ only in the following
cases \cite[Corollary 24.13]{MaTe}:
\begin{itemize}
\item $(\bG,F)=\SL_n(q)$, $n\ge 2$ and $\ell\mid(n,q-1)$;
\item $(\bG,F)=\SU_n(q)$, $n\ge 3$ and $\ell\mid(n,q+1)$;
\item $(\bG,F)=E_6(q)$ and $\ell\mid(3,q-1)$;
\item $(\bG,F)={^2E}_6(q)$ and $\ell\mid(3,q+1)$.
\end{itemize}

Let $H=\bG^F/(Z(\bG)^F)_\ell$. Suppose that the image of $A$ in $kH$ has non-trivial defect groups.

Assume that $(\bG,F)=\SL_n(q)$, $n\ge 2$ and $\ell \mid (n,q-1)$
or $(\bG,F)=\SU_n(q)$, $n\ge 3$ and $\ell \mid (n,q+1)$. In those cases,
the only unipotent block $A$ is the principal block \cite[Theorem 13]{CaEn93},
so $H$ has cyclic Sylow $\ell$-subgroups: this is impossible.

Assume that $(\bG,F)=E_6(q)$ and $\ell\mid(3,q-1)$. Note that $A$ cannot be the principal
block, as $H$ does not have cyclic Sylow $3$-subgroups. There is
a unique non-principal unipotent block $b$, and its unipotent characters
are the ones in the 
Harish-Chandra series with Levi subgroup $\bL$ of type $D_4$
\cite[``Donn\'ees cuspidales 7,8,9'', p.352--353]{En00}. Those three
unipotent characters are trivial on $Z(\bG)^F$. It is easily seen that there
is no equality between their degrees nor is the sum of two degrees equal
the third one. As a consequence, they cannot belong to a
block of $kH$ with cyclic defect and inertial index at most $2$.

The same method (replacing $q$ by $-q$) shows also that $b$ cannot have
cyclic defect when $(\bG,F)={^2E}_6(q)$ and $\ell\mid(3,q+1)$.
\end{proof}

Let $H$ be a finite simple group of Lie type. Then there is 
a simple simply connected reductive algebraic group $\bG$ endowed
with an isogeny $F$ such that $H\simeq \bG^F/Z(\bG)^F$, unless
$H$ is the Tits group, $(\bG,F)={^2}F_4(2)$
and we have $H=[\bG^F/Z(\bG)^F,\bG^F/Z(\bG)^F]$,
a subgroup of index $2$ of $\bG^F/Z(\bG)^F$.

The previous lemma shows that if the image in $kH$ of a unipotent $\ell$-block of $k\bG^F$ has cyclic defect groups, then the block of $k\bG^F$ already has cyclic defect groups. By folding (\S\ref{se:folding}), the Brauer tree of a unipotent block of $\OC \bG^F$ determines the Brauer tree of the corresponding block of $\OC H$.

\begin{prop}
\label{pr:centre}
Let $A$ be a unipotent block of $\OC\bG^F$ with cyclic defect group $D$.
We have $C_\bG^\circ(x)=C_\bG^\circ(D)$ and
$C_G(x)=C_\bG^\circ(x)^F$ for all non trivial elements $x\in D$.
Furthermore, one of the two following statements hold:
\begin{itemize}
\item $D$ is the Sylow $\ell$-subgroup of $Z^\circ(\bG)^F$ and
there is a finite subgroup $H$ of $G$ containing $[\bG,\bG]^F$
such that $G=D\times H$;
\item $|Z(\bG)^F|_\ell=1$, $D\neq 1$ and
$A$ is Morita equivalent to a unipotent block 
of a simple factor of $\bG/Z(\bG)$
with cyclic defect groups isomorphic to $D$.
\end{itemize}
In particular, $Z(\bG)^F/Z^\circ(\bG)^F$ is an $\ell'$-group.
\end{prop}

\begin{proof}
Let $\bH$ be a simple factor of $[\bG,\bG]$.
Consider a simply connected cover 
$\bH_{\mathrm{sc}}$ of $\bH$. The restriction of unipotent characters in $A$
to $H$ and then to $\bH_{\mathrm{sc}}^F$
are sums of unipotent characters, and the blocks that contain them
have a defect group that is cyclic modulo $Z(\bH_{\mathrm{sc}})^F$.
It follows from Lemma \ref{le:centre} that $\ell\nmid Z(\bH_{\mathrm{sc}})^F$, and therefore $\ell\nmid Z(\bG_{\mathrm{sc}})^F$, where
$\bG_{\mathrm{sc}}$ is a simply connected cover of $[\bG,\bG]$.
Note that as a consequence, both $(Z(\bG)/Z^\circ(\bG))^F$ and
$(Z(\bG^*)/Z^\circ(\bG^*))^F$ are $\ell'$-groups, where $\bG^*$ is
a Langlands dual of $\bG$.

Let $\bG_{\mathrm{ad}}=\bG/Z(\bG)$.
By \cite[Theorem 17.7]{CaEn}, we have 
$A\simeq \OC Z(\bG)^F_\ell\otimes A'$, where
$A'$ is the unipotent block of $G_{\mathrm{ad}}$ containing the
unipotent characters of $A$. Also, $D\simeq Z(\bG)^F_\ell\times D'$,
where $D'$ is a defect group of $A'$. So, if $\ell$ divides
$|Z(\bG)^F|$, then $\ell$ divides $|Z^\circ(\bG)^F|$, $D'=1$ and
$D$ is the Sylow $\ell$-subgroup of $Z(\bG)^F$. Otherwise,
consider a decomposition
$\bG_{\mathrm{ad}}=\bG_1\times\cdots\times\bG_r$ where the $\bG_i$ are simple
and $F$-stable factors. There is a corresponding decomposition
$A'=A_1\otimes\cdots\otimes A_r$ where $A_i$ is a unipotent
block of $G_i$. So, there is a unique $i$ such that $A_i$ does not have
trivial defect groups, and $A$ is Morita equivalent to $A_i$.

\smallskip
Let us now prove the first statement of the proposition. 
We have $C_\bG^\circ(x)^F=C_G(x)$  by
\cite[Proposition 13.16]{CaEn}.
The block idempotent $\mathrm{br}_x(b)$ gives a (nilpotent) block of
$\OC C_G(x)$
with defect group $D$.
By \cite[Theorem 3.2]{BrMi}, this is a unipotent block.
We deduce from the other part of the proposition
that $D\subset Z(C_\bG^\circ(x))^F$, and hence $C_\bG^\circ(x)=C_\bG^\circ(D)$.
\end{proof}


\subsubsection{Local subgroups and characters}
\label{se:localsub}
Let $A$ be a unipotent block of $\OC G$ with a non-trivial
cyclic defect group $D$.
Let $(D,b_D)$ be a maximal $b$-subpair as in \S\ref{se:Brauertreesblocks} and
let $E=N_G(D,b_D)/C_G(D)$. Recall that we assume that $\ell$ is odd.

Let $Q$ be the subgroup of order $\ell$ of $D$ and
let $\bL=C_\bG^\circ(Q)$.

\begin{theorem}
\label{th:structureblocks}
\begin{itemize}
\item $\bL=C_\bG^\circ(D)$ is a Levi subgroup of $\bG$.
\item $D$ is the Sylow $\ell$-subgroup of $Z^\circ(\bL)^F$ and $L=D\times H$ for some subgroup $H$ of $L$ containing $[\bL,\bL]^F$.
\item There is a (unique) unipotent character $\lambda$ of $L$ such that $R_L^G(\lambda)=\sum_{\chi\in\Uch(KA)}\varepsilon_\chi \chi$ for some
$\varepsilon_\chi\in\{\pm 1\}$.
\item We have $|\Uch(KA)|=|E|$ and $\Irr(KA)$ is the disjoint union of $\Uch(KA)$ and of $\{(-1)^{r_G+r_L}R_L^G(\lambda\otimes\xi)\}_{\xi\in(\Irr(KD)\setminus\{1\})/E}$.
\item If $|E|\neq |D|-1$, then $\Uch(KA)$ is the set of
non-exceptional characters of $A$.
\end{itemize}
\end{theorem}

\begin{proof}
Let $A'$ be the block of $\OC L$ corresponding to $A$. This is
a unipotent block with defect group $D$.
By Proposition \ref{pr:centre}, we have $Q\le Z^\circ(\bL)\not=1$, hence
$\bL$ is a Levi subgroup of $\bG$, since it is the centralizer
of the torus $Z^\circ(\bL)$. Also, 
$D$ is the Sylow $\ell$-subgroup of $Z^\circ(\bL)^F$ and
$L=D\times H$ for some subgroup $H$ of $L$ containing $[\bL,\bL]^F$.

There is a (unique) unipotent irreducible representation $\lambda$ in 
$\Irr(KA')$
and $\Irr(KA')=\{\lambda\otimes\xi\}_{\xi\in\Irr(KD)}$.

Let $\xi\in\Irr(KD)\setminus\{1\}$. The character 
$\chi_\xi=(-1)^{r_G+r_L}R_L^G(\lambda\otimes\xi)$ is 
irreducible and it depends only on $\Ind_D^{D\rtimes E}\xi$ \cite[Theorem 13.25]{DiMi}.
Furthermore, $\chi_\xi=\chi_{\xi'}$ implies $\xi'\in E\cdot\xi$.

Assume that $|E|\neq |D|-1$. There are at least two
$E$-orbits on the set of non-trivial characters of $D$, so the $\chi_\xi$
for $\xi\in(\Irr(KD)\setminus\{1\})/E$ are exceptional
characters. Since $A$ and $A'$ have the same number of exceptional
characters, we have found all exceptional characters of $A$.

Let $\chi_1=(-1)^{r_G+r_L}R_L^G(\lambda)$. We have $d(\chi_1)=d(\chi_\xi)$
for any $\xi\in\Irr(KD)\setminus\{1\}$. There are integers $n_\chi\in\bbZ$
such that $\chi_1=\sum_{\psi\in\Uch(KA)}n_\psi\psi$.
The restriction of the decomposition map to $\bbZ\Uch(KA)$ is injective,
since we have removed exceptional characters (if $|E|\neq |D|-1$, otherwise
one character) from $\Irr(KA)$.
It follows that $\chi_1$ is the unique linear combination of 
unipotent characters of $A$ such that $d(\chi_1)=d(\xi)$ for some
$\xi\in\Irr(D)\setminus\{1\}$. On the other hand,
this unique solution satisfies $n_\psi=\pm 1$ and
the number of unipotent characters in $A'$ is $|E|$.
%
%
%
%
%
\end{proof}

\begin{rmk}
Choose a bijection $\Irr(KE)\xrightarrow{\sim}\Uch(KA'),\ \phi\mapsto\chi_\phi$.
Define 
$I:\bbZ\Irr(KD\rtimes E)\xrightarrow{\sim}\bbZ\Irr(KA')$ by
$I(\Ind_D^{D\rtimes E}\xi)=R_L^G(\xi)$ if $\xi\in\Irr(KD)\setminus\{1\}$ and
$I(\phi)=\varepsilon_{\chi_\phi}\chi_\phi$ for $\phi\in\Irr(KE)$.
The proof of Theorem \ref{th:structureblocks} above shows that $I$ is an
isotypy, with local isometries $I_x:\bbZ\Irr(KD)\xrightarrow{\sim}
\bbZ\Irr(KA'),\ \xi\mapsto\lambda\otimes\xi$ for $x\in D\setminus\{1\}$.
\end{rmk}

\subsubsection{Genericity}
\label{se:CaEn}
We assume in \S\ref{se:CaEn} that $F$ is a Frobenius endomorphism.
Let $A$ be a unipotent block of $\OC G$ with a non-trivial
cyclic defect group $D$ and let $\bL=C_\bG^\circ(D)$.

Let $d$ be the order of $q$ modulo $\ell$. Note that $\ell$ divides $\Phi_{e}(q)$ if and only if $e = d\ell^j$ for some $j \geq 0$.

Brou\'e--Michel \cite{BMi1} and Cabanes--Enguehard \cite{CaEn94} showed that under a mild additional assumption
on $\ell$ (for quasisimple groups not of type $A$, $\ell$ good is enough), unipotent characters in $\ell$-blocks with abelian defect groups are
$\Phi_d$-blocks. 
We show below that this results holds for 
$\ell$-blocks with cyclic defect groups without assumptions on $\ell$.
Using the knowledge of generic degrees,
the unipotent $\Phi_d$-blocks with defect $1$ for simple $\bG$ can
be easily determined, using for example \Chevie{} \cite{Mi15}.

\begin{theorem}
With the notations of \S\ref{se:localsub}, we have the following
assertions:
\begin{itemize}
\item $\bL$ is a $\Phi_d$-split Levi subgroup of $\bG$;
\item $D$ has order $|\Phi_d(q)|_\ell$;
\item $\lambda=\bblambda_q$ for a unipotent $\Phi_d$-cuspidal character $\bblambda$ of $\bbL$ and there is a bijection
$\Uch(\bbG,(\bbL,\bblambda))\xrightarrow{\sim}\Uch(KA)$ given by $\bbchi\mapsto \bbchi_q$;
\item the $\Phi_d$-block $\Uch(\bbG,(\bbL,\bblambda))$ has defect $1$;
\item if $\ell$ is a bad prime for $\bG$ or $\ell=3$ and $(\bG,F)$ has type ${^3D}_4$, then we are in one of the cases listed in Table \ref{ta:bad}.
\end{itemize}
\end{theorem}

\begin{proof}
By Proposition \ref{pr:centre}, we can assume that $\bG$ is simple and simply
connected. When $\ell$ is good and different from $\ell=3$ for type
${^3D}_4$, the theorem is \cite{CaEn94}.

Otherwise, the result follows from
\cite[Th\'eor\`eme A]{En00},
by going through the list of $d$-cuspidal pairs with
$\ell$-central defect and checking if the defect groups given in
\cite[\S3.2]{En00} are cyclic.
We list the unipotent blocks with cyclic defect for $\ell$ bad in Table \ref{ta:bad}, following
\cite[\S 3.2]{En00}. Note that in \cite[p.358, No 29]{En00}, 
`$E_7[\pm\xi]$' should be replaced by `$\phi_{512,11},\phi_{512,12}$', as in \cite[Table 1, Cases 42, 43]{BMM}.
\end{proof}

\begin{table}
$$\begin{array}{|c|c|c|c|}
\cline{1-4} 
(\bG,F) & \ell & d & (\bL,\lambda) \\
\cline{1-4} 
E_6(q)         & 3   & 2 & (A_5(q)\cdot(q-1),\phi_{321}) \\
{^2E}_6(q)     & 3   & 1 & ({^2A}_5(q)\cdot (q+1),\phi_{321}) \\
E_8(q)         & 3,5 & 1 & (E_7(q)\cdot(q-1),E_7[\pm\II]) \\
E_8(q)         & 3,5 & 2 & (E_7(q)\cdot(q+1),\phi_{512,11}\text{ or }
\phi_{512,12}) \\
\cline{1-4} 
\end{array}$$
\caption{Unipotent blocks with cyclic defect for $\ell$ bad}
\label{ta:bad}
\end{table}

Brou\'e \cite{Brperfect} conjectured that there is a parabolic subgroup $\bP$
with an $F$-stable Levi complement $\bL$ such that
$b\Rgc(\rY_\bG(\bL\subset\bP),\OC)$
induces a derived equivalence between $A$ and the corresponding
block of $\OC N_G(D,b_D)$.
In \cite{ChRou}, it is conjectured that such an equivalence should
be perverse. It is further shown there how the Brauer tree of $A$ could then
be combinatorially constructed from the perversity function.
The perversity function can be encoded in the data of 
a function $\pi:\Uch(KA)\to\bbZ$ that describes the (conjecturally)
unique $i$ such that $V\in\Uch(KA)$ occurs in
$\Hc^i(\rY_\bG(\bL\subset\bP),K)$.

In \cite{Cra3}, the first author gave a conjectural description $\gamma$
of the function $\pi$, depending  on $\Phi_d$ and not on $\ell$
(this is defined for $\Phi_d$-blocks
with arbitrary defect). Using this function, and the combinatorial
procedure to recover a Brauer tree from a perversity function, \cite{Cra3}
associates a {\em generic Brauer tree}
to a $\Phi_d$-block of defect $1$.
This is a planar-embedded tree, together with an exceptional vertex (but no multiplicity) and
the non-exceptional vertices are parametrized by the unipotent characters
in the given $\Phi_d$-block.
The Brauer tree of $A$ is conjectured in \cite{Cra3} to be
obtained from the generic Brauer tree,
by associating the appropriate multiplicity if it is greater than $1$, and turning the
exceptional vertex into a non-exceptional one if the multiplicity is $1$. The trees we construct in this paper in \S\ref{setrees} match the generic trees constructed in \cite{Cra3}, and hence we prove the following theorem.

\begin{theorem}
\label{th:maintheorem}
Let $A$ be a unipotent $\ell$-block with cyclic defect of
$G$. Then the unipotent characters of $KA$ form a unipotent $\Phi_d$-block
and the Brauer tree of $A$ is obtained from the generic
Brauer tree of that $\Phi_d$-block.
\end{theorem}

Since the trees constructed in \S\ref{setrees} match the conjectural
trees in \cite{Cra3} that would result from a perverse equivalence between
$D\rtimes E$ and $A$, we get the following corollary.

\begin{cor}
There is a perverse derived equivalence between $A$ and $D\rtimes E$
with perversity function $\gamma$.
\end{cor}

\subsubsection{Determination of the trees}

Let us now discuss the known Brauer trees.
The Brauer trees for classical groups were determined by Fong and
Srinivasan \cite{FS1,FS2}. The Brauer trees for the following exceptional
groups are known:  Burkhart \cite{Bur} for ${}^2B_2$,
Shamash \cite{Sh1} for $G_2$, Geck \cite{Ge} for ${}^3D_4$,
Hiss \cite{Hi90} for ${}^2G_2$ and ${}^2 F_4$,
Hiss--L\"ubeck \cite{HL} for $F_4$ and ${}^2E_6$ (building on earlier
work on $F_4$ by Wings \cite{Wi}) and Hiss--L\"ubeck--Malle \cite{HLM} 
for $E_6$.

More recently, the second and third authors determined in \cite{DR} the
Brauer trees of the principal $\Phi_h$-block of $E_7$ and $E_8$ for $h$
the Coxeter number, using new geometric methods which are also at the heart of
this paper. Also, the first author determined in \cite{Cra3} the
Brauer trees of several unipotent blocks with cyclic defect of 
$E_7$ and $E_8$.

We determine the remaining unknown trees. They
correspond to certain unipotent blocks of
${}^2E_6$ (cf.\ Remark \ref{re:2E6}), $E_7$ (\S \ref{sec:E7}) and $E_8$
(\S \ref{sec:E8}). We list in Table \ref{tab:unknown} the group $G$, the order
$d$ of $q$ modulo $\ell$ and the $d$-cuspidal pair (when the block is not
principal) associated to each of these blocks. We also indicate the type of
the minimal proper standard $F$-stable Levi subgroups $\bL_I$
with $\ell\mid |L_I|$.

\begin{table}
$$\begin{array}{|c|c|c|c|} \hline
\vphantom{\Big(} G & d & ([\bL,\bL],\lambda) & \ell \mid |L_I| \\\hline
{}^2 E_6 \vphantom{\Big(} & 12 &  & \\\hline
E_7 \vphantom{\big(^i}& 9 & & E_6 \\
\vphantom{\big(}& 10 & ({}^2A_2(q),\phi_{21}) & D_6\\
\vphantom{\big(_i}& 14 & & \\\hline
\vphantom{\big(^i}E_8 & 9 & (A_2(q), \phi_3) & E_6 \\
 &  & (A_2(q), \phi_{21}) &E_6 \\
 &  & (A_2(q), \phi_{1^3}) &E_6 \\
\vphantom{\big(} & 12 & ({}^3D_4(q),{}^3D_4[1]) & E_6, D_7\\
\vphantom{\big(} & 15 & &\\
\vphantom{\big(} & 18 & ({}^2 A_2(q),\phi_{21}) & E_7 \\
\vphantom{\big(} & 20 & & \\
\vphantom{\big(_i} & 24 & & \\\hline \end{array}$$
\caption{Blocks with unknown Brauer tree}
\label{tab:unknown}
\end{table}


\smallskip
Let us note that the Brauer trees of other blocks of exceptional groups were determined up to choices of fields of character values in each block. Using Lusztig's parametrization of unipotent characters we can remove this ambiguity by choosing appropriate
roots of unity in $\overline{\mathbb{Q}}_\ell$ with respect to $q$. 

\begin{cor}\label{corallfinitegroups} 
Let $G$ be a finite group with cyclic Sylow $\ell$-subgroups.
If $\ell\neq 29,41,47,59,71$, then the (unparametrized)
Brauer tree of the principal $\ell$-block of $G$ is known.
\end{cor}

\begin{proof}
Let $G$ be a finite group with a non-trivial cyclic Sylow $\ell$-subgroup.
Since the principal block of $G$ is isomorphic to that of $G/O_{\ell'}(G)$,
we can assume that $O_{\ell'}(G)=1$. If $G$ has a normal Sylow
$\ell$-subgroup, then the Brauer tree is a star. So, we assume $G$ does not
have a normal Sylow $\ell$-subgroup.
It follows from the classification of finite simple groups \cite[\S 5]{FoHa} that $G$ has a normal simple
subgroup $H$ with $G/H$ an $\ell'$-subgroup of $\mathrm{Out}(H)$.

If $H$ is an alternating group, the Brauer trees are foldings of those of symmetric groups, which are lines as all characters are real. If $H$ is a sporadic group, then the Brauer tree of the principal block of $H$ is known under the assumptions of $\ell$, \cite{HiLu,CHLM}.

Assume now $H$ is a finite group of Lie type. If $\ell$ is the defining
characteristic, then $H=\mathrm{PSL}_2(\bbF_\ell)$ and the Brauer tree of the
principal block is well known. Otherwise, the Brauer tree is
known by Theorem \ref{th:maintheorem}.
\end{proof}

\subsection{Properties of the trees}
\label{se:properties}
We assume here that $\bG$ is simple and we denote by $A$ a unipotent block
with cyclic defect group $D$ of $\OC G$. Let $E=N_G(D,b_D)/C_G(D)$,
where $(D,b_D)$ is a maximal $A$-subpair. We assume $|E|>1$.
We denote by $T$ the Brauer tree of $A$. Recall
(Theorem \ref{th:structureblocks}) that its $|E|$ unipotent
vertices are non-exceptional.
We define the {\em non-unipotent vertex} of $T$ to be the one corresponding
to the sum of the non-unipotent characters in $KA$. It is exceptional
if $|E|\neq |D|-1$.

\subsubsection{Harish-Chandra branches}

Let $I$ be an $F$-stable subset of $S$ and $X$ be a cuspidal simple
unipotent $KL_I$-module with central $\ell$-defect, i.e.,
such that $(\dim X)_\ell=[L_I:Z(L_I)]_\ell$.
Since the centre acts trivially on simple unipotent modules, the 
$\ell$-block $b_I$ of $L_I$ containing $X$ has central defect group,
and $X$ is the unique unipotent simple module in $b_I$.
This yields the following three facts.
\begin{itemize}
 \item[(a)] There exists a unique  (up to isomorphism) $\OC L_I$-lattice
$\widetilde X$ such that $X\simeq K\widetilde X$. The $kL_I$-module 
$k\widetilde X$ is irreducible.
\item[(b)] $X$ is the unique unipotent module that lifts $k\widetilde X$.  In particular $N_G(L_I,X) = N_G(L_I,k\widetilde X)$.
 \item[(c)] If $P$ is a projective cover of $\widetilde X$, then 
 $K\mathrm{ker}(P\twoheadrightarrow\widetilde X)$ has only non-unipotent
 constituents, therefore $R_{L_I}^G(X)$ and $K\mathrm{ker}(R_{L_I}^G(P)
 \twoheadrightarrow R_{L_I}^G(\widetilde X))$ have no irreducible constituents
 in common.
\end{itemize}
Under the properties (a) and (b), Geck showed in \cite[2.6.9]{Ge90} that
the endomorphism algebra $\mathrm{End}_{\OC G}(R_{L_I}^G(\widetilde X))$
is reduction-stable, i.e. 
$$ k\mathrm{End}_{\OC G}(R_{L_I}^G(\widetilde X)) \simeq 
\mathrm{End}_{kG}(R_{L_I}^G(k\widetilde X)).$$
Property (c) was used by Dipper (see \cite[4.10]{Di90}) to show that 
the decomposition matrix of $\mathrm{End}_{\OC G}(R_{L_I}^G(\widetilde X))$
embeds in the decomposition matrix of $b$. 

\smallskip

It follows from \cite{Ge92}
that the full subgraph of $T$ whose vertices are in the Harish-Chandra series
defined by $(L_I,X)$ is a union of lines. Note that \cite{Ge92} proves
a corresponding result for blocks of Hecke algebras at roots of unity,
in characteristic $0$. The fact that the tree does not change when reducing
modulo $\ell$ follows from the following two facts:
\begin{itemize}
\item a symmetric algebra over
a discrete valuation ring that is an (indecomposable) Brauer tree algebra
over the field of fractions and over the residue field has the same
Brauer tree over those two fields;
\item the blocks of the Hecke algebra
$\mathrm{End}_{\OC G}(R_{L_I}^G(\widetilde X))$
correspond to blocks of the Hecke algebra in characteristic
$0$ for a suitable specialization at roots of unity.
\end{itemize}

Each such line in $T$ is called a Harish-Chandra branch.
In particular, the {\em principal series part} of $T$ is the full subgraph
whose vertices are in the Harish-Chandra series of the trivial representation
of a quasi-split torus.

%
%
%
%
%
%
%
%
%

\begin{prop}
\label{pr:cuspedges}
Let $N$ be an edge of $T$ and let $V_1$ and $V_2$ be its vertices.
Let $I$ be a minimal $F$-stable subset
of $S$ such that ${^*R}_{L_I}^G(N)\not=0$.

If $\ell\nmid |L_I|$, then given $i\in\{1,2\}$, the $F$-stable subset
$I$ is also minimal with respect to the
property that ${^*R}_{L_I}^G(V_i)\not=0$.
\end{prop}

\begin{proof}
Let $M$ be an $\OC L_I$-lattice such that $KM$ is simple and
$N$ is a quotient of $R_{L_I}^G(M)$. Note that $M$ is projective, hence it
follows by Harish-Chandra theory that $KM$ is cuspidal.
Since $R_{L_I}^G(M)$ is projective,
it follows that $V_1$ and $V_2$ are direct summands of $KR_{L_I}^G(M)$.
The proposition follows by Harish-Chandra theory.
\end{proof}

\begin{cor}
\label{co:HCedges}
Suppose that $\ell\nmid |L_I|$ for all $F$-stable $I\subsetneq S$. 
Then the edges that are not in a Harish-Chandra branch are cuspidal.
\end{cor}

The following result is a weak form of \cite[Theorem 3.5]{Hi}.

\begin{prop}
If $\St$ is a vertex of $T$, then the edge corresponding to $\St_\ell$
connects $\St$ and the non-unipotent vertex.
\end{prop}

\begin{proof}
Recall that $b\Gamma_\psi$ is the projective cover of $\St_\ell$. Since
$\St$ is the unique unipotent component of $K\Gamma_\psi$,
the proposition follows.
\end{proof}

\begin{prop}
\label{pr:principalline}
Assume that $A$ is the principal block and
 $\ell\nmid |L_I|$ for any $F$-stable $I\subsetneq S$. Let $L$ be the
full subgraph of $T$ whose vertices are at distance at most $r$ from $1$.
Then $L$ is a line whose leaves are $1$ and $\St$.
\end{prop}

\begin{proof}
The tables in \cite[Appendix F]{GePf} show that the Brauer tree of the
principal
block of the Hecke algebra $\End_{\OC G}(R_T^G(\OC))$ is a line with
$r+1$ vertices, with leaves corresponding
to the trivial and sign characters. So, $T$ has a full subgraph $L$
that is a line with $r+1$ vertices and with leaves $1$ and $\mathrm{St}$.
Using Proposition \ref{pr:solomontits} and duality, we deduce that
all vertices at distance at most $r$ from $1$ are in $L$.
\end{proof}

\subsubsection{Real stem}

We fix a square root of $q^\delta$ in $K$ (specific choices
will be made in Section \ref{sec:determination}).
Let $V$ be a unipotent irreducible $KG$-module.
Let $w\in W$ such that
$V$ occurs in $\Hc^i(\rX(w),K)$. The eigenvalues of $F^\delta$ on
the $V$-isotypic component of $\Hc^i(\rX(w),K)$ are of the form 
$\lambda_V q^{\delta j}$ where $\lambda_V$ is a root of unity (depending
only on $V$, not on $w$ nor $i$), for some $j\in\frac{1}{2}\bbZ$.
Note that $\lambda_{V^*}=\lambda_V^{-1}$.

So, the vertices of the real stem of $T$ consist of the non-unipotent
vertex and the unipotent vertices corresponding to the $V$ such that
$\lambda_V=\pm 1$. For classical groups, all unipotent characters have this property, and are real valued, and for exceptional groups, the unipotent characters with this property are principal-series characters and $D_4$-series characters, which are real-valued by \cite[Proposition 5.6]{Ge03}, and cuspidal characters $G[\pm 1]$, which are rational-valued by \cite[Table 1]{Ge03}.


\subsubsection{Exceptional vertex}\label{sec:cuspidalexc}
Recall from Theorem \ref{th:structureblocks} that the $\ell$-block $A$ is
attached to a cuspidal pair $(\mathbf{L},\lambda)$.  A non-unipotent character 
in $A$ is obtained by Deligne--Lusztig induction
from an irreducible non-unipotent character of $L$. We give here a condition for 
that character to be cuspidal.

\begin{prop}
\label{pr:exceptionalcuspidal}
Assume that $\lambda$ is cuspidal and $\bL$ is not contained in any proper
$F$-stable parabolic subgroup of $\bG$. Let $\bP'$ be a proper $F$-stable parabolic subgroup of $\bG$ with unipotent
radical $\bU'$ and an $F$-stable Levi complement $\bL'$. The
$(\OC G,\OC L')$-bimodule $b\OC Ge_{U'}$ is projective and its restriction
to $\OC G$ is a direct sum of projective 
indecomposable $A$-modules corresponding to edges that do not
contain the non-unipotent vertex. 

In particular, the non-unipotent characters of $A$ are cuspidal.
\end{prop}

\begin{proof}
Let $Q$ be the subgroup of order $\ell$ of $D$ and let $g\in G$ such that
$Q^g\le L$. Let $\Delta_g Q=\{(x,g^{-1}xg)|x\in Q\}$.
We have $\Br_{\Delta_g Q}(b\OC Ge_{U'})\simeq 
\Br_{\Delta Q}(b\OC G  e_{gU'g^{-1}})=b_\lambda kL e_V$
where $V=gU'g^{-1}\cap L$ (Lemma \ref{le:Brauer}).
 By assumption, $\lambda$ is cuspidal and $\bP'\cap\bL$ is a proper $F$-stable
parabolic subgroup of $\bL$, hence $b_\lambda kL e_V=0$,
hence $\Br_{\Delta_g Q}(b\OC Ge_{U'})=0$. Since the
$((\OC G)\otimes (\OC L')^{\mathrm{opp}})$-module $b\OC G$ is
a direct sum of indecomposable modules with vertices trivial or
containing $\Delta_g Q$ for some $g\in G$, we deduce that
that the $(\OC G,\OC L_I)$-bimodule $b\OC Ge_{U'}$ is projective.

Let $\xi\in\Irr(KD)\setminus\{1\}$. Since $\Res_{[\bL,\bL]^F}^L(\lambda\otimes\xi)=
\Res_{[\bL,\bL]^F}^L(\lambda)$, it follows that $\lambda\otimes\xi$
is cuspidal. Theorem \ref{th:structureblocks} shows that every non-unipotent
character of $b$ is of the form
$(-1)^{r_G+r_L}(R_L^G(\lambda\otimes\xi))$ for some $\xi\in\Irr(KD)\setminus\{1\}$.
Proposition \ref{pr:inducedcuspidal} shows that such a character is cuspidal.
\end{proof}

%

\noindent
The assumptions of Proposition \ref{pr:exceptionalcuspidal}
are satisfied in the following cases:
\begin{itemize}
 \item $\bL = \bT$ contains a Sylow $\Phi_d$-torus of $G$ and $d$ is not a reflection degree of a proper
 parabolic subgroup of $W$ (e.g. $G=E_7(q)$ and $d=14$ or
$G=E_8(q)$ and $d \in\{15,20,24\}$). In that case
 the trivial character of $L$ is cuspidal, and no proper $F$-stable
parabolic subgroup of $\bG$ can contain a Sylow $\Phi_d$-torus.
 
 \item $G=E_8(q)$, $d=12$ and $([\bL,\bL]^F,\lambda)=(^3D_4(q),{}^3D_4[1])$ or 
$d=18$ and $([\bL,\bL]^F,\lambda)=({}^2 A_2(q),\phi_{21})$.
\end{itemize}

\begin{lem}
\label{le:minwforexc}
Let $w\in W$ and let $M$ be a simple $A$-module corresponding to an edge
containing $\chi_{\mathrm{exc}}$.

If $w$ has minimal length such that
$\RHom^\bullet_{kG}(\Rgc(\rX(w),k),M) \neq 0$, then
$\ell \mid |\bT^{wF}|$.

If $\ell \nmid |\bT^{vF}|$ for all $v\le w$, then 
$\RHom^\bullet_{kG}(\Rgc(\rX(\overline{w}),k),M)=0$.
\end{lem}

\begin{proof}
Let $M$ be as in the lemma and $w$ be minimal such that
$\RHom^\bullet_{kG}(\Rgc(\rX(w),k),M) \neq 0$. Assume that
$\ell\nmid|\bT^{wF}|$. We have
$(-1)^{\ell(w)}[b\Rgc(\rX(w),k)]=\sum_{\eta}a_\eta [P_\eta]$, where
$\eta$ runs over the edges of $T$ and $a_\eta\in\bbZ$. By
Proposition \ref{prop:middledegree} we
have $a_\mu>0$ where $\mu$ is the edge corresponding to $M$.
Since $\chi_{\mathrm{exc}}$ does not occur in $[\Rgc(\rX(w),K)]$, it follows
that there is an edge $\nu$ containing $\chi_{\mathrm{exc}}$ such that
$a_{\nu}<0$. Let $N$ be the simple $A$-module corresponding to $\nu$.
The complex $\RHom^\bullet_{kG}(\Rgc(\rX(w),k),N)$ has non-zero cohomology
in a degree other than $-\ell(w)$, hence there is $v<w$ such that
$\RHom^\bullet_{kG}(\Rgc(\rX(v),k),N) \neq 0$ by 
Proposition \ref{prop:middledegree}, a contradiction. The lemma follows.
\end{proof}

\subsubsection{In the stable category}
\label{se:stable}

Assume in \S\ref{se:stable} that $\delta=1$ and
$\bL$ is a maximal torus of $\bG$. This is a $\Phi_d$-torus.
Let $w\in W$ be a $d$-regular element.
The next result follows from \cite[Corollary 2.11 and its proof]{DR}.

\begin{prop}
\label{pr:OmegaDL}
Let $m\in\{0,\ldots,d-1\}$. The complex $\Rgc(\rX(w),k)_{(q^{ m})}$ is
isomorphic in $kG\mstab$ to $\Omega^{2m}k$.
\end{prop}

\begin{rmk}
\label{re:OmegaDLbar}
If $\bT^{vF}$ is an $\ell'$-group for all $v<w$, then Proposition
\ref{pr:OmegaDL} holds with $\rX(w)$ replaced by $\overline{\rX}(w)$.
\end{rmk}

\subsubsection{Coxeter orbits\label{secoxeter}}

The following lemma holds for general symmetric $\OC$-algebras
$A$ such that $kA$ is a Brauer tree algebra.

\begin{lem}
\label{le:twononzerocoh}
Let $C$ be a bounded complex of finitely generated projective $A$-modules.
Assume that $T$ has a subtree of the form
$$\xymatrix{V_t \ar@{-}[r]^{S_{t-1}} & V_{t-1} \ar@{--}[rr] &&
V_2 \ar@{-}[r]^{S_1} & V_1 \ar@{-}[r]^{S_0} & V_0}$$
all of whose vertices are non-exceptional, and:
\begin{itemize}
\item[\emph{(i)}] $K\H^i(C)=0$ for $i\not\in\{0,-t\}$, $\H^i(kC)=0$ for $i<-t$
 and $K\H^0(C)\simeq V_0$;
\item[\emph{(ii)}] given an edge $M$ of $T$, given an integer $i<t$
and given a map $f\in\Hom_{D^b(A)}(C,M[i])$, the induced map
from the torsion part of $\H^{-i}(C)$ to $M$ vanishes;
\item[\emph{(iii)}] letting $M$ be an edge of $T$
that contains $V_i$, and assuming that $M$ is strictly between $S_{i-1}$ and $S_i$ in
the cyclic ordering of edges at $V_i$ (for $0<i\le t-1$) or $M\neq S_0$ (for
$i=0$), then $\Hom_{D^b(A)}(C,M[j])=0$ for $j\in\{i,i+1\}\cap\{0,\ldots,t-1\}$;
\item[\emph{(iv)}] $S_i$ is not a composition factor of the torsion part
of $\H^{-i+1}(C)$ for $1\le i\le t-1$.
\end{itemize}
Then 
$C$ is homotopy equivalent to
$$0\to P\to P_{S_{t-1}}\xrightarrow{\delta_{t-1}}P_{S_{t-2}}\to\cdots
\xrightarrow{\delta_1}P_{S_0}\to 0$$
where $P$ is a projective $A$-module in degree $-t$ with
$KP\simeq K\H^{-t}(C)\oplus V_t$ and $\Hom_A(P_{S_i},P_{S_{i-1}})=
\OC\delta_i$.
Furthermore, given $i\in\{0,\ldots, t-2\}$,
the composition factors of the torsion part of
$\H^{-i}(C)$ correspond to the edges strictly between $S_i$ and $S_{i+1}$
in the cyclic ordering of edges at $V_i$.

If $V=K\H^{-t}(C)$ is simple and distinct from $V_{t-1}$, then there is an edge
$S_t$ between $V$ and $V_t$ and $P\simeq P_{S_t}$.
Furthermore, the composition factors of the torsion part of
$\H^{-t+1}(C)$ correspond to the edges strictly between $S_{t-1}$ and $S_t$
in the cyclic ordering of edges at $V_t$.
\end{lem}

\begin{proof}
We can assume that $C$ has no non-zero direct summand homotopic to zero.
Since $\H^{<-t}(kC)=0$, it follows that $C^{<-t}=0$.
Let $m$ be maximal such that $C^m\neq 0$. Suppose that $m>0$. By (i),
$\H^m(C)$ is a non-zero torsion $A$-module. Let $M$ be a simple quotient
of $\H^m(C)$. Assumption (ii) gives a contradiction.
We deduce that $m=0$.

By (iii), any simple quotient of $\H^0(C)_\mathrm{free} =
\H^0(C)/\H^0(C)_\mathrm{tor}$ is isomorphic to $S_0$. Moreover,
since $K\H^0(C)_\mathrm{free} = V_0$ and $S_0$ occurs only
once in any $\ell$-reduction of $V_0$, there exists a surjective
map $P_{S_0} \twoheadrightarrow  \H^0(C)_\mathrm{free}$.
It follows that 
there is an isomorphism $P_{S_0}\oplus Q\xrightarrow{\sim}C^0$
such that the composite map $KQ\to KC^0\to K\H^0(C)$ vanishes.
Let $N$ be the image of $P_{S_0}$ in
$\H^0(C)$. Suppose that there is a simple quotient $M$ of $\H^0(C)$ vanishing on $N$
(\emph{i.e.} such that $N$ is in the kernel of the quotient map $\H^0(C)\twoheadrightarrow M$).
Then $M$ is a quotient of the torsion part of $\H^0(C)$ and the composite map
$Q \rightarrow \H^0(C) \rightarrow M$ is non-zero.  We deduce that this map
induces a non-zero map from the torsion part of $\H^0(C)$ to $M$, which 
contradicts (ii). Consequently the retriction of $C \twoheadrightarrow \H^0(C)$
to $P_{S_0}$ is surjective and $Q = 0$ by minimality of $C$. 

\smallskip
Given $1\le i\le t-1$, fix $\delta_i:P_{S_i}\to P_{S_{i-1}}$ such that
$\Hom_{A}(P_{S_i},P_{S_{i-1}})=\OC \delta_i$. We put $\delta_0=0:P_{S_0}\to 0$.
We prove by induction on $i\in\{0,\ldots,t-1\}$ that
$0\to C^{-i}\to C^{-i+1}\to\cdots$ is
isomorphic
to the complex $0\to P_{S_i}\xrightarrow{\delta_i}P_{S_{i-1}}\to\cdots \to
P_{S_1} \xrightarrow{\delta_1}P_{S_0}\to 0$, where $P_{S_0}$ is in degree
$0$.
This holds for $i=0$ and we assume now this holds for some $i\le t-2$.
We have $\dim\Hom_{kA}(kP_{S_{i+1}},kP_{S_i})=1$ and we denote by
$N$ the image of a non-zero map $P_{S_{i+1}}\to P_{S_i}$. It is
contained in $k\ker\delta_i$.
Let $M$ be a composition factor of
$(k\ker\delta_i)/N$.
If $i=0$, then the edge corresponding to $M$ contains $V_0$ and
$M{\not\simeq}S_0$ or
it contains $V_1$ and is strictly between $S_0$ and $S_1$ in
the cyclic ordering of edges at $V_1$.
If $i>0$, then the edge corresponding to $M$ contains $V_i$ and is
strictly between $S_{i-1}$ and $S_i$ in
the cyclic ordering of edges at $V_i$ or it contains $V_{i+1}$ and
is strictly between $S_i$ and $S_{i+1}$ in
the cyclic ordering of edges at $V_{i+1}$.
By (iii), $P_M$ is not a direct summand
of $C^{-i-1}$. It follows from (iv) that there is an isomorphism
$P_{S_{i+1}}\oplus Q\xrightarrow{\sim}C^{-i-1}$ such that the composition
$Q\to C^{-i-1}\to C^{-i}$ vanishes. Let $M$ be a simple quotient of $Q$.
By minimality of $C$, $M$ occurs as a quotient of $H^{-i-1}(C)$, which is torsion by (i).
So (ii) gives a contradiction.
We deduce that $C^{-i-1}\simeq P_{S_{i+1}}$ and the differential $kC^{-i-1}\to
kC^{-i}$ is not zero. This shows that the induction statement holds
for $i+1$.

We deduce that $C$ is isomorphic to
$$0\to P\to
P_{S_{t-1}}\xrightarrow{\delta_{t-1}}P_{S_{t-2}}\to\cdots\to
P_{S_1}\xrightarrow{\delta_1} P_{S_0}\to 0$$
for some projective $A$-module $P$ in degree $-t$. We have
$[KP]=(-1)^t[KC]+[KP_{S_{t-1}}]-[KP_{S_{t-2}}]+\cdots+(-1)^t[KP_{S_0}]=
[K\H^{-t}(C)]+[V_t]$, hence $KP\simeq K\H^{-t}(C)\oplus V_t$.

If $V=K\H^{-t}(C)$ is simple, then $KP\simeq V\oplus V_t$, hence
$P\simeq P_{S_t}$ where $S_t$ is the edge containing $V$ and $V_t$. The last
statement follows from the fact that the differential $kP\to kP_{S_{t-1}}$
is non-zero.
\end{proof}

The following theorem deals with direct summands of
$\tRgc(\rX(c),\mathcal{O})$ that have exactly two non-zero cohomology groups
over $K$. Extra assumptions on the block are needed here.

\begin{theorem}\label{thm:coxeter}
Assume that $\ell \nmid |\bT^{cF}|$. Let $C$ be a direct summand of
$b\tRgc(\rX(c),\mathcal{O})$ in $\Ho^b(\mathcal{O}G\mmod)$. Suppose that there
are $r'\ge r$ and $t>0$ such that
  \begin{itemize}
    \item[\emph{(i)}] the torsion part in $\H^*(C)$ is cuspidal,
    \item[\emph{(ii)}] $\H^i(KC) =0$ for $i{\not\in}\{r',r'+t\}$ and
$V_0=\H^{r'+t}(KC)$ and $V'=\H^{r'}(KC)$ are simple $KG$-modules, and
    \item[\emph{(iii)}] $T$ has a subgraph 
with non-exceptional vertices and non-cuspidal edges
$$\xymatrix{V_t \ar@{-}[r]^{S_{t-1}} & V_{t-1} \ar@{--}[rr] &&
V_2 \ar@{-}[r]^{S_1} & V_1 \ar@{-}[r]^{S_0} & V_0}$$
such that 
$\xymatrix{V_{t-1} \ar@{--}[rr] &&
V_2 \ar@{-}[r]^{S_1} & V_1 \ar@{-}[r]^{S_0} & V_0}$
is a connected component
of the subgraph of $T$ obtained by removing the edge $S_{t-1}$ and
all cuspidal edges.
\end{itemize}
\noindent Then:
\begin{itemize}
\item there is an edge $S_t$ between $V_t$ and $V'$ and
$C$ is homotopy equivalent to
$$ C'=
0 \longrightarrow P_{S_t}
\longrightarrow P_{S_{t-1}} \longrightarrow \cdots \longrightarrow P_{S_0} \longrightarrow 0$$
with $P_{S_t}$ in degree $r'$;
\item the complex $C'$ is, up to isomorphism,
the unique complex such that the differential
$P_{S_i}\to P_{S_{i-1}}$
generates the $\OC$-module
$\Hom(P_{S_i},P_{S_{i-1}})$ for $1\le i\le~t$;
\item the composition factors of the torsion part of $\H^{r'+t-i}(C)$
correspond to the edges strictly between $S_i$ and $S_{i+1}$ in the cyclic
ordering 
of edges at $V_{i+1}$ (for $0\le i\le t-1$). In particular, the edges between
$S_{t-1}$ and $S_t$ around $V_t$ are also cuspidal.
\end{itemize}
\end{theorem}

\begin{proof}
We apply Lemma \ref{le:twononzerocoh} to $C[r'+t]$.
Assumptions (i), (ii) and (iv) of the lemma follow from the assumptions of the
theorem. By Corollary \ref{cor:middledegree}, we have
$\Hom_{D^b(A)}(C,M[i])=0$ for $i>r$ and $M$ cuspidal. If $M$ is simple
non-cuspidal and not in $\{S_0,\ldots,S_{t-1}\}$, then $M$ does not
occur as a composition factor of $\H^i(kC)$ for $i>r'$. This shows
that Assumption (iii) of the lemma holds.
The theorem follows.
\end{proof}

Assumption (iii) in Theorem \ref{thm:coxeter} may look rather difficult to check if only part of 
the tree is known. However, it will be satisfied for most of the Brauer trees 
we will consider, thanks to the following proposition.

\begin{prop}\label{prop:assumptionsforcoxeter}
Let $V$ be a simple unipotent $KA$-module.
Assume that
\begin{itemize}
\item $\ell \nmid |\bT^{cF}|$,
\item
$V$ is a leaf of $T$,  i.e. $V$ remains irreducible after $\ell$-reduction,
\item the Harish-Chandra branch of $V$ has at least $t$ edges, and
\item $\ell \nmid |L_I|$ for all $F$-stable subsets $I\subsetneq S$.
\end{itemize}
Then assumptions (i) and (iii) in Theorem \ref{thm:coxeter} are satisfied with
$C = b\tRgc(\rX(c),\mathcal{O})$ and $V_t,\ldots,V_0 = V$ being the
Harish-Chandra branch ending at the leadf $V$.
\end{prop}

\begin{proof}
Assumption (i) is satisfied by Proposition \ref{pr:torsionCoxeter},
while assumption (iii) is satisfied by Corollary \ref{co:HCedges}.
\end{proof}

\begin{cor}\label{co:Coxetertriv}
Let $b$ be the block idempotent of the principal block of $\mathcal{O}G$. 
Assume that $\ell \nmid |\bT^{cF}|$ and
$\ell \nmid |L_I|$ for all $F$-stable subsets $I\subsetneq S$.

Let $T'$ be the full subgraph of $T$ with vertices at distance at most $r+1$
of the trivial character.

\begin{itemize}
\item
The real stem of $T'$ is a line with leaves $1$ and the non-unipotent vertex
\item the edge $\St_\ell$ has vertices $\St$ and the non-unipotent vertex
\item any
non-real vertex of $T'$ is connected to $\St$ by an edge
\item $V=K\Hc^r(\rX(c),\mathcal{O})_{(q^r)}$ is a non-real
simple $KGb$-module
and the edge connecting $V$ and
$\St$ comes between the one connecting $\St$ to a unipotent vertex
and $\St_\ell$ in the cyclic ordering of edges at $\St$.
\end{itemize}

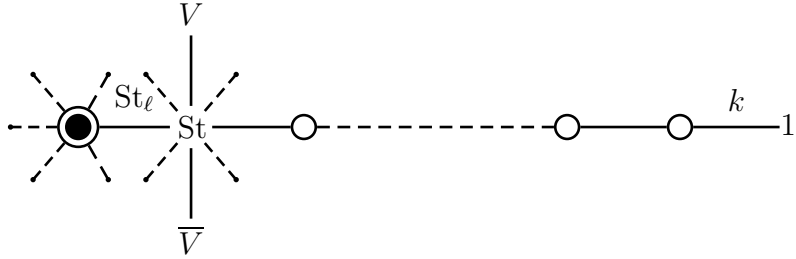
\begin{center}
 \begin{pspicture}(10,4.5)
  \psset{linewidth=1pt}

 \cnode[fillstyle=solid,fillcolor=black](0,2){5pt}{AA}
 \cnode(0,2){8pt}{A}
 \cnode(0.4,2.7){0pt}{A1}
 \cnode(0.4,1.3){0pt}{A2}
 \cnode(-0.6,2.7){0pt}{A3}
 \cnode(-0.6,1.3){0pt}{A4}
 \cnode(-0.9,2){0pt}{A5}
  \cnode[linestyle=none](1.5,2){8pt}{B}
 \cnode(2.1,2.7){0pt}{B1}
 \cnode(2.1,1.3){0pt}{B2}
 \cnode(0.9,2.7){0pt}{B3}
 \cnode(0.9,1.3){0pt}{B4}
 \cnode[linestyle=none](1.5,3.5){8pt}{BV}
 \cnode[linestyle=none](1.5,0.5){8pt}{BV2}
 \cnode(2.1,4.2){0pt}{V1}
 \cnode(0.9,4.2){0pt}{V2}
 \cnode(2.1,-0.2){0pt}{V3}
 \cnode(0.9,-0.2){0pt}{V4}
 \cnode(3,2){5pt}{C}
  \cnode(6.5,2){5pt}{C2}
 \cnode(8,2){5pt}{D}
 \cnode[linestyle=none](9.5,2){5pt}{E}

  \ncline{A}{B} \naput{$\mathrm{St}_\ell$} 
  \ncline{B}{C} \ncput[npos=-0.25]{$\mathrm{St}$}
  \ncline[linestyle=dashed]{C}{C2}
  \ncline{C2}{D}  
  \ncline{D}{E} \naput{$k$}   \ncput[npos=1.1]{$\vphantom{\Big(}1$} 
  \ncline{B}{BV}\ncput[npos=1.3]{$V$}
  \ncline{B}{BV2} \ncput[npos=1.3]{$V^*$}
  \ncline[linestyle=dashed]{BV}{V1}  
  \ncline[linestyle=dashed]{BV}{V2}  
  \ncline[linestyle=dashed]{BV2}{V3}  
  \ncline[linestyle=dashed]{BV2}{V4}  
  \ncline[linestyle=dashed]{A}{A1}  
  \ncline[linestyle=dashed]{A}{A2}  
  \ncline[linestyle=dashed]{A}{A3}  
  \ncline[linestyle=dashed]{A}{A4}
  \ncline[linestyle=dashed]{A}{A5}  
  \ncline[linestyle=dashed]{B}{B1}  
  \ncline[linestyle=dashed]{B}{B2}  
  \ncline[linestyle=dashed]{B}{B3}  
  \ncline[linestyle=dashed]{B}{B4}

\end{pspicture}
\end{center}
\end{cor}

\begin{proof}
The description of Frobenius eigenvalues on the cohomogy of $\rX(c)$
in \cite[(7.3)]{Lu} shows that
$K\Hc^i(\rX(c),\mathcal{O})_{(q^r)}=0$ for $i\not\in\{r,2r\}$ and
$V=K\Hc^r(\rX(c),\mathcal{O})_{(q^r)}$ is simple, under our assumptions
on $\ell$.
The result follows now from 
Theorem \ref{thm:coxeter}, Corollary \ref{prop:assumptionsforcoxeter}
and Propositions \ref{pr:principalline} and \ref{gelfandgraevxbar}(iii).
\end{proof}

By Remark \ref{rmk:torsion}, the previous results have a counterpart for
the compactification.

\begin{prop}\label{prop:resultsxbar}
Lemma \ref{lem:onedegree}, Theorem \ref{thm:coxeter}, Proposition
\ref{prop:assumptionsforcoxeter} and Corollary \ref{co:Coxetertriv}
hold with $\overline{\rX}(c)$ instead of $\rX(c)$ if we replace the assumption
$\ell \nmid |\bT^{cF}|$ by $\ell \nmid |\bT^{vF}|$ for all $v \leq c$.
\end{prop}

\begin{rmk}
We have $|\bT^{cF}| = (q+1)(q^6-q^3+1) = \Phi_2(q) \Phi_{18}(q)$ for $\bG$
simple of type $E_7(q)$, and $|\bT^{cF}| = q^8 + q^7 - q^5 - q^4 - q^3 + q + 1
= \Phi_{30}(q)$ for $\bG$ simple of type $E_8(q)$. In particular, when $\ell$ is good and $d\notin \{2,h\}$, the condition $\ell \nmid |\bT^{cF}|$ will always be satisfied for $E_7(q)$ and $E_8(q)$.
\end{rmk}

\smallskip

\begin{rmk}
We can easily read off the cohomology of a complex $C$ as in
Theorem \ref{thm:coxeter} from the Brauer tree. As a consequence of
Theorem \ref{th:maintheorem}, one can check
that the cohomology of $C$ is concentrated in degrees $r$ and $r+t$ (and irreducible in degree $r+t$), and is torsion-free. 
Other calculations in \S\ref{setrees} give a strong evidence that the cohomology of a variety associated to a Coxeter element is always torsion-free. By \cite{BR2} this holds for groups of type $A$. Such a statement does not hold for more general Deligne--Lusztig varieties: if
$\Hc^{2\ell(w)-1}(\rX(w),K) =0$ and $\ell$ divides $|\bT^{wF}|$, then
$\Hc^{2\ell(w)-1}(\rX(w),k) = \H^1(\rX(w),k)^*$ is non-zero since the
connected Galois covering $\rY(\dot w) \twoheadrightarrow \rX(w)$ yields
non-trivial connected abelian $\ell$-coverings. Therefore by the universal
coefficient theorem, $\Hc^{2\ell(w)-1}(\rX(w),\mathcal{O})$ is a torsion
module. However, one can ask whether the property $\ell \nmid |\bT^{wF}|$
forces the cohomology to be torsion-free (see also Proposition
\ref{xbartorsion}). 
\end{rmk}

\subsection{Summary of the algebraic methods\label{sec:arguments}}
We summarize here some facts and arguments about Brauer trees that we shall use throughout \S\ref{setrees}. We consider a unipotent block
with a cyclic defect group and non-trivial automizer. We also assume that
the block is real (this is the case for all the unipotent blocks we will
consider).
\renewcommand{\descriptionlabel}[1]{\hspace{\labelsep}{\bf#1}}
\begin{description}[leftmargin=7mm]

  \item[(Parity)] The distance between two unipotent vertices is even if and
only if their degree are congruent modulo $\ell$.

  \item[(Real stem)] The collection of unipotent vertices $V$ with
$\lambda_V=\pm 1$, together with the non-unipotent vertex,
form a subgraph of the Brauer tree in the shape of a line, called the \emph{real stem}. Taking duals of characters corresponds to a reflection of the tree in
the real stem.

  \item[(Hecke)] The union of the full subgraphs of $T$ obtained by considering
unipotent characters in a given Harish-Chandra series is a collection of lines,
which is known.

\item[(Degree)] The dimension of the simple module corresponding to an edge is the alternating sum of the degrees of the vertices in a minimal path from the edge to a leaf. This dimension is a positive integer, and this can be used to show that certain configurations are not possible. Broadly speaking, the effect of this condition is to force the degrees of the unipotent characters, as polynomials in $q$, to increase towards the non-unipotent node.

\item[(Steinberg)] The vertices of the edge $\St_\ell$ are $\St$ and the
non-unipotent vertex. 
If the proper standard Levi subgroups of $G$ are $\ell'$-groups,
then the full subgraph of $T$
whose vertices are at distance at most $r$ from $1$ is a line whose leaves are
$1$ and $\St$ and the edge $\St_\ell$ is cuspidal.
\end{description}

Our strategy is to first study the `mod-$\ell$ generalized
eigenspaces' of $F$ on
the complex of cohomology of a Coxeter Deligne--Lusztig variety (or its
compactification), for those
eigenvalues corresponding to unipotent
cuspidal $KG$-modules. This gives information about
the location of the corresponding vertex with respect to the real stem.

A second step is required if there are cuspidal unipotent
$KG$-modules in the block
that do not occur in the cohomology of a Coxeter Deligne--Lusztig variety.
In that case, we consider the eigenspaces in the complex of cohomology
of a Deligne--Lusztig variety associated to a $d$-regular element, which is
minimal for the property that this module occurs in the cohomology.

\section{Determination of the trees\label{setrees}}
\label{sec:determination}

We now determine the Brauer trees of the blocks from Table \ref{tab:unknown}.
The edges corresponding
to cuspidal simple modules will be drawn as double lines.

\smallskip

Throughout this section, $A$ denotes a block of $\mathcal{O}G$ with cyclic
defect and $b$ is the corresponding block idempotent.

\smallskip

We shall start with the case of exceptional groups of type $E_7$
and $E_8$, for which $\delta = 1$. If $G$ is a standard Levi of a
simple group of type $E_8$, it follows from Lusztig's classification
that a cuspidal unipotent character $\rho$ of $G$ is 
uniquely determined by the eigenvalue of $F$ on the $\rho$-isotypic 
part of the cohomology of the various Deligne--Lusztig varieties. 
Following the convention in  \Chevie{} \cite{Mi15}, we will denote by $G[\alpha]$ 
a cuspidal simple unipotent $KG$-module such that the eigenvalues
of $F$ in the $G[\alpha]$-isotypic component of $\Hc^*(\rX(w),K)$ are
in $q^{\frac{1}{2} \bbZ} \alpha$ for any $w\in W$, with the exception
of the cuspidal unipotent character of $D_4(q)$ which will be denoted
by $D_4$ and not $D_4[-1]$. The choice of a square root of $q$ is actually 
only  needed when considering the two cuspidal characters of $E_7(q)$.
The roots of unity $\alpha$ which occur has always order $6$ or less. 

\smallskip

For the $\Phi_d$-blocks we will study it will be enough to consider the
following situations.
\begin{itemize}
 \item If $3\mid d$ (resp.  $4\mid d$, $5\mid d$), we denote by $\theta$ (resp. 
   $\mathrm{i}$, $\eta$)
    the unique third (resp. fourth, fifth) root of unity in $\OC$ whose image in $k$
    is $q^{d/3}$ (resp. $q^{d/4}$, $q^{d/5}$). The corresponding cuspidal characters are
    $E_6[\pm \theta]$, $E_6[\pm \theta^2]$, $E_8[\pm \theta]$ and $E_8[\pm \theta^2]$
    (resp. $E_8[\pm \mathrm{i}]$, $E_8[\eta^j]$ for $j =1,\ldots, 4$). 
 \item If $d=2e$ with $e$ odd, we fix a square root $\sqrt q$ of $q$ in 
   $\OC^\times$ and we denote by $\mathrm{i}$ the unique fourth root of unity in $\OC$
   whose image in $k$ is $(\sqrt q)^{e}$. The corresponding cuspidal characters are
   $E_7[\pm \mathrm{i}]$. 
\end{itemize}

\subsection{Groups of type $E_7$\label{sec:E7}}

For groups of type $E_7$, we need to consider the principal $\Phi_d$-blocks for $d=9, 14$ and the $\Phi_{10}$-block corresponding to the $d$-cuspidal pair $({}^2A_2(q).(q^5+1),\phi_{21})$.

\subsubsection{$d=14$} \label{se:E7d14}

In that case, the proper Levi subgroups of $G$ are $\ell'$-groups.
Let us determine the Brauer tree of the principal $\Phi_{14}$-block of
$E_7(q)$. Using (Hecke), (Degree) and (Steinberg) arguments, we obtain
the real stem as
shown in Figure \ref{14E7} (see the Appendix). The difficult part is to
locate the two complex conjugate cuspidal unipotent characters.
Let $C = b\tRgc(\rX(c),\mathcal{O})_{(-1)}$ be the generalized
`$-1\pmod\ell$-eigenspace' of $F$.
By \cite[Table 7.3]{Lu}, we have
$$KC \, \simeq \, (E_7[\II])[-7] \oplus K[-14],$$
where $E_7[\II]$ is defined as the unipotent cuspidal $KG$-module that 
appears with an eigenvalue of $F$ congruent to $-1$ modulo $\ell$
in $\Hc^7(\rX(c),K)$.

Corollary \ref{co:Coxetertriv} shows
 that $E_7[\II]$ is connected to $\mathrm{St}$ and that
it is the first edge coming after the edge $S_6$ in the cyclic ordering
of edges containing $\St$. This completes the determination of the tree.

\smallskip
Let us describe more explicitely the minimal representative of the complex
$C$.
 Let $k= S_0, S_1, \ldots, S_6$ be the non-cuspidal modules forming the path from the characters $1$ (the character of the trivial $KG$-module $K$) to $\mathrm{St}$ in the tree (see Figure \ref{E7d14}).

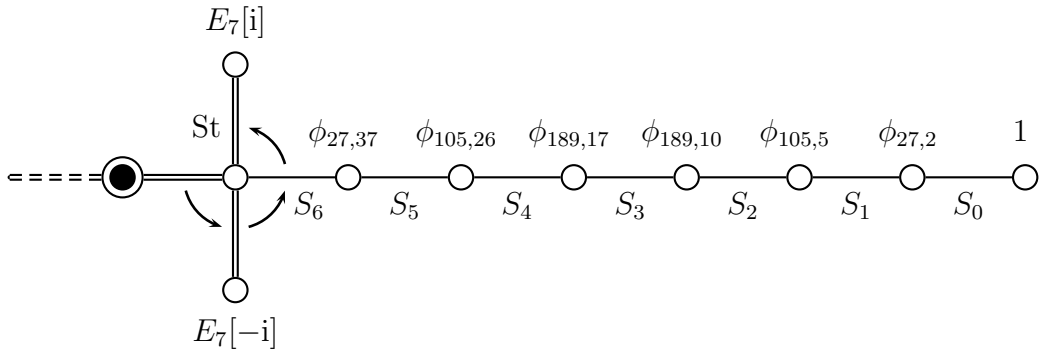
\begin{figure}[h] 
\begin{center}
\begin{pspicture}(11,4.4)

  \cnode[fillstyle=solid,fillcolor=black](0,2.2){5pt}{A2}
    \cnode(0,2.2){8pt}{A}
  \cnode(1.5,2.2){5pt}{B}
  \cnode(3,2.2){5pt}{C}
  \cnode(4.5,2.2){5pt}{D}
  \cnode(6,2.2){5pt}{E}
  \cnode(7.5,2.2){5pt}{F}
  \cnode(9,2.2){5pt}{G}
  \cnode(10.5,2.2){5pt}{H}
  \cnode(12,2.2){5pt}{K}  
  \cnode(1.5,0.7){5pt}{I}
  \cnode(1.5,3.7){5pt}{J}
  \cnode(-1.5,2.2){0pt}{L}

  \ncline[nodesep=0pt,doubleline=true]{A}{B}\naput[npos=0.8]{$\vphantom{\Big(} \mathrm{St}$}
    \ncline[nodesep=0pt]{B}{C}\naput[npos=1.1]{$\vphantom{\Big(}\phi_{27,37}$}\nbput[npos=0.7]{$S_6$}
  \ncline[nodesep=0pt]{C}{D}\naput[npos=1.1]{$\vphantom{\Big(}  \phi_{105,26}$}\nbput{$S_5$}
  \ncline[nodesep=0pt]{D}{E}\naput[npos=1.1]{$\vphantom{\Big(}\phi_{189,17}$}\nbput{$S_4$}
  \ncline[nodesep=0pt]{E}{F}\naput[npos=1.1]{$\vphantom{\Big(} \phi_{189,10}$}\nbput{$S_3$}
  \ncline[nodesep=0pt]{F}{G}\naput[npos=1.1]{$\vphantom{\Big(} \phi_{105,5}$}\nbput{$S_2$}
  \ncline[nodesep=0pt]{G}{H}\naput[npos=1.1]{$\vphantom{\Big(} \phi_{27,2}$}\nbput{$S_1$}
  \ncline[nodesep=0pt,doubleline=true]{B}{I}\ncput[npos=1.65]{$\vphantom{\Big(} E_7[-\II] $}
  \ncline[nodesep=0pt,doubleline=true]{B}{J}\ncput[npos=1.65]{$\vphantom{\Big(} E_7[\II] $}
  \ncline[nodesep=0pt]{H}{K}\naput[npos=1.1]{$\vphantom{\Big(} 1 $}\nbput{$S_0$}
  \ncline[linestyle=dashed,nodesep=0pt,doubleline=true]{A}{L}
  \psellipticarc[linewidth=1pt]{->}(1.5,2.2)(0.7,0.7){15}{75}
  \psellipticarc[linewidth=1pt]{->}(1.5,2.2)(0.7,0.7){195}{255}
  \psellipticarc[linewidth=1pt]{->}(1.5,2.2)(0.7,0.7){285}{345}

\end{pspicture}
\end{center}
\caption{Right-hand side of the Brauer tree of the principal $\Phi_{14}$-block of $E_7(q)$} 
\label{E7d14}
\end{figure}

\smallskip

The complex $b\tRgc(\rX(c),k)_{(-1)}^{\redu}$ is given as follows:
$$ 0 \longrightarrow \begin{array}{c} \textcolor{violet}{E_7[\II]} \\ \textcolor{violet}{\St_\ell} \\ \textcolor{violet}{E_7[-\II]} \\ \textcolor{violet}{S_6} \\ \fbox{$E_7[\II]$} \end{array} \longrightarrow \begin{array}{c} \textcolor{purple}{S_6} \\ \textcolor{violet}{E_7[\II]} \hphantom{AAAA} \\ \textcolor{violet}{\St_\ell} \qquad \textcolor{purple}{S_5} \\ \textcolor{violet}{E_7[-\II]} \hphantom{AAAA} \\ \textcolor{violet}{S_6} \end{array}  \hskip -4mm \longrightarrow \begin{array}{c} \textcolor{violet}{S_5} \\\textcolor{purple}{S_6} \quad \textcolor{violet}{S_4}  \\ \textcolor{purple}{S_5} \end{array} \longrightarrow \begin{array}{c} \textcolor{purple}{S_4} \\ \textcolor{violet}{S_5} \quad \textcolor{purple}{S_3}  \\ \textcolor{violet}{S_4} \end{array}  \longrightarrow \cdots \longrightarrow  \begin{array}{c} \fbox{$k$} \\ \textcolor{violet}{S_1} \\  \textcolor{violet}{k} \end{array} \longrightarrow 0.$$

\begin{rmk}
\label{re:2E6}
 This argument applies to many other trees, especially to those associated to the principal $\Phi_d$-block when $d$ is the largest degree of $W$ distinct from the Coxeter number
(in that case the assumptions on $\ell$ in Proposition \ref{prop:assumptionsforcoxeter} are satisfied). This shows for example that the Brauer tree of the principal $\Phi_{12}$-block of ${}^2 E_6(q)$ given in \cite{HL} is valid without any restriction on $q$. It is also worth mentioning that it gives not only the planar embedding but also the labelling of the vertices with respect to Lusztig's classification of unipotent characters (in terms of eigenvalues of Frobenius). In the previous example $\mathrm{Ext}_{kG}^1(E_7[\II],\St_\ell) \neq 0$ whereas $\mathrm{Ext}_{kG}^1(E_7[-\II],\St_\ell) = 0$. 
\end{rmk}

\subsubsection{$d=9$}
It follows from Lemma \ref{le:relativeproj} that $A$
is Harish-Chandra
projective relatively to the principal block of $E_6(q)$, hence
$A$ has no cuspidal simple modules.

The real stem gives most of the Brauer tree of the principal $\ell$-block (see Figure \ref{9E7}). It remains to locate the pairs of
complex conjugate characters $\{E_6[\theta]_{\varepsilon},E_6[\theta^2]_{\varepsilon}\}$ and 
 $\{E_6[\theta]_{1},E_6[\theta^2]_{1}\}$.
To this end we use the homological information contained in the cohomology of the Coxeter variety $\rX(c)$. 
Let $I$ be a proper subset of $S$. If $\bL_I$ is not a group of type $E_6$, then $L_I$ is an $\ell'$-group and the cohomology of $\rX(c_I)$ is torsion-free by \cite[Proposition 3.1]{Du4}. This remains true when $\bL_I$ has type $E_6$.

\begin{lem}\label{notorsionE6d9} If $q$ has order $9$ modulo $\ell$, the cohomology of the Coxeter variety in a simple group of type $E_6$ is torsion-free.
\end{lem}
\begin{proof} Denote by $\rX$ the Coxeter variety of $E_6(q)$. By 
Proposition \ref{pr:torsionCoxeter}, the torsion of $\Hc^*(\rX,\OC)$
is cuspidal. Let $\lambda\in k^\times$ and let $C_\lambda=
R\Gamma_c(\rX,\OC)_{(\lambda)}$.

Assume that $\lambda\not\in\{1,q^6\}$. The
cohomology of $\H^*(KC_\lambda)$ is
an irreducible module $V$ corresponding to a block idempotent 
$b_\lambda$ of defect zero.

If $V$ is cuspidal, then it occurs in degree $6$ in $\H^*(KC_\lambda)$,
hence $\H^*(C_\lambda)$ is torsion-free by Lemma \ref{lem:onedegree}.
If $V$ is not cuspidal, then $\H^*(b_\lambda C_\lambda)$ has no torsion.
On the other hand, $\H^*((1-b_\lambda)C_\lambda)$ is torsion and cuspidal,
hence $0$ by Lemma \ref{lem:onedegree}.

\smallskip

Assume now that $\lambda=1$. We have $\H^6(KC_1)=\St\oplus E_6[\theta^2]$ and
$\H^i(KC_1)=0$ for $i\neq 6$, so
$\H^*(C_1)$ is torsion-free by Lemma \ref{lem:onedegree} (so,
$\mathrm{St}+E_6[\theta^2]$ is a projective character of
$E_6(q)$, as was shown in \cite{HLM}).

\smallskip

Assume finally that $\lambda=q^6$. We have $\H^6(KC_\lambda)=E_6[\theta]$,
$\H^{12}(KC_\lambda)=1$ and $\H^i(KC_\lambda)=0$ for $i\not\in\{6,12\}$.
Corollary \ref{co:Coxetertriv} shows that $\H^*(C_\lambda)$ is torsion-free.
\end{proof}

From this lemma together with Proposition \ref{pr:torsionCoxeter},
 we deduce that the torsion of $\Hc^*(\rX(c),\mathcal{O})$ is cuspidal,
hence the principal block part of $\Hc^*(\rX(c),\mathcal{O})$ is
torsion-free.
In particular, the complexes $D_\lambda=b\Rgc(\rX(c),\OC)_{(\lambda)}$ for
$\lambda\in\{q^6,q^7\}$ have no torsion in their cohomology. We have
$$\begin{aligned}
  KD_{q^6} \, \simeq &\ E_6[\theta]_\varepsilon[-7] \oplus
\phi_{7,1}[-13], \\
  KD_{q^7} \, \simeq &\ E_6[\theta]_1[-8] \oplus K[-14]. \\
\end{aligned}$$

Theorem \ref{thm:coxeter} gives the planar-embedded Brauer tree as shown
 in Figure \ref{9E7}.

\subsubsection{$d=10$}
For the $\Phi_{10}$-block the situation is similar: there is a unique proper
$F$-stable subset $I$ of $S$ such that $L_I$ is not an $\ell'$-group. This
Levi subgroup $\bL_I$ has type $D_6$. Since the Coxeter number of $D_6$ is
$10$, \cite[Theorem]{Du4} asserts that $\Hc^*(\rX_{\bL_I}(c_I),\OC)$ is
torsion-free.
Let $E_7[\II]$ be the unipotent cuspidal $KG$-module that appears
with eigenvalue congruent to $q^6$ modulo $\ell$ in $\Hc^7(\rX(c),K)$.
Theorem \ref{thm:coxeter} applied to $C = b \tRgc(\rX(c),k)_{(q^6)}$
gives the planar-embedded Brauer tree as shown
 in Figure \ref{10E7}. 

\subsection{Groups of type $E_8$\label{sec:E8}}

The blocks we need to consider are:
\begin{itemize}

\item the three $\Phi_{9}$-blocks associated to the $d$-cuspidal pairs $(A_2.(q^6+q^3+1), \phi)$ for $\phi = \phi_3, \phi_{21}$ and $\phi_{1^3}$;

\item the $\Phi_{12}$-block associated to the $d$-cuspidal pair $(^3D_4(q).(q^4+q^2+1),{}^3D_4[1])$;

\item the $\Phi_{18}$-block associated to the $d$-cuspidal pair $({}^2 A_2(q).(q^6-q^3+1),\phi_{21})$;

\item the principal $\Phi_d$-blocks for $d= 15$, $20$ and $24$.
In those cases, the proper Levi subgroups of $G$ are $\ell'$-groups.
\end{itemize}

\subsubsection{$d=9$}
There are three unipotent blocks with non-trivial cyclic defect.
The real stem is given by Figure \ref{9E7}, where we have given the
correspondence with vertices of the $E_7$ tree. For each
of the three trees, there are two pairs of complex conjugate characters that need to be located, namely:
\begin{itemize}
  \item[(1)] $\{E_6[\theta]_{\phi_{1,0}},E_6[\theta^2]_{\phi_{1,0}}\}$ and  $\{E_6[\theta]_{\phi_{1,3}''},E_6[\theta^2]_{\phi_{1,3}''}\}$ for the block $b_1$ associated to the $d$-cuspidal pair $(A_2, \phi_3)$;
  
  \item[(2)] $\{E_6[\theta]_{\phi_{2,1}},E_6[\theta^2]_{\phi_{2,1}}\}$ and   $\{E_6[\theta]_{\phi_{2,2}},E_6[\theta^2]_{\phi_{2,2}}\}$ for the block $b_2$ associated to   the $d$-cuspidal pair $(A_2, \phi_{21})$;
  
  \item[(3)] $\{E_6[\theta]_{\phi_{1,6}},E_6[\theta^2]_{\phi_{1,6}}\}$ and   $\{E_6[\theta]_{\phi_{1,3}'},E_6[\theta^2]_{\phi_{1,3}'}\}$ for the block $b_3$ associated to the $d$-cuspidal pair  $(A_2, \phi_{1^3})$.
\end{itemize}
To this end we again use the cohomology of the Coxeter variety $\rX(c)$, which we first show
to be torsion-free on each block $b_i$.
Lemma \ref{le:relativeproj} shows that all three unipotent blocks
are Harish-Chandra projective relative to the principal block of $E_6(q)$.
By Lemma \ref{notorsionE6d9}, the cohomology of the Coxeter variety of $E_6$ is torsion-free. Therefore by Proposition \ref{pr:torsionCoxeter}
the cohomology of $\rX(c)$, cut by the sum of the $b_i$, is torsion-free.

\smallskip

We can now use the same argument as for the principal $\Phi_9$-block of $E_7$:
Theorem \ref{thm:coxeter} shows that the part of the tree to the right of the non-unipotent
node in Figure \ref{9E7} is correct.
We consider the standard Levi subgroup
$L_I$ of semisimple type $E_7$ and we use the Harish-Chandra induction of the
isomorphism $E_6[\theta]_1 \simeq \Omega^7 \OC$ in $\OC L_I$-$\mathrm{mod}$.
 It gives 
$$E_6[\theta]_1 \oplus E_6[\theta]_{\phi_{2,1}}  \oplus 
E_6[\theta]_{\phi_{2,2}} \oplus E_6[\theta]_{\phi_{1,3}'}  \simeq 
\Omega^7 (\OC \oplus \phi_{8,1} \oplus \phi_{35,2} \oplus \phi_{112,3})$$
in $\OC G$-$\mathrm{stab}$. By cutting by each $b_i$ and using the information above (on the part of the tree to the right of the non-unipotent node) we get $E_6[\theta]_{\phi_{2,2}} \simeq \Omega^7 \phi_{35,2}$ and $E_6[\theta]_{\phi_{1,3}'} \simeq \Omega^7 \phi_{112,3}$. 
The same procedure starting with the isomorphism $E_6[\theta]_{\varepsilon} \simeq \Omega^7 \phi_{7,1}$ yields $E_6[\theta]_{\phi_{1,3}''} \simeq \Omega^7 \phi_{160,7}$. This completes the determination of
the three planar-embedded trees.

\smallskip

Note that even though each of these three blocks is Morita equivalent to the principal
$\Phi_9$-block of $E_7$, the Harish-Chandra induction functor (cut by each block)
does not induce that equivalence. 

\subsubsection{$d=12$}
The real stem is as given in Figure \ref{12E8},
therefore knowing the tree amounts to locating the cuspidal character
$E_8[-\theta^2]$.

Let $C = b\tRgc(\rX(c),\mathcal{O})_{(q^6)}$.
The non-cuspidal simple $A$-modules are in the principal series, hence they
cannot occur in the torsion of $\Hc^*(\rX(c),\OC)$. It follows that
Assumption (i) of Theorem \ref{thm:coxeter} is satisfied. Assumption (iii)
follows from the knowledge of the real stem of the tree.
Finally, Assumption (ii) follows from the decomposition
$$ b\Hc^*(\rX(c),K)_{(q^6)} \simeq E_8[-\theta^2] [-8] \oplus \phi_{28,8}[-14].$$ 
Theorem \ref{thm:coxeter} shows that Figure \ref{12E8} gives
the correct planar-embedded Brauer tree.

\subsubsection{$d=18$}
\label{se:E8d18}
The real stem is as in Figure \ref{18E8}. 

\smallskip
$\bullet\ $Step 1: position of $E_8[-\theta^2]$.

The only proper standard $F$-stable
Levi subgroup $\bL_I$ with $\ell \mid |L_I|$ has type $E_7$.
It follows from Proposition \ref{pr:exceptionalcuspidal} that
$bR_{L_I}^G(M)$ is projective for any $M\in\OC L_I\mmod$.
Since $18$ is the Coxeter number of $E_7$, \cite[\S 4.3]{Du4} shows
that the cohomology of the perfect complex
$b\Rgc(\rX(c_I),\OC)$ is torsion-free. It follows from
Proposition \ref{pr:torsionCoxeter} that the torsion of 
$b\Hc^*(\rX(c),\mathcal{O})$ is cuspidal.
Since
$$bK\Hc^*(\rX(c),K)_{(q^7)} = (E_8[-\theta^2])[-8] \oplus \phi_{8,1}[-15],$$
Proposition \ref{prop:assumptionsforcoxeter} and Theorem \ref{thm:coxeter} show
that there is an edge
between $\phi_{35,74}$ and $E_8[-\theta^2]$, and that edge comes
between the edges $\phi_{35,74} \trait \phi_{300,44}$ and
$\phi_{35,74} \trait \phi_{8,91}$ in the cyclic ordering of edges around $\phi_{35,74}$.

\smallskip
$\bullet\ $Step 2: $E_8[\theta]$ is connected to the non-unipotent node.

From the Brauer tree of the principal $\Phi_{18}$-block given in \cite{Cra3}, we know that $\Omega^{12} k$ lifts to an $\mathcal{O}G$-lattice of character $E_6[\theta]_{1}$. Now, if $E_8[\theta]$ is not connected to the non-unipotent node, then $\Omega^{12} \phi_{8,1}$  lifts to an $\mathcal{O}G$-lattice  of character
$D_{4,\phi_{1,12}''}$ or $\phi_{8,91}$ 
 depending on whether $E_8[\theta]$ is connected to the $D_4$-series or the principal series. 
Since the degree of $\phi_{8,1} \otimes E_6[\theta]_1$ is smaller than that of
$\phi_{8,91}$ and of $D_{4,\phi_{1,12}''}$, we obtain a contradiction.
This proves that $E_8[\theta]$ and $E_8[\theta^2]$ are connected to the
non-unipotent node, and we obtain the planar-embedded Brauer tree up to swapping
these two characters (see Figure \ref{18E8}).

\smallskip
$\bullet\ $Step 3: description of $b\Rgc(\overline{\rX}(c),\mathcal{O})_{(q)}$.

Let $C = b\tRgc(\overline{\rX}(c),\mathcal{O})_{(q)}^{\redu}$.
Its cohomology over $K$ is given by
\[ KC \simeq (\phi_{8,1}^{\oplus7}\oplus\phi_{35,2}^{\oplus14}\oplus\phi_{300,8}^{\oplus10}\oplus\phi_{840,13}^{\oplus4})[-2] \oplus (E_8[-\theta])[-8].\]
Corollary \ref{cor:middledegree} (or rather its analogue for compactifications,
which holds since $\ell \nmid \bT^{vF}$ for all $v \leq c$, see also 
Remark \ref{rmk:xbar}) shows
that the terms of $C$ are projective and do not
involve the projective cover of a cuspidal module, except possibly in degree
$8$. The character of $KC$ shows that only the projective cover of
$E_8[-\theta]$ can occur, and it occurs once in degree $8$. In addition, the
torsion of the cohomology of $C$ must be cuspidal by Proposition
\ref{xbartorsion} (there are no modules lying in an $E_7$-series in $b$). 
Let $i<2$ be minimal such that $\H^i(kC)\neq 0$. Then $\H^i(kC)$ is cuspidal
and $(kC)^i$ contains an injective hull of $\H^i(kC)$, a contradiction. So,
$\H^i(kC)=0$ for $i<2$.
Let $P_0,\ldots, P_7$ be the projective indecomposable modules lying in the principal series of $A$, with $[P_0] = \phi_{8,1} + \phi_{35,2}$ and 
$\Hom(P_i,P_{i+1})\neq 0$ for $0\le i<7$, so that
$[P_7] = \phi_{35,74}+\phi_{8,91}$. 
It follows from Lemma \ref{le:twononzerocoh} that
\begin{equation}
\label{eq:xbar18}
C\simeq 0 \to P\to P_2 \to \cdots \to P_6 \to P_{E_8[-\theta]} \to 0
\end{equation}
where $P \simeq P_0^{\oplus 7}\oplus P_1^{\oplus 7}\oplus P_2^{\oplus 4}$ is
in degree $2$.

\smallskip
$\bullet\ $Step 4: the torsion part of 
$b\Rgc(\overline{\rX}(w),\mathcal{O})_{(q)}$ is cuspidal.

Let $w\in W$ be the unique (up to conjugation) element of minimal length for
which $E_8[\theta]$ occurs in $\Hc^*(\rX(w))$. Here $\ell(w) =14$.
Let us consider $R =  b\Rgc(\overline{\rX}(w),\mathcal{O})_{(q)}$.
Using the trace formula (see \cite[Corollaire 3.3.8]{DMR}), we find that
$$KR \simeq (\phi_{8,1}^{\oplus4}\oplus\phi_{35,2}^{\oplus6}\oplus
\phi_{300,8}^{\oplus3}\oplus\phi_{840,13})[-2] \oplus (E_8[\theta^2])[-14].$$
By Proposition
\ref{xbartorsion}, the torsion-part of the cohomology in $R$ is either
cuspidal or in an $E_7$-series. Since there are no modules in $E_7$-series
in $A$, we deduce that the torsion part is cuspidal. In particular, if $j = 4,5,6$ then 
$\Hom^\bullet_{kG}(P_j,kR)\simeq 0$, and if $j = 0,\ldots,3$ the cohomology of 
$\Hom^\bullet_{kG}(P_j,kR)$ vanishes outside degree $2$. Note that
$\Hom^\bullet_{kG}(P_j,kR)\simeq P_j\otimes_{kG}kR$ where $P_j$ is viewed
as a right $kG$-module via the anti-automorphism $g\mapsto g^{-1}$ of $G$, since
$P_j$ is self-dual.

\smallskip
$\bullet\ $Step 5: $E_8[-\theta^2]$ is not a composition factor of $\H^*(kR)$.

Let $C'$ be the cone of 
the canonical map $P_{E_8[-\theta]}[-8] \to kC$. By
\eqref{eq:xbar18} it is homotopic to a complex involving only projective
modules in the principal series. Tensoring by $kR$ gives a distinguished
triangle
$$ P_{E_8[-\theta]}[-8] \otimes_{kG} kR \to kC \otimes_{kG} kR \to C' \otimes_{kG} kR \rightsquigarrow.$$ 
From the explicit representative of $C'$ and step 4 above we know that the cohomology of $C' \otimes_{kG} kR$ vanishes outside the degrees $4$, $5$ and $6$. 
Proposition \ref{xbarxbar} shows that the cohomology of $kC \otimes_{kG} kR$
vanishes outside degree $4$. The previous distinguished triangle shows that
the cohomology of $P_{E_8[-\theta]} \otimes_{kG} kR$ vanishes outside the
degrees $-4,\ldots,-1$. Since $\H^i(kR) = 0$ for $i < 0$ this proves that
$E_8[-\theta^2]$ is not a composition factor of $\H^*(kR)$.

\smallskip
$\bullet\ $Step 6: $E_8[-\theta]$ is not a composition factor of $\H^i(kR)$ for $i \neq 6,7,8$. 

The same method as in Step 5
with $D= b\Rgc(\overline{\rX}(c),\OC)_{(q^{7})}\simeq C^*[-16]$ and $D' =
\mathrm{Cone}(kD \to P_{E_8[-\theta^2]}[-8])$ yields a distinguished triangle
$$ kD \otimes_{kG} kR \to P_{E_8[-\theta^2]}[-8] \otimes_{kG} kR \to
D' \otimes_{kG} kR \rightsquigarrow.$$ 
Here $kD \otimes_{kG} R$ has non-zero cohomology only in degree $16$ by
Proposition \ref{xbarxbar}. As for $D' \otimes_{kG} kR$, its cohomology
vanishes outside the degrees $14$, $15$ and $16$, and we deduce from the distinguished triangle that $E_8[-\theta]$ is not a composition factor of $\H^i(kR)$ for $i \neq 6,7,8$.

 If $v<w$ and
$\ell\mid |\bT^{vF}|$, then $v$ is conjugate to $c_I$, the Coxeter element
of type $E_7$. Since $b\Rgc(\rX(c_I),\OC)$ is perfect, it
follows from Lemma \ref{le:vanishingonminimal} that 
$b\Rg_c(\rX(v),\OC)$ is perfect for all $v<w$.

\smallskip
$\bullet\ $Step 7: given $v<w$, the complex $b\Rgc(\rX(v),k)$ is
quasi-isomorphic to a bounded complex of projective modules whose 
indecomposable summands correspond to edges that do not contain the non-unipotent
vertex.

Consider $v<w$. If $v$ is not conjugate to a Coxeter element $c_I$ of
$E_7$, then $\ell\nmid |\bT^{vF}|$ and $b\Rgc(\rX(v),k)$ is perfect
and quasi-isomorphic to a bounded complex of projective modules whose 
indecomposable summands correspond to edges that do not contain the non-unipotent
vertex by Lemma \ref{le:minwforexc}.
If $v=c_I$, the perfectness has been shown in Step 1 and the second
part holds, because the edges that contain the
non-unipotent vertex are cuspidal. We deduce that the statement of Step 7
holds for all $v<w$ by Lemma \ref{le:vanishingonminimal}.

\smallskip
$\bullet\ $Step 8: $\H^{>14}(R)=0$.

Steps 4, 5 and 6 show that the composition factors of $\H^i(kR)$ for $i > 14$
are
cuspidal modules $M$ corresponding to an edge containing the non-unipotent vertex.
Let $M$ be a simple module corresponding to an edge containing the 
non-unipotent vertex.
Step 7 shows that the canonical map $\Rg(\overline{\rX}(w),k)\to \Rg(\rX(w),k)$
induces an isomorphism
$$ \RHom^\bullet_{kG}(\Rg(\rX(w),k),M) \simeq
\RHom^\bullet_{kG}(\Rg(\overline{\rX}(w),k),M).$$
Let $M_0 = \H^{i_0}(kR)$ be the non-zero cohomology group of $kR$ of largest
degree. We have $\mathrm{Hom}_{D^b(kG)}(\Rg(\overline{\rX}(w),k),M_0[-i_0]) \neq 0$. Since $\Rg(\rX(w),k)$ has a representative with terms in degrees $0,\ldots,\ell(w)=14$, we deduce from the previous isomorphism that $i_0 \leq 14$.

\smallskip
$\bullet\ $Step 9: $E_8[\theta]$ and $E_8[\theta^2]$ do not
occur as composition factors of the torsion part of $\H^*(R)$ and
$E_8[\theta^2]$ is a direct summand of $\H^{14}(kR)$.

Step 7 shows that $E_8[\theta]$ and $E_8[\theta^2]$ are not
composition factors of $\Hc^*(\rX(v),k)$ for $v<w$.
It follows that if $M$ is any of the simple modules $E_8[\theta]$
or $E_8[\theta^2]$, then the canonical map
$\Hc^*(\rX(w),k)\to\Hc^*(\overline{\rX}(w),k)$ induces an isomorphism
\begin{equation}\label{eq:step9}
\Hom_{kG}\big(P_M, \Hc^i(\rX(w),k)\big) \simto \Hom_{kG}\big(P_M,\Hc^i(\overline{\rX}(w),k)\big).
\end{equation}
Since $\Hc^i(\rX(w),k)=0$ for $i<14$, we deduce that  $E_8[\theta]$ and 
$E_8[\theta^2]$ do not occur as composition factors of
$\H_c^i(\overline{\rX}(w),k)$ for $i<14$. By Poincar\'e duality and 
the isomorphism \eqref{eq:step9}, it follows that $E_8[\theta]$ and $E_8[\theta^2]$ cannot
occur as composition factors of $\Hc^i(\rX(w),k)$ for $i>14$. On the
other hand, $E_8[\theta]$ does not occur in $[K\Rgc(\rX(w),\OC)_{(q)}]$
(nor does $\chi_{\mathrm{exc}}$), hence 
$E_8[\theta]$ does not occur as a composition factor of
$\Hc^{14}(\rX(w),k)_{(q)}$ or $\Hc^{14}(\overline{\rX}(w),k)_{(q)}$. Similarly,
$E_8[\theta^2]$ occurs with multiplicity $1$ as a composition factor of
$\Hc^{14}(\rX(w),k)_{(q)}$ and of $\Hc^{14}(\overline{\rX}(w),k)_{(q)}$.
Proposition \ref{prop:middledegree} shows that $E_8[\theta^2]$ is actually
a submodule of $\Hc^{14}(\rX(w),k)_{(q)}$ and hence of
$\Hc^{14}(\overline{\rX}(w),k)_{(q)}$ by \eqref{eq:step9}. Since 
$K\Hc^{14}(\overline{\rX}(w),K)_{(q)}=E_8[\theta^2]$, it follows that
$E_8[\theta^2]$ is a quotient of $\Hc^{14}(\overline{\rX}(w),k)_{(q)}$, hence
it is a direct summand.

\smallskip
$\bullet\ $Step 10: $(E_8[\theta^2])[-14]$ is a direct summand of
$kR$ in $D^b(kG)$.

Let $Z$ be the cone of the canonical map $\Rgc(\rX(w),k)_{(q)}\to
\Rgc(\overline{\rX}(w),k)_{(q)}$. Step 7 shows that $Z$ can be chosen (up
to isomorphism in $D^b(kG)$) to
be a bounded complex of projective modules that do not involve
edges containing the non-unipotent vertex.
The complex $kR$ is quasi-isomorphic
to the cone $D'$ of a map $Z[-1]\to b\Rgc(\rX(w),k)_{(q)}$, hence to the truncation
$\tau^{\le 14}(D')$, a complex $N$ with $N^i=0$ for $i<0$ and $i>14$ and
with $N^i$ a direct sum of projective modules corresponding to edges 
that do not contain the non-unipotent vertex
for $i\le 13$. Note that $E_8[\theta]$ and $E_8[\theta^2]$ are
not composition factors of $N^{13}$, hence $E_8[\theta]$ is not
a composition factor of $N^{14}$ while $E_8[\theta^2]$ is a
composition factor of $N^{14}$ with multiplicity $1$ (see Step 9 above).
Consider a non-zero
morphism $P_{E_8[\theta^2]}\to N^{14}$ and let $U$ be its image. Since
$E_8[\theta^2]$ is a direct summand of $\H^{14}(N)$, it follows that the
image of $U$ in $\H^{14}(N)$ is $E_8[\theta^2]$.
On the other hand, the simple modules corresponding to edges containing the non-unipotent
vertex but not $E_8[\theta]$ nor $E_8[\theta^2]$ are not quotients of $N^{13}$, hence
$U\simeq E_8[\theta^2]$ embeds in $\H^{14}(N)$.
It follows that $U[-14]$ is a direct summand of $N$.


\smallskip
$\bullet\ $Step 11: $\Rgc(\overline{\rX}(w),\OC)_{(q^{-2})}\simeq V[-14]$,
where $V$ is an $\OC G$-lattice such that the simple factors of 
$KV$ are in the $D_4$-series.

Let $R'=\Rgc(\overline{\rX}(w),\OC)_{(q^{-2})}$.
We have $K\H^i(R')=0$ for $i\neq 14$ and the simple
factors of $K\H^{14}(R')$ are in the $D_4$-series.
As in Step 4, one shows that
the torsion of $\Rgc(\overline{\rX}(w),\OC)_{(q^{-2})}$ is cuspidal.
We show as in Steps 5 and 6 that $E_8[-\theta]$ and $E_8[-\theta^2]$
are not composition factors of $\H^*(kR')$. Furthermore, 
$\H^{>14}(R')=0$ as in Step 8. Proceeding as in Step 8, one sees
that the canonical map $\Rgc(\rX(w),k)\to\Rgc(\overline{\rX}(w),k)$
induces an isomorphism
$$\RHom^\bullet_{kG}(M,\Rgc(\rX(w),k))\simeq
\RHom^\bullet_{kG}(M,\Rgc(\overline{\rX}(w),k)).$$
Let $i_0$ be minimal such that $\H^{i_0}(kR')\neq 0$, and suppose that $i_0<14$. Then $\H^{i_0}(kR')$ is cuspidal and
$$\Hom_{D^b(kG)}(\H^{i_0}(kR'),\Rgc(\rX(w),k)[i_0])\simeq
\RHom_{D^b(kG)}(\H^{i_0}(kR'),\Rgc(\overline{\rX}(w),k)[i_0])\neq 0.$$
This contradicts the fact that $\Rgc(\rX(w),k)$ has no cohomology in
degrees less than $14$. Thus, $\H^i(kR')=0$ for $i\neq 14$.

\smallskip
$\bullet\ $Step 12: conclusion.

Lemma \ref{le:shiftFrob} shows that $\Rgc(\rX(w),k)_{(q^{-2})}\simeq \Rgc(\rX(w),k)_{(q)}[6]$ in $kG\mstab$. By Step 11, we deduce that
$kV$ has a direct summand isomorphic to $\Omega^{-6}(E_8[\theta^2])$ in
$kG\mstab$. If the Brauer tree is not the one given in Figure \ref{18E8}
(i.e., $E_8[\theta]$ and $E_8[\theta^2]$ need to be swapped), then
$\Omega^{-6}(E_8[\theta^2])$ is the reduction of a lattice in
$\phi_{300,44}$, which cannot be a direct summand of $kV$.
We deduce that the planar-embedded tree is as shown in Figure \ref{18E8}. 

\begin{rmk}
\label{re:d=18ford=15}
We use the determination of the tree to obtain a character-theoretic
statement that will be needed in the study of the case $d=15$.

The Brauer tree of the principal $\Phi_{18}$-block of $G$ is given in \cite[Remark 3.11]{DR}.
In particular, $E_6[\theta^2]_1 \simeq \Omega^{24}k$. Since 
$\Omega^{24}\phi_{8,1}\simeq\ E_8[\theta^2]$, we deduce that
 $\phi_{8,1} \otimes E_6[\theta^2]_1$ is isomorphic to $E_8[\theta^2]$ plus
a projective $\OC G$-module $P$. If $E_8[\theta]$ occurs in the 
character of $P$, then the non-unipotent vertex occurs as well.
As the degree of the non-unipotent vertex is larger than the degree
of $\phi_{8,1} \otimes E_6[\theta^2]_1$, we obtain a contradiction.
So, the character of $E_8[\theta]$ is not a constituent of 
$\phi_{8,1} \otimes E_6[\theta^2]_1$.
\end{rmk}

\subsubsection{$d=15$}
The real stem is known and comprises the principal series characters in the principal $\ell$-block.
A (Hecke) argument also gives the two subtrees consisting of characters in the $E_6[\theta]$-series
and the $E_6[\theta^2]$-series as shown in Figure \ref{15E8}.

\smallskip
Except for the two characters $E_8[\theta]$ and $E_8[\theta^2]$, each Harish-Chandra series meeting the principal $\Phi_{15}$-block has a character which 
appears in the cohomology of the Coxeter variety. The generalized $(\lambda)$-eigenspaces on the cohomology of the Coxeter variety are given by
$$\begin{aligned}
  b\Hc^*(\rX(c),K)_{(q^8)} \, \simeq &\ E_6[\theta]_\varepsilon[-8] \oplus K[-16], \\
  b\Hc^*(\rX(c),K)_{(q^{10})} \, \simeq &\ (E_8[\zeta^2])[-8] \oplus E_6[\theta]_1[-10], \\
  b\Hc^*(\rX(c),K)_{(q^7)} \, \simeq &\ (E_8[\zeta])[-8] \oplus \phi_{8,1}[-15]. \\
\end{aligned}$$

Corollary \ref{co:Coxetertriv} applied to $\lambda=q^8$
shows that there is an edge between $\St$ and
$E_6[\theta]_\varepsilon$, and this edge comes between the one containing
$\phi_{84,64}$ and $\St_\ell$ in the cyclic ordering of edges around $\St$.

 Only cuspidal characters remain to be located, and since they have a larger degree than $E_6[\theta]_1$, we deduce that $E_6[\theta]_1$
 remains irreducible modulo $\ell$.

From Proposition \ref{prop:assumptionsforcoxeter} and  Theorem \ref{thm:coxeter} applied to $C= b\Rgc(\rX(c),\mathcal{O})_{(q^{10})}$, we deduce that there is
an edge between $E_8[\zeta^2]$ and $E_6[\theta]_{\varepsilon}$.

Similarly, using $C= b\Rgc(\rX(c),\mathcal{O})_{(q^{7})}$, we deduce that
 there is an edge between $\phi_{112,63}$ and
$E_8[\zeta]$, and this edge comes between the one containing
$\phi_{1400,37}$ and the one containing $\phi_{8,91}$
in the cyclic ordering of edges around $\phi_{112,63}$.

 Consequently, the trees in Figures \ref{subtree15right} and \ref{subtree15left} are subtrees of the
 Brauer tree $T$ (although as of yet we cannot fix the planar embedding around $E_6[\theta]_\varepsilon$).

\smallskip

We claim that $E_8[\theta]$ and $E_8[\theta^2]$ are not connected to the subtree shown in Figure \ref{subtree15right}. Let us assume otherwise. By a (Parity) argument, they are not connected to the non-unipotent node. Let $w\in W$ be an element of minimal length such that $E_8[\theta^2]$ appears in the cohomology of $\rX(w)$ (we have $\ell(w) = 14$). We have 
$$[b\Hc^*(\rX(w),K)_{q^2}]=[E_8[\theta^2]] + [\phi_{8,91}]=
([\phi_{8,91}] + [\chi_{\mathrm{exc}}])-([\St]+[\chi_{\mathrm{exc}}])+
\eta,$$
where $\eta=[KP]$ and $P$ is a projective $b\OC G$-module whose character does
not involve
$\chi_{\mathrm{exc}}$. Therefore there is
an odd integer $i$ such that 
$\Hom_{D^b(kG)}(\Rgc(\rX(w),k),\St_\ell[i])\not=0$. Since $i\neq 14$,
it follows from
Proposition \ref{prop:middledegree} that $\langle [\Hc^*(\rX(v),k)],\St_\ell
\rangle\neq 0$ for some $v<w$. 
One easily checks on the character of $\Hc^*(\rX(v),K)$ that this is impossible.
As a consequence, $E_8[\theta]$ and $E_8[\theta^2]$ are connected to the subtree shown in Figure \ref{subtree15left}.

\smallskip

We next return to the planar embedding of the edges around the node
$E_6[\theta]_\varepsilon$. If the embedding is not as in 
Figure \ref{subtree15right} then $\Omega^{-13}k$ would lift to an
$\OC G$-lattice of character $E_6[\theta^2]_1$, and as $\Omega^{30}k$ lifts to
$\phi_{8,1}$, we get that $\Omega^{30}k\otimes \Omega^{-13}k$ lifts to an
$\OC G$-lattice with character the sum of the non-unipotent character
plus a projective $\OC G$-module. The sum of the degrees of the
non-unipotent characters is $(q^6-1)(q^8-1)(q^{10}-1)(q^{12}-1)(q^{14}-1)(q^{18}-1)(q^{20}-1)(q^{24}-1)\Phi_{30}$. Since this is larger than the degree
of $E_6[\theta^2]_1\otimes \phi_{8,1}$, we deduce that $\Omega^{-13}k$ does not
lift to $E_6[\theta^2]_1$,
and we obtain the planar-embedded Brauer tree as in Figure \ref{subtree15right}.

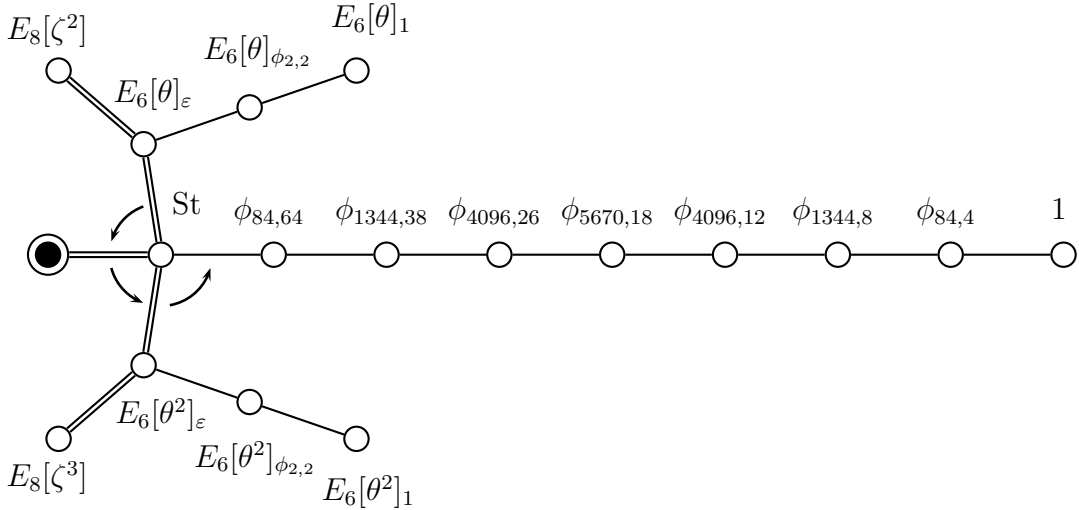
\begin{figure}[h!]
\begin{center}
\begin{pspicture}(13.2,6)

  \cnode[fillstyle=solid,fillcolor=black](0,3){5pt}{A2}
    \cnode(0,3){8pt}{A}
  \cnode(1.5,3){5pt}{B}
  \cnode(1.27,4.47){5pt}{X1}
  \cnode(0.14,5.45){5pt}{X2}
  \cnode(2.68,4.96){5pt}{X3}
  \cnode(4.10,5.45){5pt}{X4}
    \cnode(1.27,1.53){5pt}{Y1}
  \cnode(0.14,0.55){5pt}{Y2}
  \cnode(2.68,1.04){5pt}{Y3}
  \cnode(4.10,0.55){5pt}{Y4}
  \cnode(3,3){5pt}{C}
  \cnode(4.5,3){5pt}{D}
  \cnode(6,3){5pt}{E}
  \cnode(7.5,3){5pt}{F}
  \cnode(9,3){5pt}{G}
  \cnode(10.5,3){5pt}{H}
  \cnode(12,3){5pt}{K} 
  \cnode(13.5,3){5pt}{Q}

   \ncline[nodesep=0pt,doubleline=true]{A}{B}\naput[npos=1.5]{$\vphantom{\Big(} \mathrm{St}$}
  \ncline[nodesep=0pt]{B}{C}\naput[npos=1.1]{$\vphantom{\Big(}\phi_{84,64}$}
  \ncline[nodesep=0pt]{C}{D}\naput[npos=1.1]{$\vphantom{\Big(}  \phi_{1344,38}$}
  \ncline[nodesep=0pt]{D}{E}\naput[npos=1.1]{$\vphantom{\Big(}\phi_{4096,26}$}
  \ncline[nodesep=0pt]{E}{F}\naput[npos=1.1]{$\vphantom{\Big(} \phi_{5670,18}$}
  \ncline[nodesep=0pt]{F}{G}\naput[npos=1.1]{$\vphantom{\Big(} \phi_{4096,12}$}
  \ncline[nodesep=0pt]{G}{H}\naput[npos=1.1]{$\vphantom{\Big(} \phi_{1344,8}$}
  \ncline[nodesep=0pt]{H}{K}\naput[npos=1.1]{$\vphantom{\Big(} \phi_{84,4} $}
  \ncline[nodesep=0pt]{K}{Q}\naput[npos=1.1]{$\vphantom{\Big(} 1 $}
    \ncline[nodesep=0pt,doubleline=true]{B}{X1}\ncput[npos=1.75]{$\phantom{\Big(a} E_6[\theta]_\varepsilon $}
      \ncline[nodesep=0pt,doubleline=true]{B}{Y1}\ncput[npos=1.75]{$\phantom{\Big(aa} E_6[\theta^2]_\varepsilon $}
      \ncline[nodesep=0pt,doubleline=true]{X1}{X2}\ncput[npos=1.85]{$\phantom{\Big(aaa} E_8[\zeta^2]$}
      \ncline[nodesep=0pt,doubleline=true]{Y1}{Y2}\ncput[npos=1.85]{$\phantom{\Big(aaa} E_8[\zeta^3]$}
      \ncline[nodesep=0pt]{X1}{X3}\naput[npos=1.6]{$ E_6[\theta]_{\phi_{2,2}} $}
      \ncline[nodesep=0pt]{Y1}{Y3}\nbput[npos=1.6]{$ E_6[\theta^2]_{\phi_{2,2}}$}
      \ncline[nodesep=0pt]{X3}{X4}\naput[npos=1.6]{$E_6[\theta]_1$}
      \ncline[nodesep=0pt]{Y3}{Y4}\nbput[npos=1.6]{$E_6[\theta^2]_1$}
     \psellipticarc[linewidth=1pt]{->}(1.5,3)(0.7,0.7){110}{165}
  \psellipticarc[linewidth=1pt]{->}(1.5,3)(0.7,0.7){195}{250}
  \psellipticarc[linewidth=1pt]{->}(1.5,3)(0.7,0.7){280}{345}
   
\end{pspicture}
\end{center}
\caption{Subtree of the principal $\Phi_{15}$-block of $E_8(q)$}
\label{subtree15right}
\end{figure}

\begin{figure}[h!]
\begin{center}
\begin{pspicture}(13.5,4.5)

  \cnode[fillstyle=solid,fillcolor=black](13.5,2){5pt}{A2}
    \cnode(13.5,2){8pt}{A}  
    \cnode(7.5,2){5pt}{L}
  \cnode(9,2){5pt}{M}
  \cnode(10.5,2){5pt}{N}
  \cnode(12,2){5pt}{O}
  \cnode(6,2){5pt}{P}
    \cnode(4.5,2){5pt}{V}
     \cnode(3,2){5pt}{W}
  \cnode(1.5,2){5pt}{X}
  \cnode(0,2){5pt}{Y}
  \cnode(10.5,3.5){5pt}{N1}  
  \cnode(10.5,0.5){5pt}{N2}

\ncline[nodesep=0pt]{L}{M}\naput[npos=-0.1]{$\vphantom{\Big(} \phi_{4096,27} $}  
  \ncline[nodesep=0pt]{M}{N}\naput[npos=-0.1]{$\vphantom{\Big(} \phi_{1400,37}$}
  \ncline[nodesep=0pt]{N}{O}\naput[npos=0.4]{$\vphantom{\Big(}\phi_{112,63} $}
  \ncline[nodesep=0pt,doubleline=true]{O}{A}\naput[npos=0.3]{$\vphantom{\Big(}\phi_{8,91}  $}
  \ncline[nodesep=0pt]{P}{L}\naput[npos=-0.1]{$\vphantom{\Big(}\phi_{5600,19} $}
    \ncline[nodesep=0pt]{V}{P}\naput[npos=-0.1]{$\vphantom{\Big(}\phi_{4096,11} $}
  \ncline[nodesep=0pt]{W}{V}\naput[npos=-0.1]{$\vphantom{\Big(}\phi_{1400,7} $}
  \ncline[nodesep=0pt]{X}{W}\naput[npos=-0.1]{$\vphantom{\Big(}\phi_{112,3} $}
  \ncline[nodesep=0pt]{Y}{X}\naput[npos=-0.1]{$\vphantom{\Big(}\phi_{8,1} $}
      \ncline[nodesep=0pt,doubleline=true]{N}{N1}\ncput[npos=1.65]{$\vphantom{\Big(} E_8[\zeta^4] $}
  \ncline[nodesep=0pt,doubleline=true]{N}{N2}\ncput[npos=1.65]{$\vphantom{\Big(} E_8[\zeta] $}
    \psellipticarc[linewidth=1pt]{->}(10.5,2)(0.7,0.7){105}{165}
  \psellipticarc[linewidth=1pt]{->}(10.5,2)(0.7,0.7){195}{255}
  \psellipticarc[linewidth=1pt]{->}(10.5,2)(0.7,0.7){285}{345}
        
\end{pspicture}
\end{center}
\caption{Subtree of the principal $\Phi_{15}$-block of $E_8(q)$}
\label{subtree15left}
\end{figure}
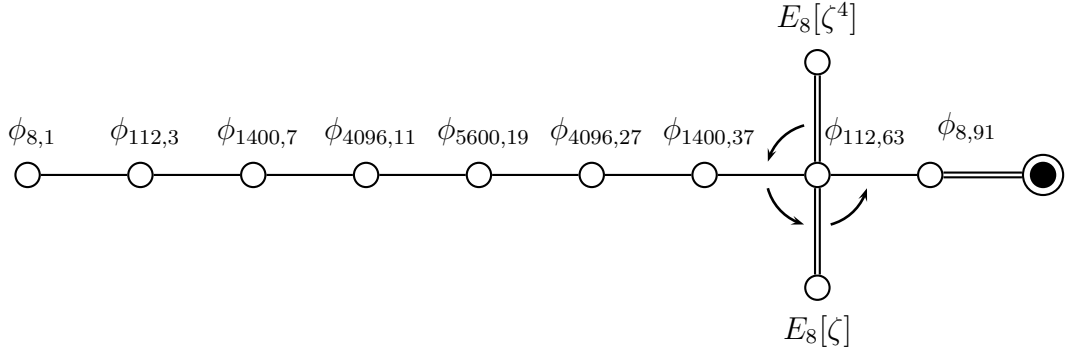
 
\smallskip

It remains to locate $E_8[\theta]$ and $E_8[\theta^2]$. If they were not connected to $\phi_{8,91}$, then  $\Omega^{19} k $ would lift to an $\mathcal{O}G$-lattice of character $\phi_{112,63}$, although $\Omega^{30} k \otimes  \Omega^{-11} k$
lifts to a lattice of character $\phi_{8,1} \otimes E_6[\theta^2]_1$ plus
a projective module. Since
that tensor product has a degree smaller than $\phi_{112,63}$, we obtain
a contradiction.
Consequently, we obtain the planar-embedded Brauer tree given in Figure \ref{15E8}, up to swapping $E_8[\theta]$ and $E_8[\theta^2]$. 
Assume the planar embedded tree shown in Figure \ref{15E8} is not correct.
Then $\Omega^{19} k$ lifts to a lattice of character $E_8[\theta]$. Since
$\Omega^{30} k \otimes  \Omega^{-11} k$ lifts to a lattice of character
 $\phi_{8,1} \otimes E_6[\theta^2]_1$, we deduce that $E_8[\theta]$
occurs as a constituent of that tensor product, contradicting Remark \ref{re:d=18ford=15}. Consequently, the tree in Figure \ref{15E8} is
correct.

\subsubsection{$d=20$}\label{se:E8d20}
The real stem of the tree is easily determined
(see Figure \ref{20E8}).
 The difficult part is to locate the six cuspidal
characters in the block.

We have
$$\begin{aligned}
  b\Hc^*(\overline{\rX}(c),K)_{(q^8)} &\, \simeq (E_8[\zeta])[-8] \oplus K[-16], \\
  b\Hc^*(\overline{\rX}(c),K)_{(q^{16})} &\, \simeq (E_8[\zeta^3])[-8] \oplus D_{4,1}[-12]. \\
\end{aligned}$$

Proposition \ref{prop:resultsxbar} and Corollary \ref{co:Coxetertriv}
show that there is
an edge connecting $E_8[\zeta]$ to $\St$ and that this edge comes between
the edge containing $\phi_{112,63}$ and the one containing
$\St_\ell$ in the cyclic
ordering of edges containing $\St$. Also, there is no cuspidal edge connected
to a principal series character other than $\St$ and we have
\begin{equation}
\label{eq:complexxbar}
b\Rgc(\overline{\rX}(c),k)_{(q^8)} \simeq
0 \to P_{E_8[\zeta]} \to P_7 \to \cdots  \to P_0 \to 0,
\end{equation}
where $P_0$ is in degree $16$ and $P_0,\ldots, P_7$ are projective
indecomposable modules labelling the principal series edges from $1$ to $\St$.

Proposition \ref{prop:resultsxbar} and Theorem \ref{thm:coxeter} show that
there is an edge connecting $E_8[\zeta^3]$ and $D_{4,\varepsilon}$.

\smallskip

We now want to locate the characters $E_8[\II]$ and $E_8[-\II]$.
A (Parity) argument shows that they are not connected to the non-unipotent
vertex.
 The smallest Deligne--Lusztig variety in which they appear is associated to a $24$-regular element $w$ of length $10$.
Note that $\ell \nmid |\bT^{vF}|$ for all $v\le w$.
In particular, the character $\eta=[b\Hc^*(\rX(w),K)_{(1)}]=[\St]+[E_8[-\II]]$ is
virtually projective.
It follows from Lemma \ref{le:minwforexc} that
$\chi_\mathrm{exc} + D_{4,\varepsilon}$ does not occur in the
decomposition of $\eta$ in the basis of projective indecomposable modules.
As a consequence, $E_8[-\II]$ is not connected by an edge to the $D_4$-series,
hence $E_8[\II]$ and $E_8[-\II]$ are connected to the
Steinberg character.

\smallskip

We are therefore left with determining the planar embedding around 
$D_{4,\varepsilon}$ and $\mathrm{St}$. Assume that we are in the case shown in Figure \ref{subtree20D4}. 
\begin{figure}[h!]

\begin{center} 
\begin{pspicture}(5,4)
\cnode[fillstyle=solid,fillcolor=black](4.5,2){5pt}{A2}
    \cnode(4.5,2){8pt}{A}
  \cnode(3,2){5pt}{B}
  \cnode(1.5,2){5pt}{C}
  \cnode(0.5,2){0pt}{D}
    \cnode(3,0.5){5pt}{I}
  \cnode(3,3.5){5pt}{J}

   \ncline[nodesep=0pt,doubleline=true]{A}{B}\nbput[npos=0.6]{$\vphantom{\Big(} D_{4,\varepsilon}$}
   \ncline[nodesep=0pt]{B}{C}\nbput[npos=1.15]{$\vphantom{\Big(} D_{4,\phi_{9,10}}$}
  \ncline[linestyle=dashed,nodesep=0pt]{C}{D}
  \ncline[nodesep=0pt,doubleline=true]{B}{I}\ncput[npos=1.65]{$\vphantom{\Big(} E_8[\zeta^2] $}
  \ncline[nodesep=0pt,doubleline=true]{B}{J}\ncput[npos=1.65]{$\vphantom{\Big(} E_8[\zeta^3] $}
     \psellipticarc[linewidth=1pt]{->}(3,2)(0.7,0.7){105}{165}
  \psellipticarc[linewidth=1pt]{->}(3,2)(0.7,0.7){195}{255}
  \psellipticarc[linewidth=1pt]{->}(3,2)(0.7,0.7){285}{345}
     
\end{pspicture}
\end{center}
\caption{Wrong planar embedding for the principal $\Phi_{20}$-block of $E_8(q)$}
\label{subtree20D4}
\end{figure}
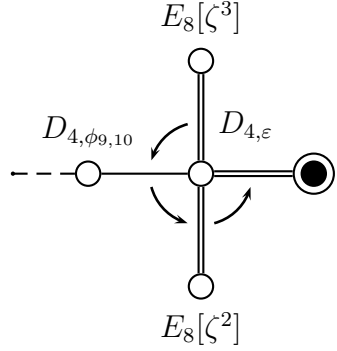
Let $S_0,\ldots,S_4$ be the simple modules labelling the edges
from $D_{4,1}$ to the non-unipotent node so that $[P_{S_4}] = \chi_{\mathrm{exc}} + D_{4,\varepsilon}$. A minimal
representative of $b\tRgc(\overline{\rX}(c),k)_{(q^{16})}$ is given by
$$D= 0 \longrightarrow \begin{array}{c} \textcolor{violet}{E_8[\zeta^3]} \\ \textcolor{violet}{S_3} \\ \fbox{\parbox{1cm}{$E_8[\zeta^2]$ \\ \phantom{i} $S_4$ \\ $E_8[\zeta^3]$}}  \end{array} \longrightarrow \begin{array}{c} \textcolor{purple}{S_3} \\ \fbox{\parbox[b]{1cm}{$E_8[\zeta^2]$  \\ \phantom{i} $S_4$ }} \ \ \textcolor{purple}{S_2} \\ \textcolor{violet}{E_8[\zeta^3]} \hphantom{AA}\\ \textcolor{violet}{S_3} \end{array}  \longrightarrow \begin{array}{c} \textcolor{violet}{S_2} \\\textcolor{purple}{S_3} \quad \textcolor{violet}{S_1}  \\ \textcolor{purple}{S_2} \end{array} \longrightarrow \begin{array}{c} \textcolor{purple}{S_1} \\ \textcolor{violet}{S_2} \quad \textcolor{purple}{S_0}  \\ \textcolor{violet}{S_1} \end{array}  \longrightarrow  \begin{array}{c} \fbox{$S_0$} \\ \textcolor{purple}{S_1} \\  \textcolor{purple}{S_0} \end{array} \longrightarrow 0,$$
where the cohomology groups (represented by the boxes) are non-zero in degrees $8$, $9$ and $12$ only.
A non-zero map 
$P_{E_8[\zeta^2]} \to P_{E_8[\zeta^3]}$ gives a non-zero element of
$\Hom_{D^b(kG)}(D^*[-16],D)$. Consequently,
$\H^{16}(D\otimes_{kG}D)\neq 0$.
We have
$$b\Rgc(\overline{\rX}(c),K)_{(q^{16})} \otimes_{KG}
b\Rgc(\overline{\rX}(c),K)_{(q^{16})} \simeq K[-24],$$
 and Proposition \ref{xbarxbar} shows that
the cohomology of $b\Rgc(\overline{\rX}(c),\mathcal{O})_{(q^{16})} 
\otimes_{\mathcal{O}G} b\Rgc(\overline{\rX}(c),\mathcal{O})_{(q^{16})}$ is
torsion-free, hence
$\H^i(b\Rgc(\overline{\rX}(c),k)_{(q^{16})} 
\otimes_{kG} b\Rgc(\overline{\rX}(c),k)_{(q^{16})})=0$ for $i\neq 24$: this
gives a contradiction.

\smallskip

We now turn to the four possibilities for the planar embedding around the node labelled by the Steinberg character. We need to rule out the three of them shown in Figure \ref{subtree20st}.
\begin{figure}[h!]

\begin{center} 
\begin{pspicture}(12.5,3.7)
\cnode[fillstyle=solid,fillcolor=black](0,2){5pt}{A2}
    \cnode(0,2){8pt}{A}
  \cnode(1.5,2){5pt}{B}
  \cnode(2.5,2){0pt}{C}
    \cnode(2.3,0.6){5pt}{I}
  \cnode(2.3,3.2){5pt}{J}
    \cnode(0.7,0.6){5pt}{R}
  \cnode(0.7,3.2){5pt}{S}
 
\cnode[fillstyle=solid,fillcolor=black](5,2){5pt}{bA2}
    \cnode(5,2){8pt}{bA}
  \cnode(6.5,2){5pt}{bB}
  \cnode(7.5,2){0pt}{bC}
    \cnode(7.3,0.6){5pt}{bI}
  \cnode(7.3,3.2){5pt}{bJ}
    \cnode(5.7,0.6){5pt}{bR}
  \cnode(5.7,3.2){5pt}{bS} 

\cnode[fillstyle=solid,fillcolor=black](10,2){5pt}{cA2}
    \cnode(10,2){8pt}{cA}
  \cnode(11.5,2){5pt}{cB}
  \cnode(12.5,2){0pt}{cC}
    \cnode(12.3,0.6){5pt}{cI}
  \cnode(12.3,3.2){5pt}{cJ}
    \cnode(10.7,0.6){5pt}{cR}
  \cnode(10.7,3.2){5pt}{cS} 
     
   \ncline[nodesep=0pt,doubleline=true]{A}{B}\naput[npos=1.15]{$\vphantom{\Big(} \mathrm{St}$}
  \ncline[linestyle=dashed,nodesep=0pt]{B}{C}
  \ncline[nodesep=0pt,doubleline=true]{B}{I}\ncput[npos=1.65]{$\vphantom{\Big(} E_8[\zeta^4] $}
  \ncline[nodesep=0pt,doubleline=true]{B}{J}\ncput[npos=1.65]{$\vphantom{\Big(} E_8[\zeta] $}
  \ncline[nodesep=0pt,doubleline=true]{B}{R}\ncput[npos=1.65]{$\vphantom{\Big(} E_8[\II] $}
  \ncline[nodesep=0pt,doubleline=true]{B}{S}\ncput[npos=1.65]{$\vphantom{\Big(} E_8[-\II] $}
  \psellipticarc[linewidth=1pt]{->}(1.5,2)(0.7,0.7){15}{45}
  \psellipticarc[linewidth=1pt]{->}(1.5,2)(0.7,0.7){135}{165}
  \psellipticarc[linewidth=1pt]{->}(1.5,2)(0.7,0.7){195}{225}
  \psellipticarc[linewidth=1pt]{->}(1.5,2)(0.7,0.7){255}{285}
  \psellipticarc[linewidth=1pt]{->}(1.5,2)(0.7,0.7){315}{345}
  
   \ncline[nodesep=0pt,doubleline=true]{bA}{bB}\naput[npos=1.15]{$\vphantom{\Big(} \mathrm{St}$}
  \ncline[linestyle=dashed,nodesep=0pt]{bB}{bC}
  \ncline[nodesep=0pt,doubleline=true]{bB}{bI}\ncput[npos=1.65]{$\vphantom{\Big(} E_8[-\II] $}
  \ncline[nodesep=0pt,doubleline=true]{bB}{bJ}\ncput[npos=1.65]{$\vphantom{\Big(} E_8[\II] $}
  \ncline[nodesep=0pt,doubleline=true]{bB}{bR}\ncput[npos=1.65]{$\vphantom{\Big(} E_8[\zeta^4] $}
  \ncline[nodesep=0pt,doubleline=true]{bB}{bS}\ncput[npos=1.65]{$\vphantom{\Big(} E_8[\zeta] $}
  \psellipticarc[linewidth=1pt]{->}(6.5,2)(0.7,0.7){15}{45}
  \psellipticarc[linewidth=1pt]{->}(6.5,2)(0.7,0.7){135}{165}
  \psellipticarc[linewidth=1pt]{->}(6.5,2)(0.7,0.7){195}{225}
  \psellipticarc[linewidth=1pt]{->}(6.5,2)(0.7,0.7){255}{285}
  \psellipticarc[linewidth=1pt]{->}(6.5,2)(0.7,0.7){315}{345}
    
     \ncline[nodesep=0pt,doubleline=true]{cA}{cB}\naput[npos=1.15]{$\vphantom{\Big(} \mathrm{St}$}
  \ncline[linestyle=dashed,nodesep=0pt]{cB}{cC}
  \ncline[nodesep=0pt,doubleline=true]{cB}{cI}\ncput[npos=1.65]{$\vphantom{\Big(} E_8[\II] $}
  \ncline[nodesep=0pt,doubleline=true]{cB}{cJ}\ncput[npos=1.65]{$\vphantom{\Big(} E_8[-\II] $}
  \ncline[nodesep=0pt,doubleline=true]{cB}{cR}\ncput[npos=1.65]{$\vphantom{\Big(} E_8[\zeta^4] $}
  \ncline[nodesep=0pt,doubleline=true]{cB}{cS}\ncput[npos=1.65]{$\vphantom{\Big(} E_8[\zeta] $}
  \psellipticarc[linewidth=1pt]{->}(11.5,2)(0.7,0.7){15}{45}
  \psellipticarc[linewidth=1pt]{->}(11.5,2)(0.7,0.7){135}{165}
  \psellipticarc[linewidth=1pt]{->}(11.5,2)(0.7,0.7){195}{225}
  \psellipticarc[linewidth=1pt]{->}(11.5,2)(0.7,0.7){255}{285}
  \psellipticarc[linewidth=1pt]{->}(11.5,2)(0.7,0.7){315}{345}
  
\end{pspicture}
\end{center}
\caption{Wrong planar embeddings for the principal $\Phi_{20}$-block of $E_8(q)$}
\label{subtree20st}
\end{figure}
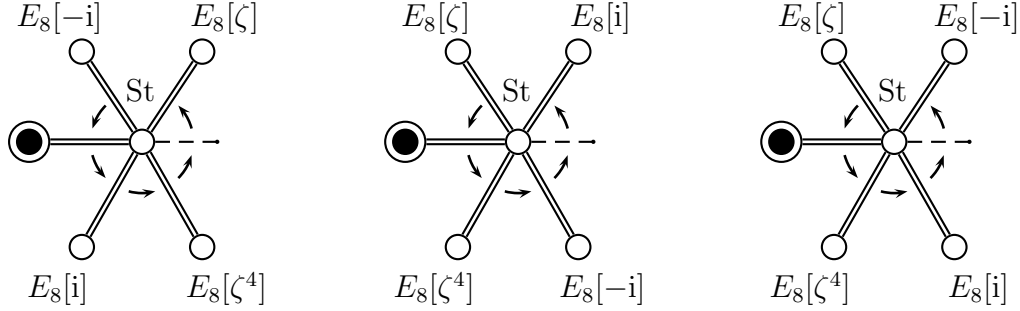
Recall that $w$ denotes a $24$-regular element. As in the case of $\overline{\rX}(c)$,  Proposition \ref{xbartorsion} ensures that the torsion part in the cohomology of $b\Rgc(\overline{\rX}(w),\mathcal{O})$ is cuspidal. Let 
$C=\bigl(b\tRgc(\overline{\rX}(w),\mathcal{O})_{(q^{10})}\bigr)^{\redu}$,
a complex with $0$ terms in degrees less than $0$ and greater than $20$. We have
$$KC  \, \simeq \, E_8[\II] [-10] \oplus K[-20].$$

We will describe completely the complex $C$, and rule out the wrong
planar embeddings. We will proceed in a number of steps.

\smallskip
$\bullet\ $Step 1: the only non-cuspidal simple module that
can appear as a composition factor of $\H^*(kC)$ is $K$, and it can only
appear in $\H^{20}(kC)$. The simple modules $\St_\ell$,
$E_8[\zeta^2]$ and $E_8[\zeta^3]$ do not occur as 
composition factors of $\H^*(kC)$.

The first statement follows from the discussion above.
As a consequence, we have $P_{S_i} \otimes_{kG} kC \simeq 0$
 for $i =0,\ldots,3$ and therefore 
$$P_{E_8[\zeta^3]} [-8]\otimes_{k G} k C \, \simeq \,  b\Rgc(\overline{\rX}(c),k)_{(q^{16})} \otimes_{k G} kC.$$
The latter is a direct summand of 
$\Rgc(\overline{\rX}(c) \times_G \overline{\rX}(w))$,
which by Proposition \ref{xbarxbar} has no torsion in its cohomology. We deduce that $P_{E_8[\zeta^3]} \otimes_{kG} kC$ is quasi-isomorphic to zero, which
means that $E_8[\zeta^2]$ does not occur as a composition factor in
$\H^*(kC)$. The same result can be shown to hold for $E_8[\zeta^3]$, by
replacing $b\Rgc(\overline{\rX}(c),k)_{(q^{16})}$ by
$(b\Rgc(\overline{\rX}(c),k)_{(q^{16})})^*[-16]\simeq 
b\Rgc(\overline{\rX}(c),k)_{(q^{4})}$.
The statement about $\St_\ell$ follows from Proposition \ref{gelfandgraevxbar}.

\smallskip
$\bullet\ $Step 2: 
$E_8[\zeta^4]$ does not occur as a composition factor of
$\H^*(kC)$ and
$E_8[\zeta]$ does not occur as a composition factor of
$\H^i(kC)$ for $i{\not\in}\{12,13\}$.

We have $b\Rgc(\overline{\rX}(c),k)_{(1)} \otimes_{kG} kC \simeq k[-20]$ and
$\Rgc(\overline{\rX}(c),k)_{(q^8)} \otimes_{kG} kC \simeq k[-36]$. Moreover,
$P_i \otimes_{kG} kC \simeq 0$  for $i =1,\ldots,7$ but $P_0 \otimes_{kG} kC \simeq k[-20]$, so we obtain from (\ref{eq:complexxbar}) a
 distinguished triangle
$$
P_{E_8[\zeta]} [-9] \otimes_{kG} kC \to k[-36] \to k[-36]  \rightsquigarrow$$
Using $\Rgc(\overline{\rX}(c),k)_{(q^{12})}\simeq 
\bigl(\Rgc(\overline{\rX}(c),k)_{(q^8)}\bigr)^*[-16]$ instead of
$\Rgc(\overline{\rX}(c),k)_{(q^8)}$, we obtain a distinguished triangle
$$k[-20] \to k[-20] \to P_{E_8[\zeta^4]} [-7] \otimes_{kG} kC \rightsquigarrow$$

\noindent The variety $\overline{\rX}(w)$ has dimension $10$, and therefore its cohomology vanishes outside the degrees $0,\ldots,20$. Therefore $P_{E_8[\zeta]} \otimes_{kG} kC\simeq 0$. We also deduce that 
$P_{E_8[\zeta^4]} \otimes_{kG} kC$ is quasi-isomorphic to  either $0$ or
 $k[-12] \oplus k[-13]$.

\smallskip
$\bullet\ $Step 3: $P_{S_4}$, $P_{\St_\ell}$ and
$P_{E_8[-\II]}$ and do not occur in $C$, while
$P_{E_8[\II]}$ occurs with multiplicity $1$ in $C$ (and this is in $C^{10}$).

The statements about $P_{E_8[\pm\II]}$ are clear using Proposition 
\ref{prop:middledegree}, while the other
two statements follow from Lemma \ref{le:minwforexc}.

\smallskip
We have now enough information to determine $C$
and rule out the planar embeddings given in Figure \ref{subtree20st}.

\smallskip
$\bullet\ $Step 4: $C^i=0$ for $i<10$.

Let $i$ be the smallest degree for which $\H^i(C)$ has non-zero
torsion. Assume that $i\le 10$.
The cohomology $\H^{i-1}(kC)$ is cuspidal with socle in
$\{S_4, E_8[\II],E_8[-\II]\}$.
On the other hand, $kC^{<(i-1)}=0$ and the injective hulls of $S_4$ and
$E_8[\pm\II]$ do not occur as direct summands of $kC^{i-1}$, a contradiction.
It follows that $\H^i(C)=0$ for $i<10$ and $\H^{10}(C)$ is torsion-free. So,
$\H^i(kC)=0$ for $i<10$, hence $(kC)^i=0$ for $i<10$.

\smallskip
$\bullet\ $Step 5: We have $\H^i(C)=0$ for $14\le i\le 19$ and $\H^{20}(C)=\OC$.

Lemma \ref{le:twononzerocoh} applied to the stupid truncation
$C^{13}\to C^{14}\to\cdots\to
C^{20}$ (viewed in degrees $-7,\ldots,0$) shows that 
$$C\simeq 0\to C^{10}\to C^{11}\to C^{12}\to C^{13}\to P_6\to P_5\to\cdots\to P_1\to
P_0\to 0,$$
$\H^i(C)=0$ for $14\le i\le 19$ and $\H^{20}(C)=\OC$.

\smallskip
$\bullet\ $Step 6: $\H^{10}(kC) = E_8[\II]$ and $\H^{11}(C)=0$.

By the universal coefficient theorem $\H^{10}(kC)$ is an extension of 
$L=\mathrm{Tor}_1^{\OC}(\H^{11}(C),k)$ by $k\H^{10}(C)=E_8[\II]$.
Since $\Ext^1(M,E_8[\II])=0$ for all $kG$-modules $M$ with composition factors
in $\{S_4, E_8[\II],E_8[-\II]\}$, it follows that
the $kG$-module $L$
is a direct summand of $\H^{10}(kC)$, hence $C^{10}$ has an injective hull of
$L$ as a direct summand of $kC^{10}$. This shows that $C^{10} =  P_{E_8[\II]}$
and $L=0$.

\smallskip
$\bullet\ $Step 7: $\Ext^1(\St_\ell,E_8[\II])=0$.

The differential $C^{10} \to C^{11}$ induces an injective map
$\Omega^{-1} E_8[\II] \hookrightarrow C^{11}$. Since $P_{\St_\ell}$ is
not a direct summand of $C^{11}$, it follows that $\St_\ell$ does not
occur in the socle of $C^{11}$, hence not in the socle of $\Omega^{-1} E_8[\II]$.

\smallskip
 This rules out the first possibility of the planar embedding around $\St$
in Figure \ref{subtree20st}.

\smallskip
$\bullet\ $Step 8: $\H^{12}(C)=0$.

Let $L=\mathrm{Tor}_1^{\OC}(\H^{12}(C),k)$.
The $kG$-module $\Omega^{-2} E_8[\II]$ has no composition factors isomorphic to
$S_4$,  $E_8[\II]$ or $E_8[-\II]$, hence
$\Hom(L,\Omega^{-2}E_8[\II])$. It follows
that
$\Ext^1(L,\Omega^{-1}E_8[\II]) =0$,
hence an injective hull of $L$ is a direct summand of $kC^{11}$, which 
forces $L=0$, hence $\H^{12}(C)=0$.

\smallskip
$\bullet\ $Step 9: $C^{13}\simeq P_7$.

We have $C^{13}\simeq P_7\oplus R$ for some projective $kG$-module $R$, whose
head is in $\H^{13}(kC)$. It follows from Steps 1-3 that $R\simeq P_{E_8[\zeta]}^{\oplus n}$
for some $n\ge 0$. Assume that $n>0$. Since $\St_\ell$ does not occur as a
composition factor
of $\H^{13}(kC)$, it follows that $E_8[\varepsilon i]$ must occur immediately after
$E_8[\zeta]$ in the cyclic ordering around $\St$ for some $\varepsilon\in\{+,-\}$
and $P_{E_8[\varepsilon i]}$ occurs as a direct summand of $C^{12}$: this is
a contradiction. We deduce that $C^{13}\simeq P_7$.

\smallskip
$\bullet\ $Step 10: conclusion.

Assume that the configuration around $\St$ is the second one in Figure
\ref{subtree20st}. Then $\Omega^{-3}E_8[\II]$ is an extension of $S_6$ by $S_5$.
Since $S_5$ does not occur as a composition factor of $\H^{12}(kC)$, it
follows that $P_5$ is a direct summand of $C^{13}$, a contradiction.
Assume now the configuration is the third one in Figure \ref{subtree20st}.
The socle of $\Omega^{-3}E_8[\II]$ is $\St_\ell$. Since $\St_\ell$ does
not occur as a composition factor of $\H^{12}(kC)$ and a projective
cover does not occur as a direct summand of $C^{13}$, we obtain
a contradiction. This concludes the determination of the Brauer tree.
Note that now $\Omega^{-2}E_8[\II]=E_8[\zeta]$, $C^{12}\simeq
P_{E_8[\zeta]}$ and $\H^{12}(kC)=0$. In particular, $\H^*(C)$ is
torsion-free and
$C$ is 
$$ 0 \to P_{E_8[\II]} \to P_{E_8[\zeta]}\to P_{E_8[\zeta]} \to P_7 \to P_6 \to
\cdots \to P_0 \to 0.$$

\subsubsection{$d=24$}
Several Harish-Chandra series lie in the principal  $\Phi_d$-block, and a (Hecke) argument gives
the corresponding subtrees, as well as the real stem,
as shown in Figure \ref{24E8}. 

\smallskip

$\bullet\ $Step 1: cuspidal modules $E_8[-\theta]$ and $E_8[-\theta^2]$.

The two cuspidal characters $E_8[-\theta]$ and $E_8[-\theta^2]$ appear in the cohomology of a
Coxeter variety $\rX(c)$. To locate them on the Brauer tree we shall look at the cohomology
of a compactification $\overline{\rX}(c)$ and proceed as in the beginning of
\S\ref{se:E8d20}. 
We have
$$
b\Rgc(\overline{\rX}(c),K)_{(q^8)} \simeq (E_8[-\theta^2])[-8] \oplus K[-16].
$$
we deduce from Corollary \ref{co:Coxetertriv} and Theorem \ref{thm:coxeter}
(see Proposition \ref{prop:resultsxbar}) that
\begin{equation}
\label{complexrep}
 b\Rgc(\overline{\rX}(c),\OC)_{(q^8)} \simeq
 0 \to P_{E_8[ -\theta^2]} \to P_7 \to
\cdots \to P_0\to 0,
\end{equation}
where $P_1, \ldots, P_7$ is the unique path of projective covers of
non-cuspidal simple modules corresponding to edges from $k$ to $\St$ in the
Brauer tree, and the tree in Figure \ref{subtree24right} is a subtree of $T$.
Furthermore, the only principal series vertex connected by an edge to
a non-principal series vertex is $\St$.

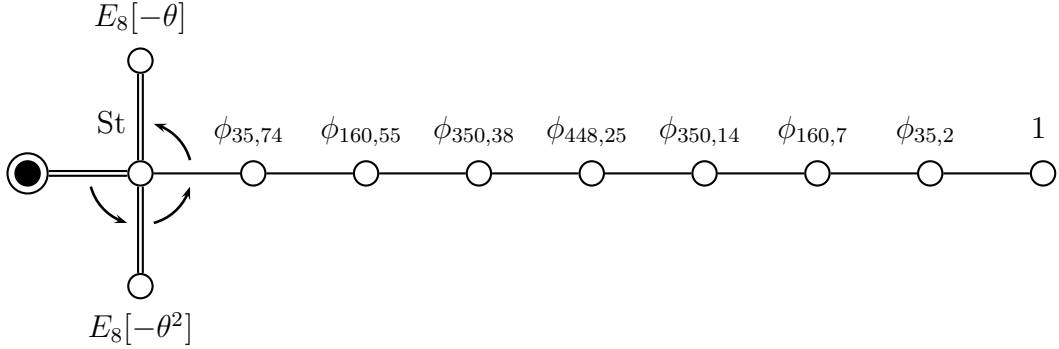
\begin{figure}[h!]
\begin{center}
\begin{pspicture}(13.5,4)

  \cnode[fillstyle=solid,fillcolor=black](0,2){5pt}{A2}
    \cnode(0,2){8pt}{A}
  \cnode(1.5,2){5pt}{B}
  \cnode(3,2){5pt}{C}
  \cnode(4.5,2){5pt}{D}
  \cnode(6,2){5pt}{E}
  \cnode(7.5,2){5pt}{F}
  \cnode(9,2){5pt}{G}
  \cnode(10.5,2){5pt}{H}
  \cnode(12,2){5pt}{K} 
   \cnode(13.5,2){5pt}{L}  
  \cnode(1.5,0.5){5pt}{I}
  \cnode(1.5,3.5){5pt}{J}

  \ncline[nodesep=0pt,doubleline=true]{A}{B}\naput[npos=0.8]{$\vphantom{\Big(} \mathrm{St}$}
  \ncline[nodesep=0pt]{B}{C}\naput[npos=1.1]{$\vphantom{\Big(}\phi_{35,74}$}
  \ncline[nodesep=0pt]{C}{D}\naput[npos=1.1]{$\vphantom{\Big(}  \phi_{160,55}$}
  \ncline[nodesep=0pt]{D}{E}\naput[npos=1.1]{$\vphantom{\Big(}\phi_{350,38}$}
  \ncline[nodesep=0pt]{E}{F}\naput[npos=1.1]{$\vphantom{\Big(} \phi_{448,25}$}
  \ncline[nodesep=0pt]{F}{G}\naput[npos=1.1]{$\vphantom{\Big(} \phi_{350,14}$}
  \ncline[nodesep=0pt]{G}{H}\naput[npos=1.1]{$\vphantom{\Big(} \phi_{160,7}$}
  \ncline[nodesep=0pt,doubleline=true]{B}{I}\ncput[npos=1.65]{$\vphantom{\Big(} E_8[-\theta^2] $}
  \ncline[nodesep=0pt,doubleline=true]{B}{J}\ncput[npos=1.65]{$\vphantom{\Big(} E_8[-\theta] $}
  \ncline[nodesep=0pt]{H}{K}\naput[npos=1.1]{$\vphantom{\Big(} \phi_{35,2}$}
  \ncline[nodesep=0pt]{K}{L}\naput[npos=1.1]{$\vphantom{\Big(} 1 $}
  \psellipticarc[linewidth=1pt]{->}(1.5,2)(0.7,0.7){15}{75}
  \psellipticarc[linewidth=1pt]{->}(1.5,2)(0.7,0.7){195}{255}
  \psellipticarc[linewidth=1pt]{->}(1.5,2)(0.7,0.7){285}{345}

\end{pspicture}
\end{center}
\caption{Subtree of the principal $\Phi_{24}$-block of $E_8(q)$}
\label{subtree24right}
\end{figure}

\smallskip

$\bullet\ $Step 2: $E_6$-series.

We now locate the $E_6$-series characters. By a (Degree) argument,
$E_6[\theta]_{\phi_{1,3}''}$ and $E_6[\theta^2]_{\phi_{1,3}''}$ are not
leaves in the tree, so, by a (Parity) argument, must be connected to one of
the non-unipotent vertex, $D_{4,\phi_{8,9}''}$ or $D_{4,\phi_{8,3}'}$, and
they are connected to the same node. Note that
$E_6[\theta^{\pm 1}]_{\phi_{1,3}''}$ is connected to exactly two characters:
$E_6[\theta^{\pm 1}]_{\phi_{2,2}}$ and the real character above (by a (Parity) argument,
it cannot be connected to $E_8[\pm\II]$). For all $q$,
the degree of
$$[D_{4,\phi_{8,9}''}]-\bigl([E_6[\theta]_{\phi_{1,3}''}]
-[E_6[\theta]_{\phi_{2,2}}] +[E_6[\theta^2]_{\phi_{1,3}''}]-
[E_6[\theta^2]_{\phi_{2,2}}]\bigr)$$
is negative, hence it cannot be the class of a $kG$-module. As a consequence,
$E_6[\theta^{\pm 1}]_{\phi_{1,3}''}$ is not connected to
$D_{4,\phi_{8,9}''}$. The same statement holds for $D_{4,\phi_{8,3}'}$,
hence $E_6[\theta^{\pm 1}]_{\phi_{1,3}''}$ is connected to
the non-unipotent node.

\smallskip
 Again, (Parity) and (Degree) arguments show that the characters $E_8[\pm \II]$
are connected to the non-unipotent node, or one of the nodes 
$D_{4,\phi_{8,9}''}$, $E_6[\theta]_{\phi_{2,2}}$ or
$E_6[\theta^2]_{\phi_{2,2}}$. 
Note that, from the subtree constructed so far and since $E_8[\pm\II]$ and $E_6[\theta^{\pm 1}]_{\phi_{1,3}'}$ have the same parity, they must both be leaves in the tree and so remain irreducible modulo $\ell$.

\smallskip
Let $w\in W$ be a regular element of order $24$ and length $10$.
 We have
  $$\begin{aligned}
  b\Rgc(\overline{\rX}(w),K)_{(q^{11})}  & \, \simeq E_8[\II] [-10], \\[8pt]
  b\Rgc(\overline{\rX}(w),K)_{(q^{14})} &\, \simeq E_6[\theta]_{\phi_{1,3}'} [-12].
  \end{aligned}$$

\smallskip

$\bullet\ $Step 3: $E_8[-\theta]$ and $E_8[-\theta^2]$ do not occur
in $\H_c^*(\overline{\rX}(w),k)_{(\lambda)}$.

Let $\lambda$ be either $q^{11}$ or  $q^{14}$. The torsion part in
$b\Rgc(\overline{\rX}(w),\mathcal{O})_{(\lambda)}$ is cuspidal by
Proposition \ref{xbartorsion}. Since its character has no composition factor
in the principal series we have
$\Rgc(\overline{\rX}(w),k)_{(\lambda)}\otimes_{kG} P_i = 0$ for 
$i\in\{0,\ldots,7\}$.
Using Proposition \ref{xbarxbar}   for the variety $\overline{\rX}(w)\times_{G} \overline{\rX}(c)$ together with (\ref{complexrep}) and the
dual description of $b\Rgc(\overline{\rX}(c),\OC)_{(1)}$,
we deduce that $\Rgc(\overline{\rX}(w),k)_{(\lambda)}\otimes_{kG} P_{E_8[-\theta]} = \Rgc(\overline{\rX}(w),k)_{(\lambda)}\otimes_{kG} P_{E_8[-\theta^2]} = 0$. This ensures that neither $E_8[-\theta]$ nor $E_8[-\theta^2]$ can occur as composition factors of the cohomology of  $b\Rgc(\overline{\rX}(w),k)_{(\lambda)}$.

\smallskip

$\bullet\ $Step 4: $ b\Rgc(\overline{\rX}(w),k)_{(q^{11})}$ and
position of $E_8[\pm\II]$.

Let $C = b\Rgc(\overline{\rX}(w),k)_{(q^{11})}$ and let $M = \H^i(C)$ be the
non-zero cohomology group with largest degree. Suppose that $i > 10$. The module $M$ is cuspidal and its composition factors are cuspidal
modules different from $E_8[-\theta]$ and $E_8[-\theta^2]$. 
Proposition \ref{prop:middledegree} shows that
$\RHom^\bullet_{kG}(\Rgc(\rX(v),k),M)=0$ for all $v < w$. By the construction of
the smooth compactifications, we obtain an isomorphism
$$ \RHom^\bullet_{kG}(\Rg(\rX(w),k),M) \simto \RHom^\bullet_{kG}(\Rgc(\overline{\rX}(w),k),M).$$
Since $\Rg(\rX(w),k)$ has a representative with terms in degrees $0,\ldots,\ell(w)=10$, we deduce that $\mathrm{Hom}_{D^b(kG)}(\Rgc(\overline{\rX}(w),k),M[-i]) =0$, which is impossible since $C$ is a direct summand of $\Rgc(\overline{\rX}(w),k)$ and the map $C \longrightarrow M[-i]=\H^i(C)[-i]$ is non-zero. This shows that $\H^j(C) = 0$ for $j>10$. Using the same argument with the isomorphism 
$$ \RHom^\bullet_{kG}(M,\Rgc(\rX(w),k)) \simto \RHom^\bullet_{kG}(M,\Rgc(\overline{\rX}(w),k))$$
and the fact that $\Rgc(\rX(w),k)$ has a representative with terms in degrees $10 =\ell(w),\ldots,2\ell(w) = 20$, we deduce that $\H^j(C) = 0$ for $j<10$. Therefore $C \simeq \H^{10}(C)[-10]$. 

Now, Proposition \ref{pr:OmegaDL} and Remark \ref{re:OmegaDLbar} show that
$E_8[\II] \simeq \Omega^{12}k$. We deduce that $E_8[\pm\II]$ are connected to the
non-unipotent node and this gives the whole tree as shown in Figure \ref{24E8},
up to swapping the $E_6[\theta]$ and the $E_6[\theta^2]$-series.

\smallskip

$\bullet\ $Step 5: $ b\Rgc(\overline{\rX}(w),k)_{(q^{14})}$ and conclusion.

The previous argument applied to the complex 
$D = b\Rgc(\overline{\rX}(w),k)_{(q^{14})}$ shows that the cohomology
of $D$ vanishes outside the degrees $10$, $11$ and $12$, and that
$\H^{12}(D)$ is a module with simple head isomorphic to
$E_6[\theta]_{\phi_{1,3}'}$. The radical of $\H^{12}(D)$ is cuspidal. Since
$E_6[\theta]_{\phi_{1,3}'}$ has no non-trivial extensions with simple
cuspidal modules, we deduce that $\H^{12}(D)=E_6[\theta]_{\phi_{1,3}'}$
and $\H^{12}_c(\overline{\rX}(w),\OC)_{(q^{14})}$ is torsion-free.

Let us denote the simple modules in the $E_6[\theta]$-series as in Figure \ref{subtree24left}. 
\begin{figure}[h!]
\begin{center}
\begin{pspicture}(7,2)
  \cnode[fillstyle=solid,fillcolor=black](1,1){5pt}{A2}
    \cnode(1,1){8pt}{A}
  \cnode(3,1){5pt}{B}
  \cnode(5,1){5pt}{C}
  \cnode(7,1){5pt}{D}
  \cnode(0,1){0pt}{E}
  \cnode(0.1,1.4){0pt}{E1}
  \cnode(0.1,0.6){0pt}{E2}
  \cnode(0.4,1.8){0pt}{E3}
  \cnode(0.4,0.2){0pt}{E4}

  \ncline[nodesep=0pt,doubleline=true]{A}{B}\naput[npos=1.1]{$\vphantom{\Big(} E_6[\theta]_{\phi_{1,3}''}$}\nbput{$S_3$}
  \ncline[nodesep=0pt]{B}{C}\nbput{$S_2$}\naput[npos=1.1]{$\vphantom{\Big(}E_6[\theta]_{\phi_{2,2}}$}
  \ncline[nodesep=0pt]{C}{D}\nbput{$S_1$}\naput[npos=1.1]{$\vphantom{\Big(}E_6[\theta]_{\phi_{1,3}'}$}
  \ncline[nodesep=0pt,linestyle=dashed]{E}{A}
  \ncline[nodesep=0pt,linestyle=dashed]{E1}{A}
  \ncline[nodesep=0pt,linestyle=dashed]{E2}{A}
  \ncline[nodesep=0pt,linestyle=dashed]{E3}{A}
  \ncline[nodesep=0pt,linestyle=dashed]{E4}{A}
\end{pspicture}
\end{center}
\caption{Subtree of the principal $\Phi_{24}$-block of $E_8(q)$}
\label{subtree24left}
\end{figure}
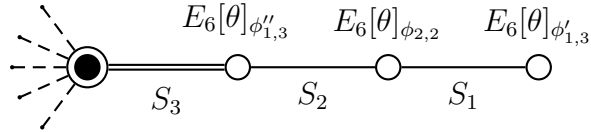
There exists a representative of $D$ of the form 
$$D=0 \to X \to P'\oplus P_{S_2} \to P_{S_1} \to 0,$$
where $P'$ is a projective module with no cuspidal simple quotient except
possibly
$E_8[-\theta]$ or $E_8[-\theta^2]$ (by Proposition \ref{prop:middledegree}).
Since $\H^{11}(D)$ is a cuspidal module with no composition factor isomorphic
to $E_8[-\theta]$ or $E_8[-\theta^2]$, we deduce that the representative
of $D$ can be chosen so that $P'=0$. 
By the universal coefficient theorem, we have $\H^{10}(D) \simeq \H^{11}(D)$.
We have $\H^{11}(D)=0$ or $\H^{11}(D)=S_3$. In both cases, we find that
$X$ is a module with composition factors $S_2$ and $S_3$.

Proposition \ref{pr:OmegaDL} and Remark \ref{re:OmegaDLbar} show that
$X \simeq \Omega^{18} k$ in $kG\mstab$. We deduce that 
 $\Omega^{18} k$ lifts to an $\mathcal{O}G$-lattice of character $E_6[\theta]_{\phi_{1,3}''}$, which gives the planar embedding.

\subsection{Other exceptional groups}
\label{sec:ree}
The Brauer trees of unipotent blocks for exceptional groups other than $E_7(q)$ and
$E_8(q)$ were determined in \cite{Bur,Sh1,Hi90,Ge,HL,HLM} (under an assumption
on $q$ for one of the blocks in ${}^2E_6(q)$), but only up to choice of field
of values in each block. This ambiguity can be removed using
Lusztig's parametrization of unipotent characters. We achieve this by choosing
carefully the roots of unity in $\overline{\mathbb{Q}}_\ell$ associated with
the cuspidal characters, as we did in the previous sections. 

\subsubsection{$E_6(q)$, ${}^2E_6(q)$, $F_4(q)$ and $G_2(q)$}

For each of the exceptional groups of type $E_6(q)$, ${}^2E_6(q)$, $F_4(q)$ and $G_2(q)$ there are only two blocks with cyclic defect groups whose Brauer trees are not lines. One of the blocks corresponds to the principal $\Phi_h$-block with $h$ the Coxeter number, and this case was solved in \cite{Du4}. For the other one, one proceeds exactly as in \S\ref{se:E7d14}, where only a pair of conjugate cuspidal characters lies outside the real stem  (these characters appear in the cohomology of a Coxeter variety). The planar-embedded Brauer trees can be found in \cite{Cra3}.

\subsubsection{${}^2B_2(q^2)$ and ${}^2G_2(q^2)$}

For the Suzuki groups ${}^2B_2(q^2)$ and the Ree groups ${}^2G_2(q^2)$, the
Frobenius eigenvalue corresponding to each unipotent character is known by \cite{Bru}. It is enough to locate a single non-real character to fix the planar
embedding. One can take this character to be a non-real cuspidal character
occurring in the cohomology of the Coxeter variety and proceed as before to
get the trees given in \cite{Cra3}. Note that for these groups the Coxeter
variety is $1$-dimensional, therefore its cohomology is torsion-free and $\Omega^2 k$ is isomophic in $kG\mstab$ to the generalized $(q^2)$-eigenspace of $F^2$ in $\Rgc(\rX(c),k)$ (when $d$ is not the Coxeter number). 

\subsubsection{${}^2F_4(q^2)$}

We now consider the Ree groups ${}^2F_4(q^2)$, whose Brauer trees have been
determined in \cite{Hi90} using the parametrization given in \cite{Ma},
but not using Lusztig's parametrization.

 Here, there are three trees that are not lines. One of them corresponds to the case solved in \cite{Du4}, and another one is similar to\
\S\ref{se:E7d14}. The only block which deserves a specific treatment is the principal $\ell$-block with $\ell \mid (q^4+\sqrt2q^3+q^2+\sqrt2q+1)$ (so,
$q$ is a $24$-th root of unity modulo $\ell$).
Let $\eta$, $i$ and $\theta$ be the roots of unity in $\OC$
having the same image as respectively $q^{15}$, $q^6$ and $q^{16}$
in the residue field $k$.

 A (Hecke) argument gives the
real stem of the Brauer tree as shown in Figure \ref{2F4} as well as the
two edges for the ${}^2B_2$-series.

 We consider the two generalized `mod- $\ell$-eigenspaces' of $F^2$ on the
cohomology of the Coxeter variety given by
$$ \begin{aligned}
\Rgc(\rX(c),K)_{(q^{-2})} & \, = \big({}^2B_2[\eta^3]_\varepsilon \oplus 
{}^2F_4[-\theta^2]\big)[-2], \\
\Rgc(\rX(c),K)_{(q^4)} & \, = {}^2B_2[\eta^5]_\varepsilon[-2] \oplus K[-4]. 
\end{aligned}$$

Lemma \ref{lem:onedegree} and Proposition \ref{pr:torsionCoxeter} 
show that ${}^2B_2[\eta^3]_\varepsilon + {}^2F_4[-\theta^2]$ is the character of a
projective module $P_{{}^2F_4[-\theta^2]}$, hence ${}^2B_2[\eta^3]_\varepsilon$ and
${}^2F_4[-\theta^2]$ are connected by an edge. Furthermore, 
$\Rgc(\rX(c),\OC)_{(q^{-2})}\simeq P_{{}^2F_4[-\theta^2]}[-2]$.

Corollary \ref{co:Coxetertriv} show that there is no non-real vertex connected
to $1$ or $\phi_{2,3}$, that there is an edge $S[\eta^5]$ connecting
$\St$ and ${}^2B_2[\eta^5]_\varepsilon$,
and, in the cyclic ordering of edges containing
$\St$, the edge $S[\eta^5]$
comes after the one containing $\phi_{2,3}$ and before $\St_\ell$.
Furthermore,
$$\Rgc(\rX(c),\OC)_{(q^{-2})}\simeq 0\to P_{S[\eta^5]}\to P_1\to
P_k\to 0,$$
where $P_k$ is in degree $4$ and
$P_1$ is projective with character $\phi_{2,3}+\St$.

We can now deduce the corresponding complexes of cohomology for
$\overline{\rX}(c)$. For $\lambda\in\{q^{-2},q^4\}$ and $I$ an
$F$-stable proper subset of $S$, we have 
$b\Rgc(\rX(c_I),\OC)_{(\lambda)}=0$ unless $(L_i,F)$ has type ${}^2B_2$, 
in which case the complex has cohomology concentrated in degree $1$. In addition, using duality for the case $\lambda\in\{q^6,1\}$, we find
$$\begin{aligned}
b\Rgc(\overline{\rX}(c),\OC)_{(q^{-2})} \simeq & \, 
0 \to \, 0 \ \to P_{{{}^2B_2}[\eta^3]_1} \to P_{{}^2F_4[-\theta^2]} \, \to \ \, 0 \ \, \to \, 0 \, \to 0, \\ 
b\Rgc(\overline{\rX}(c),\OC)_{(q^{6})} \simeq & \, 
0 \to \, 0 \ \to \, \, 0 \, \ \to P_{{}^2F_4[-\theta]} \to P_{{{}^2B_2}[\eta^5]_1}
\to \, 0\, \to 0, \\
b\Rgc(\overline{\rX}(c),\OC)_{(q^{4})} \simeq & \, 
0 \to \, 0 \ \to P_{{{}^2B_2}[\eta^5]_1} \to \, P_{S[\eta^5]} \ \to P_1  \, \to P_k \to 0, \\
b\Rgc(\overline{\rX}(c),\OC)_{(1)} \simeq & \, 
0 \to P_k \to \, P_1\, \to \, P_{S[\eta^3]} \ \to P_{{{}^2B_2}[\eta^3]_1}  \to \, 0 \,  \to 0, \end{aligned}$$
where $S[\eta^3]$ is the edge connecting $\St$ and 
${{}^2B_2}[\eta^3]_\varepsilon$.
Since $\Rgc(\overline{\rX}(c)\times_G \overline{\rX}(c),\OC)$ is torsion-free
(Proposition \ref{xbarxbar}),
we deduce that the differentials between non-zero terms of the complexes above
cannot be zero. This determines uniquely the four complexes above up to
isomorphism.

We have
$$b\Rgc(\overline{\rX}(c),\OC)_{(q^{-2})}\otimes_{\OC G}
b\Rgc(\overline{\rX}(c),\OC)_{(q^{4})} \simeq $$
$$\Hom^\bullet_{kG}(0\to  P_{S[\eta^3]}\to P_{{{}^2B_2}[\eta^3]_1}\to 0,
0 \to P_{{{}^2B_2}[\eta^3]_1} \to P_{{}^2F_4[-\theta^2]} \to 0)[-3].$$
By Proposition \ref{xbarxbar}, this complex $D$
has homology $\OC$ concentrated in degree $2$.
Assume that, in the cyclic ordering of edges containing 
${{{}^2B_2}[\eta^3]_\varepsilon}$, 
the edge containing ${{}^2F_4[-\theta^2]}$ comes after the edge
containing ${{{}^2B_2}[\eta^3]_1}$ but before the edge containing $\St$.
Then a non-zero map $kP_{S[\eta^3]}\to kP_{F_4[-\theta^2]}$ does not factor
through $kP_{{{}^2B_2}[\eta^3]_1}$: so, it gives rise to a non-zero
element of $\H^4(kD)$, a contradiction. It follows that the subtree obtained
by removing ${}^2F_4[\pm\II]$ is given by Figure \ref{2F4}.

\smallskip

Let $w\in W$ of length $6$ such that $wF$ has order $8$
and let $C = b\Rgc(\overline{\rX}(w),\mathcal{O})_{(-1)}$. There are 12
such elements and they are all $F$-conjugate. The complex $C$ is a perfect
complex; the torsion part of its cohomology is cuspidal by Proposition
\ref{xbartorsion} and it does not involve $\St_\ell$ by Proposition
\ref{gelfandgraevxbar}. In addition, there is a representative of $C$ that
involves neither $P_{{}^2F_4^{\mathrm{IV}}[-1]}$ nor $P_{\St_\ell}$
by Lemma \ref{le:minwforexc}. 
It follows that ${}^2F_4^{\mathrm{IV}}[-1]$ does not occur as a composition
factor of the cohomology of $kC$. Therefore the possible composition
factors in the torsion part of $\H^*(C)$ are the cuspidal simple
modules ${}^2F_4[\pm\II]$, ${}^2F_4[-\theta^j]$ and $S[\eta^m]$.

\smallskip
The cohomology of $KC$ is given by
$$ KC \simeq (F_4[\II])[-6] \oplus F_4[-\theta]^{\oplus 3}[-8] \oplus 
{}^2B_2[\eta^5]_1^{\oplus 5}[-9] \oplus K[-12].$$
Using Proposition \ref{xbarxbar} one can easily compute 
$kC \otimes^{\bbL}_{kG} b\Rgc(\overline{\rX}(c),k)_{(\lambda)}$ for the various eigenvalues $\lambda$ of $F^2$.
With the same method as in Steps 1 and 2 of \S\ref{se:E8d20}, the cases $\lambda = q^{-2}, q^6, q^4,1$ show that
\begin{itemize}
  \item ${}^2F_4[-\theta]$ can occur as a composition factor of $\H^*(kC)$ only in degrees $8$ or $9$, because $\Hom_{D^b(kG}(P_{{{}^2B_2}[\eta^5]_1},
kC[i])=0$ for $i\not=9$;
  \item ${}^2F_4[-\theta^2]$ does not occur as a composition factor of $\H^*(kC)$ because \linebreak$\Hom_{D^b(kG}(P_{{{}^2B_2}[\eta^3]_1},
kC[i])=0$ for all $i$;
  \item $S[\eta^3]$ does not occur as a composition factor of $\H^*(kC)$
because \linebreak$\Hom_{D^b(kG}(P_{{{}^2B_2}[\eta^3]_1},kC[i])=
\Hom_{D^b(kG}(P_1,kC[i])=0$ for all $i$;
  \item $S[\eta^5]$ can occur as a composition factor of $\H^*(kC)$ only in
degrees $9$, $10$ and $11$ because $\Hom_{D^b(kG}(P_1,kC[i])=0$ for all $i$ and
$\Hom_{D^b(kG}(P_{{{}^2B_2}[\eta^5]_1}, kC[i])=0$ for $i\not=9$.
\end{itemize}

\smallskip
There are five distinct possible planar trees other than the one in Figure
\ref{2F4}. One checks that, for each of those five bad embeddings, one of the
following holds:
\begin{itemize}
\item
$\Omega^{-3}({}^2 F_4[\II])$ and $\Omega^{-4}({}^2 F_4[\II])$ 
do not contain ${}^2F_4[-\theta]$ as a submodule,
\item 
$\Omega^{-4}({}^2 F_4[\II])$ 
does not contain ${{}^2B_2}[\eta^5]_1$ as a submodule,
\item 
$\Omega^{-4}({}^2 F_4[\II])$, $\Omega^{-5}({}^2 F_4[\II])$ 
and $\Omega^{-6}({}^2 F_4[\II])$ 
do not contain $S[\eta^5]$ as a submodule,
\item $\Omega^{-7}({}^2 F_4[\II])$ 
does not contain $k$ as a submodule, or
\item $\Omega^{-j}({}^2 F_4[\II])$ 
does not contain ${}^2 F_4[\II]$ nor ${}^2 F_4[-\II]$ as a submodule for
$1\le j\le 6$.
\end{itemize}
Since $k\H^6(C)\simeq {}^2F_4[\II]$, it
follows that $\Ext^{j+1}_{kG}(\H^{6+j}(kC),k\H^6(C))=0$ for $j\ge 1$
and $\Ext^1_{kG}(\mathrm{Tor}_1^\OC(k,\H^7(C)),k\H^6(C))=0$.
 Let $D$ be the cone of the canonical map $k\H^{6}(C)\to kC[6]$.
We have $\Hom_{D^b(kG)}(D,k\H^6(C)[1])=0$, hence $k\H^6(C)$ is isomorphic
to a direct summand of $C$. Since $C$ is perfect and 
${}^2F_4[\II]$ is not projective, we have a contradiction.
 This proves that the tree in Figure \ref{2F4} is correct. 

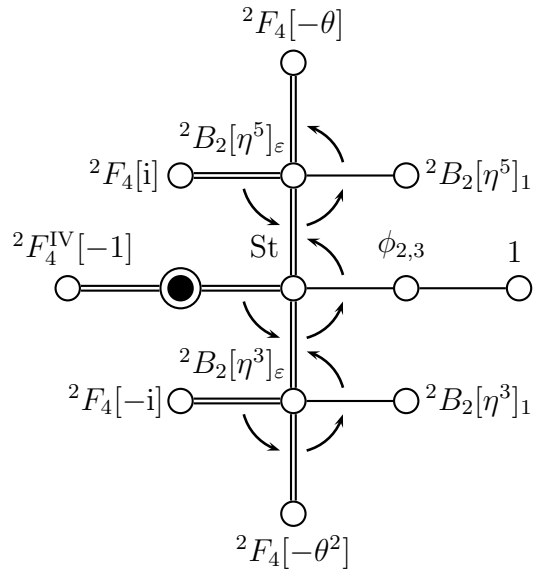
\begin{figure}[h!] 
\begin{center}
\begin{pspicture}(6,7)

  \cnode[fillstyle=solid,fillcolor=black](1.5,3.5){5pt}{A2}
  \cnode(1.5,3.5){8pt}{L}
  \cnode(3,3.5){5pt}{M}
  \cnode(3,2){5pt}{M1}
  \cnode(1.5,2){5pt}{X1}
  \cnode(4.5,2){5pt}{X2}
  \cnode(3,0.5){5pt}{M2}
  \cnode(3,5){5pt}{M3}
  \cnode(1.5,5){5pt}{Y1}
  \cnode(4.5,5){5pt}{Y2}
  \cnode(3,6.5){5pt}{M4}

  \cnode(4.5,3.5){5pt}{N}
  \cnode(6,3.5){5pt}{O}
  \cnode(0,3.5){5pt}{P}

  \ncline[nodesep=0pt]{N}{O}\naput[npos=1.1]{$\vphantom{\big(} 1 $}  
  \ncline[nodesep=0pt]{M}{N}\naput[npos=1.1]{$\vphantom{\Big(} \phi_{2,3} $}
  \ncline[nodesep=0pt,doubleline=true]{L}{M}\naput[npos=0.8]{$\vphantom{\big(}\mathrm{St}$}
  \ncline[nodesep=0pt,doubleline=true]{P}{L}\naput[npos=-0.1]{$\vphantom{\big(}{}^2F_4^{\mathrm{IV}}[-1] $}
  \ncline[nodesep=0pt,doubleline=true]{M}{M1}
  \ncline[nodesep=0pt,doubleline=true]{M}{M3}
  \ncline[nodesep=0pt,doubleline=true]{M1}{M2}\ncput[npos=1.6]{${}^2F_4[-\theta^2]$}
  \ncline[nodesep=0pt,doubleline=true]{M1}{X1}\nbput[npos=0.55]{${}^2B_2[\eta^3]_\varepsilon$}\ncput[npos=1.9]{${}^2F_4[-\II]$}

  \ncline[nodesep=0pt]{M1}{X2}\ncput[npos=2]{${}^2B_2[\eta^3]_1$}
  \ncline[nodesep=0pt,doubleline=true]{M3}{M4}\ncput[npos=1.6]{${}^2F_4[-\theta]$}
  \ncline[nodesep=0pt,doubleline=true]{M3}{Y1}\nbput[npos=0.55]{${}^2B_2[\eta^5]_\varepsilon$}\ncput[npos=1.8]{${}^2F_4[\II]$}

  \ncline[nodesep=0pt]{M3}{Y2}\ncput[npos=2]{${}^2B_2[\eta^5]_1$}  
  
  \psellipticarc[linewidth=1pt]{->}(3,3.5)(0.7,0.7){15}{75}
  \psellipticarc[linewidth=1pt]{->}(3,3.5)(0.7,0.7){195}{255}
  \psellipticarc[linewidth=1pt]{->}(3,3.5)(0.7,0.7){285}{345}

  \psellipticarc[linewidth=1pt]{->}(3,5)(0.7,0.7){15}{75}
  \psellipticarc[linewidth=1pt]{->}(3,5)(0.7,0.7){195}{255}
  \psellipticarc[linewidth=1pt]{->}(3,5)(0.7,0.7){285}{345}
 
  \psellipticarc[linewidth=1pt]{->}(3,2)(0.7,0.7){15}{75}
  \psellipticarc[linewidth=1pt]{->}(3,2)(0.7,0.7){195}{255}
  \psellipticarc[linewidth=1pt]{->}(3,2)(0.7,0.7){285}{345} 
\end{pspicture}
\end{center}
\caption{Principal $\ell$-block of ${}^2F_4(q^2)$ with $\ell \mid q^4+\sqrt2q^3+q^2+\sqrt2q+1$}
\label{2F4}
\end{figure}

\appendix

\begin{landscape}
\section{Brauer trees for $E_7(q)$ and $E_8(q)$}
\label{se:thetrees}

\begin{figure}[h!]
\begin{center}
\begin{pspicture}(21,4)

  \cnode[fillstyle=solid,fillcolor=black](9.5,2){5pt}{H2}
    \cnode(-1,2){5pt}{A}
  \cnode(0.5,2){5pt}{B}
  \cnode(2,2){5pt}{C}
  \cnode(3.5,2){5pt}{D}
  \cnode(5,2){5pt}{E}
  \cnode(6.5,2){5pt}{F}
  \cnode(8,2){5pt}{G}
   \cnode(8,0.5){5pt}{Gb} \cnode(8,3.5){5pt}{Gt}
  \cnode(9.5,2){8pt}{H}
  \cnode(11,2){5pt}{K} 
   \cnode(11,0.5){5pt}{Kb} \cnode(11,3.5){5pt}{Kt}
   \cnode(12.5,2){5pt}{L}  
  \cnode(14,2){5pt}{I}
  \cnode(15.5,2){5pt}{J}
  \cnode(17,2){5pt}{M}
  \cnode(18.5,2){5pt}{N}
  \cnode(20,2){5pt}{O}

  \ncline[nodesep=0pt]{A}{B}\naput[npos=-0.1]{$\vphantom{\Big(}\phi_{7,1}$}\naput[npos=1.1]{$\vphantom{\Big(}\phi_{56,3}$}
  \ncline[nodesep=0pt]{B}{C}\naput[npos=1.1]{$\vphantom{\Big(}\phi_{315,7}$}
  \ncline[nodesep=0pt]{C}{D}\naput[npos=1.1]{$\vphantom{\Big(}  \phi_{512,11}$}
  \ncline[nodesep=0pt]{D}{E}\naput[npos=1.1]{$\vphantom{\Big(}\phi_{280,17}$}
  \ncline[nodesep=0pt]{E}{F}\naput[npos=1.1]{$\vphantom{\Big(} \phi_{35,31}$}
  \ncline[nodesep=0pt]{F}{G}\naput[npos=1.5]{$\vphantom{\Big(} \mathrm{St}$}
  \ncline[nodesep=0pt]{G}{H}
  \ncline[nodesep=0pt]{H}{K}\naput[npos=0.6]{$\vphantom{\Big(}\phi_{7,46}$}
  \ncline[nodesep=0pt]{K}{L}\naput[npos=1.1]{$\vphantom{\Big(}\phi_{56,30}$}
  \ncline[nodesep=0pt]{L}{I}\naput[npos=1.1]{$\vphantom{\Big(}  \phi_{315,16}$}
  \ncline[nodesep=0pt]{I}{J}\naput[npos=1.1]{$\vphantom{\Big(}\phi_{512,12}$}
  \ncline[nodesep=0pt]{J}{M}\naput[npos=1.1]{$\vphantom{\Big(} \phi_{280,8}$}
  \ncline[nodesep=0pt]{M}{N}\naput[npos=1.1]{$\vphantom{\Big(}\phi_{35,4}$}
  \ncline[nodesep=0pt]{N}{O}\naput[npos=1.1]{$\vphantom{\Big(} 1$}
  \ncline[nodesep=0pt]{G}{Gb}\ncput[npos=1.6]{$E_6[\theta]_\varepsilon$}
  \ncline[nodesep=0pt]{G}{Gt}\ncput[npos=1.6]{$E_6[\theta^2]_\varepsilon$}
  \ncline[nodesep=0pt]{K}{Kb}\ncput[npos=1.6]{$E_6[\theta^2]_1$}
  \ncline[nodesep=0pt]{K}{Kt}\ncput[npos=1.6]{$E_6[\theta]_1$}

\psellipticarc[linewidth=1pt]{->}(11,2)(0.7,0.7){15}{75}
 \psellipticarc[linewidth=1pt]{->}(11,2)(0.7,0.7){195}{255}
 \psellipticarc[linewidth=1pt]{->}(11,2)(0.7,0.7){285}{345}
\psellipticarc[linewidth=1pt]{->}(8,2)(0.7,0.7){105}{165}
 \psellipticarc[linewidth=1pt]{->}(8,2)(0.7,0.7){195}{255}
 \psellipticarc[linewidth=1pt]{->}(8,2)(0.7,0.7){285}{345}

\end{pspicture}
\end{center}
\caption{Brauer tree of the principal $\Phi_9$-block of $E_7(q)$}
\label{9E7}
\end{figure}
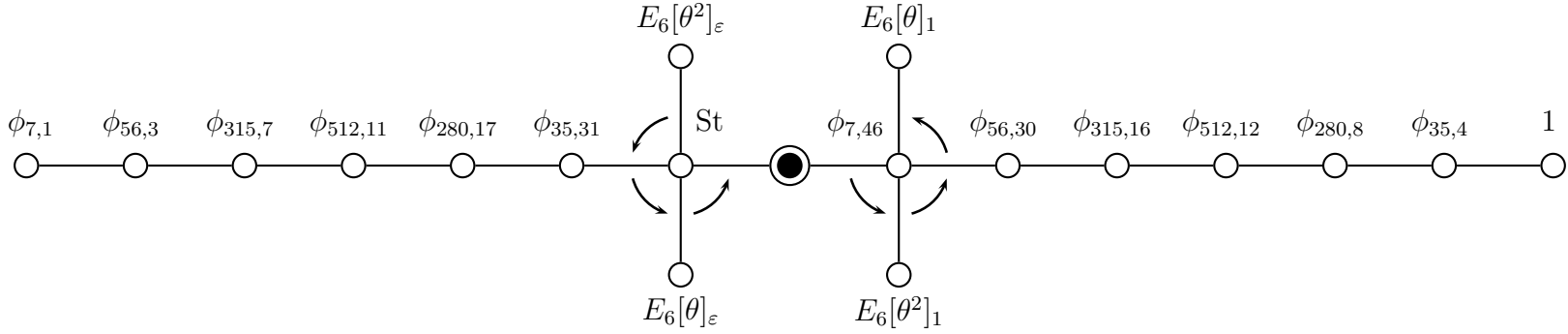
The $\Phi_9$-blocks of $E_8(q)$ have isomorphic trees, with bijection of
vertices given as follows.
$$\begin{array}{|c|ccccccccc|}
\cline{1-10}
E_7(q) & \phi_{7,1} & \phi_{56,3} & \phi_{315,7} & \phi_{512,11} & \phi_{280,17}
& \phi_{35,31} & \St & E_6[\theta^2]_\epsilon &
E_6[\theta]_\epsilon \\
E_8(q), (A_2,\phi_3) & \phi_{160,7} & \phi_{1008,9} & \phi_{2800,13} & \phi_{5600,21} &
\phi_{4096,27} & \phi_{560,47} & \phi_{112,63} &
E_6[\theta^2]_{\phi''_{1,3}} &
E_6[\theta]_{\phi''_{1,3}} \\
E_8(q), (A_2,\phi_{21}) & \phi_{35,2} & \phi_{700,6} & \phi_{2240,10} & \phi_{3150,18} &
\phi_{2240,28} & \phi_{700,42} & \phi_{35,74} & E_6[\theta^2]_{\phi_{2,2}} &
E_6[\theta]_{\phi_{2,2}} \\
E_8(q), (A_2,\phi_{1^3}) & \phi_{112,3} & \phi_{560,5} & \phi_{4096,11} & \phi_{5600,15} & \phi_{2800,25}
& \phi_{1008,39} & \phi_{160,55} & E_6[\theta^2]_{\phi'_{1,3}} &
E_6[\theta]_{\phi'_{1,3}} \\
\cline{1-10}
\end{array}$$

$$\begin{array}{|c|ccccccccc|}
\cline{1-10}
E_7(q) & \phi_{7,46} & \phi_{56,30} & \phi_{315,16}
& \phi_{512,12} & \phi_{280,8} & \phi_{35,4} & 1 &
 E_6[\theta^2]_1 & E_6[\theta]_1  \\
E_8(q), (A_2,\phi_3) & \phi_{28,68} &
\phi_{1575,34} & \phi_{4096,26} & \phi_{3200,22} & \phi_{700,16} &
\phi_{50,8} & 1 &
E_6[\theta^2]_{\phi_{1,0}} & E_6[\theta]_{\phi_{1,0}} \\
E_8(q), (A_2,\phi_{21}) & \phi_{8,91} &
\phi_{400,43} & \phi_{1400,29} & \phi_{2016,19} & \phi_{1400,11} & 
\phi_{400,7} & \phi_{8,1}  &
E_6[\theta^2]_{\phi_{2,1}} & E_6[\theta]_{\phi_{2,1}}\\
E_8(q), (A_2,\phi_{1^3}) & \phi_{1,120} & \phi_{50,56} &
\phi_{700,28}& \phi_{3200,16} & \phi_{4096,12} & \phi_{1575,10} &
\phi_{28,8}  & E_6[\theta^2]_{\phi_{1,6}} & E_6[\theta]_{\phi_{1,6}} \\
\cline{1-10}
\end{array}$$

\begin{figure}[h!]
\begin{center}
\begin{pspicture}(12,4)

  \cnode[fillstyle=solid,fillcolor=black](0,2){5pt}{A2}
    \cnode(0,2){8pt}{A}
  \cnode(1.5,2){5pt}{B}
  \cnode(3,2){5pt}{C}
  \cnode(4.5,2){5pt}{D}
  \cnode(6,2){5pt}{E}
  \cnode(7.5,2){5pt}{F}
  \cnode(9,2){5pt}{G}
  \cnode(10.5,2){5pt}{H}
  \cnode(12,2){5pt}{K} 
   \cnode(3,0.5){5pt}{Cb} \cnode(3,3.5){5pt}{Ct}

  \ncline[nodesep=0pt]{A}{B}\naput[npos=1.1]{$\vphantom{\Big(}\phi_{7,46}$}
  \ncline[nodesep=0pt]{B}{C}\naput[npos=0.6]{$\vphantom{\Big(}\phi_{27,37}$}
  \ncline[nodesep=0pt]{C}{D}\naput[npos=1.1]{$\vphantom{\Big(}  \phi_{168,21}$}
  \ncline[nodesep=0pt]{D}{E}\naput[npos=1.1]{$\vphantom{\Big(}\phi_{378,14}$}
  \ncline[nodesep=0pt]{E}{F}\naput[npos=1.1]{$\vphantom{\Big(} \phi_{378,9}$}
  \ncline[nodesep=0pt]{F}{G}\naput[npos=1.1]{$\vphantom{\Big(} \phi_{168,6}$}
  \ncline[nodesep=0pt]{G}{H}\naput[npos=1.1]{$\vphantom{\Big(}\phi_{27,2}$}
  \ncline[nodesep=0pt]{H}{K}\naput[npos=1.1]{$\vphantom{\Big(}\phi_{7,1}$}
    \ncline[nodesep=0pt,doubleline=true]{C}{Cb}\ncput[npos=1.6]{$E_7[-\II]$}
  \ncline[nodesep=0pt,doubleline=true]{C}{Ct}\ncput[npos=1.6]{$E_7[\II]$}

\psellipticarc[linewidth=1pt]{->}(3,2)(0.7,0.7){15}{75}
 \psellipticarc[linewidth=1pt]{->}(3,2)(0.7,0.7){195}{255}
 \psellipticarc[linewidth=1pt]{->}(3,2)(0.7,0.7){285}{345}

\end{pspicture}
\end{center}
\caption{Brauer tree of the $\Phi_{10}$-block of $E_7(q)$ associated to
$({}^2A_2(q).(q^5+1),\phi_{21})$}
\label{10E7}
\end{figure}
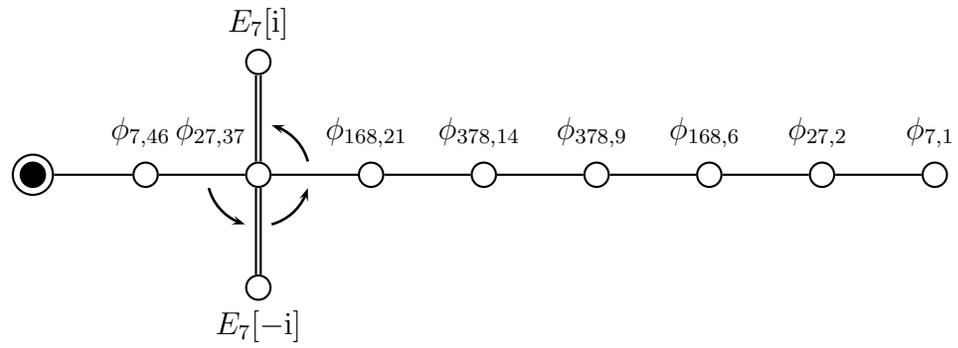

\end{landscape}

\pagebreak

\begin{landscape}

\begin{figure}[h!]
\begin{center}
\begin{pspicture}(18,4)

  \cnode[fillstyle=solid,fillcolor=black](6,2){5pt}{E2}
    \cnode(0,2){5pt}{A}
  \cnode(1.5,2){5pt}{B}
  \cnode(3,2){5pt}{C}
  \cnode(4.5,2){5pt}{D}
  \cnode(6,2){8pt}{E}
  \cnode(7.5,2){5pt}{F}
   \cnode(7.5,0.5){5pt}{Fb} \cnode(7.5,3.5){5pt}{Ft}
  \cnode(9,2){5pt}{G}
  \cnode(10.5,2){5pt}{H}
  \cnode(12,2){5pt}{K} 
   \cnode(13.5,2){5pt}{L}  
  \cnode(15,2){5pt}{I}
  \cnode(16.5,2){5pt}{J}
  \cnode(18,2){5pt}{M}

  \ncline[nodesep=0pt]{A}{B}\naput[npos=-0.1]{$\vphantom{\Big(}D_{4,\varepsilon_1}$}\naput[npos=1.1]{$\vphantom{\Big(}D_{4,r\varepsilon_1}$}
  \ncline[nodesep=0pt]{B}{C}\naput[npos=1.1]{$\vphantom{\Big(}D_{4,r\varepsilon_2}$}
  \ncline[nodesep=0pt]{C}{D}\naput[npos=1.1]{$\vphantom{\Big(}  D_{4,\varepsilon_2}$}
  \ncline[nodesep=0pt,doubleline=true]{D}{E}
  \ncline[nodesep=0pt,doubleline=true]{E}{F}\naput[npos=0.6]{$\vphantom{\Big(} \mathrm{St}$}
  \ncline[nodesep=0pt]{F}{G}\naput[npos=1.1]{$\vphantom{\Big(} \phi_{27,37}$}
  \ncline[nodesep=0pt]{G}{H}\naput[npos=1.1]{$\vphantom{\Big(}\phi_{105,26}$}
  \ncline[nodesep=0pt]{H}{K}\naput[npos=1.1]{$\vphantom{\Big(}\phi_{189,17}$}
  \ncline[nodesep=0pt]{K}{L}\naput[npos=1.1]{$\vphantom{\Big(}\phi_{189,10}$}
  \ncline[nodesep=0pt]{L}{I}\naput[npos=1.1]{$\vphantom{\Big(}  \phi_{105,5}$}
  \ncline[nodesep=0pt]{I}{J}\naput[npos=1.1]{$\vphantom{\Big(}\phi_{27,2}$}
  \ncline[nodesep=0pt]{J}{M}\naput[npos=1.1]{$\vphantom{\Big(} 1$}
    \ncline[nodesep=0pt,doubleline=true]{F}{Fb}\ncput[npos=1.6]{$E_7[-{\II}]$}
  \ncline[nodesep=0pt,doubleline=true]{F}{Ft}\ncput[npos=1.6]{$E_7[{\II}]$}

\psellipticarc[linewidth=1pt]{->}(7.5,2)(0.7,0.7){15}{75}
 \psellipticarc[linewidth=1pt]{->}(7.5,2)(0.7,0.7){195}{255}
 \psellipticarc[linewidth=1pt]{->}(7.5,2)(0.7,0.7){285}{345}

\end{pspicture}
\end{center}
\caption{Brauer tree of the principal $\Phi_{14}$-block of $E_7(q)$}
\label{14E7}
\end{figure}
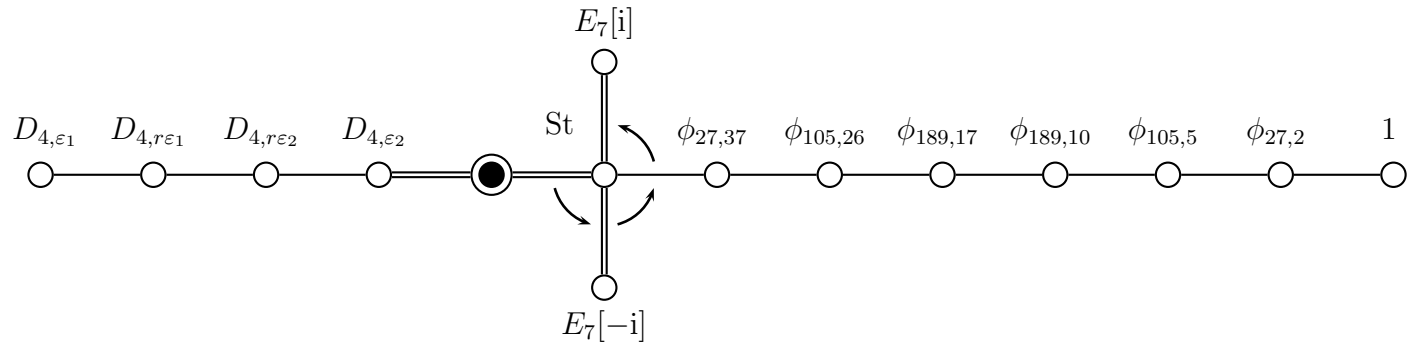

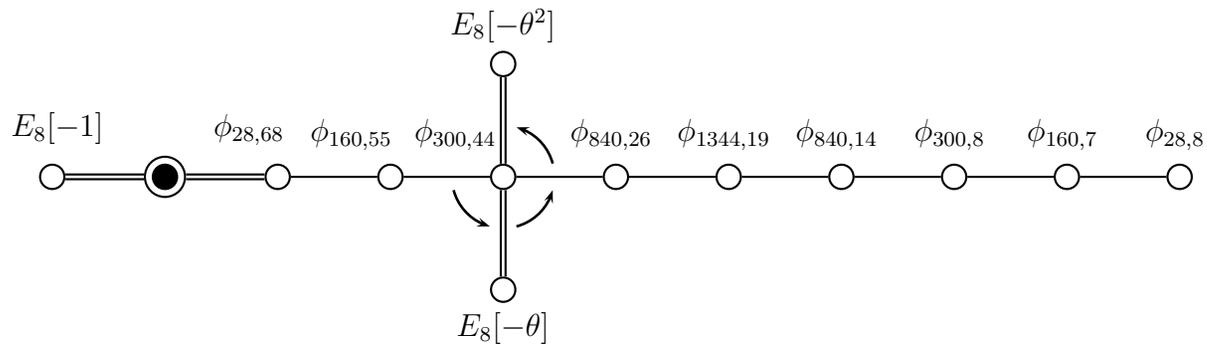
\begin{figure}[h!]
\begin{center}
\begin{pspicture}(15,4)

  \cnode[fillstyle=solid,fillcolor=black](1.5,2){5pt}{B2}
    \cnode(0,2){5pt}{A}
  \cnode(1.5,2){8pt}{B}
  \cnode(3,2){5pt}{C}
  \cnode(4.5,2){5pt}{D}
  \cnode(6,2){5pt}{E}
     \cnode(6,0.5){5pt}{Eb} \cnode(6,3.5){5pt}{Et}
  \cnode(7.5,2){5pt}{F}
  \cnode(9,2){5pt}{G}
  \cnode(10.5,2){5pt}{H}
  \cnode(12,2){5pt}{K} 
   \cnode(13.5,2){5pt}{L}  
  \cnode(15,2){5pt}{I}
  
  \ncline[nodesep=0pt,doubleline=true]{A}{B}\naput[npos=-0.1]{$\vphantom{\Big(}E_8[-1]$}
  \ncline[nodesep=0pt,doubleline=true]{B}{C}\naput[npos=0.8]{$\vphantom{\Big(}\phi_{28,68}$}
  \ncline[nodesep=0pt]{C}{D}\naput[npos=0.7]{$\vphantom{\Big(}  \phi_{160,55}$}
  \ncline[nodesep=0pt]{D}{E}\naput[npos=0.6]{$\vphantom{\Big(}  \phi_{300,44}$}
  \ncline[nodesep=0pt]{E}{F}\naput[npos=1.1]{$\vphantom{\Big(} \phi_{840,26}$}
  \ncline[nodesep=0pt]{F}{G}\naput[npos=1.1]{$\vphantom{\Big(} \phi_{1344,19}$}
  \ncline[nodesep=0pt]{G}{H}\naput[npos=1.1]{$\vphantom{\Big(}\phi_{840,14}$}
  \ncline[nodesep=0pt]{H}{K}\naput[npos=1.1]{$\vphantom{\Big(}\phi_{300,8}$}
  \ncline[nodesep=0pt]{K}{L}\naput[npos=1.1]{$\vphantom{\Big(}\phi_{160,7}$}
  \ncline[nodesep=0pt]{L}{I}\naput[npos=1.1]{$\vphantom{\Big(}  \phi_{28,8}$}
    \ncline[nodesep=0pt,doubleline=true]{E}{Eb}\ncput[npos=1.6]{$E_8[-\theta]$}
  \ncline[nodesep=0pt,doubleline=true]{E}{Et}\ncput[npos=1.6]{$E_8[-\theta^2]$}

\psellipticarc[linewidth=1pt]{->}(6,2)(0.7,0.7){15}{75}
 \psellipticarc[linewidth=1pt]{->}(6,2)(0.7,0.7){195}{255}
 \psellipticarc[linewidth=1pt]{->}(6,2)(0.7,0.7){285}{345}

\end{pspicture}
\end{center}
\caption{Brauer tree of the $\Phi_{12}$-block of $E_8(q)$ associated to  $(^3D_4(q),{}^3D_4[1])$ }
\label{12E8}
\end{figure}

\end{landscape}

\begin{landscape}
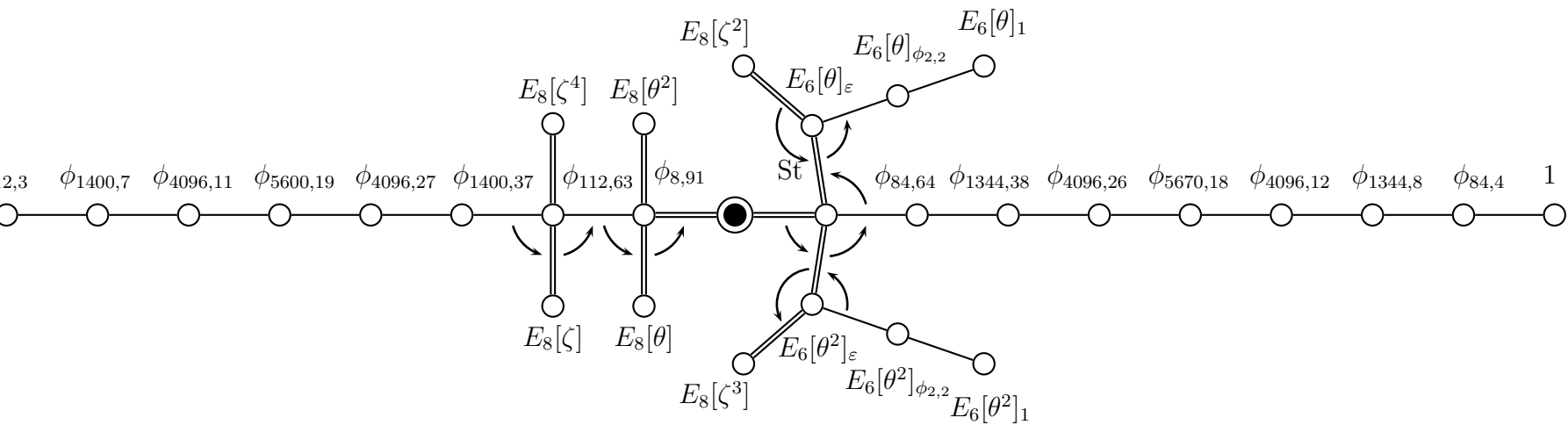
\begin{figure}[h!] 
\begin{center}
\begin{pspicture}(27,6)
\psset{unit=0.92cm}
  \cnode[fillstyle=solid,fillcolor=black](8.5,3){5pt}{A2}
    \cnode(8.5,3){8pt}{A}
  \cnode(10,3){5pt}{B}
  \cnode(9.77,4.47){5pt}{X1}
  \cnode(8.64,5.45){5pt}{X2}
  \cnode(11.18,4.96){5pt}{X3}
  \cnode(12.60,5.45){5pt}{X4}
    \cnode(9.77,1.53){5pt}{Y1}
  \cnode(8.64,0.55){5pt}{Y2}
  \cnode(11.18,1.04){5pt}{Y3}
  \cnode(12.60,0.55){5pt}{Y4}
  \cnode(11.5,3){5pt}{C}
  \cnode(13,3){5pt}{D}
  \cnode(14.5,3){5pt}{E}
  \cnode(16,3){5pt}{F}
  \cnode(17.5,3){5pt}{G}
  \cnode(19,3){5pt}{H}
  \cnode(20.5,3){5pt}{K} 
  \cnode(2.5,3){5pt}{L}
  \cnode(4,3){5pt}{M}
  \cnode(5.5,3){5pt}{N}
  \cnode(7,3){5pt}{O}
  \cnode(1,3){5pt}{P}
  \cnode(22,3){5pt}{Q}
   \cnode(7,1.5){5pt}{T}
  \cnode(7,4.5){5pt}{U}
    \cnode(-0.5,3){5pt}{V}
     \cnode(-2,3){5pt}{W}
  \cnode(-3.5,3){5pt}{X}
  \cnode(-5,3){5pt}{Y}
  \cnode(5.5,4.5){5pt}{N1}  \cnode(5.5,1.5){5pt}{N2}

   \ncline[nodesep=0pt,doubleline=true]{A}{B}\naput[npos=0.6]{$\vphantom{\Big(} \mathrm{St}$}
  \ncline[nodesep=0pt]{B}{C}\naput[npos=1]{$\vphantom{\Big(}\phi_{84,64}$}
  \ncline[nodesep=0pt]{C}{D}\naput[npos=0.9]{$\vphantom{\Big(}  \phi_{1344,38}$}
  \ncline[nodesep=0pt]{D}{E}\naput[npos=1]{$\vphantom{\Big(}\phi_{4096,26}$}
  \ncline[nodesep=0pt]{E}{F}\naput[npos=1.15]{$\vphantom{\Big(} \phi_{5670,18}$}
  \ncline[nodesep=0pt]{F}{G}\naput[npos=1.3]{$\vphantom{\Big(} \phi_{4096,12}$}
  \ncline[nodesep=0pt]{G}{H}\naput[npos=1.4]{$\vphantom{\Big(} \phi_{1344,8}$}
  \ncline[nodesep=0pt]{H}{K}\naput[npos=1.4]{$\vphantom{\Big(} \phi_{84,4} $}
  \ncline[nodesep=0pt]{L}{M}\naput[npos=0.2]{$\vphantom{\Big(} \phi_{4096,27} $}  
  \ncline[nodesep=0pt]{M}{N}\naput[npos=0.3]{$\vphantom{\Big(} \phi_{1400,37}$}
  \ncline[nodesep=0pt]{N}{O}\naput[npos=0.5]{$\vphantom{\Big(}\phi_{112,63} $}
  \ncline[nodesep=0pt,doubleline=true]{O}{A}\naput[npos=0.4]{$\vphantom{\Big(}\phi_{8,91}  $}
  \ncline[nodesep=0pt]{P}{L}\naput[npos=0.05]{$\vphantom{\Big(}\phi_{5600,19} $}
    \ncline[nodesep=0pt]{V}{P}\naput[npos=-0.1]{$\vphantom{\Big(}\phi_{4096,11} $}
  \ncline[nodesep=0pt]{W}{V}\naput[npos=-0.2]{$\vphantom{\Big(}\phi_{1400,7} $}
  \ncline[nodesep=0pt]{X}{W}\naput[npos=-0.3]{$\vphantom{\Big(}\phi_{112,3} $}
  \ncline[nodesep=0pt]{Y}{X}\naput[npos=-0.1]{$\vphantom{\Big(}\phi_{8,1} $}
  \ncline[nodesep=0pt]{K}{Q}\naput[npos=1.1]{$\vphantom{\Big(} 1 $}
    \ncline[nodesep=0pt,doubleline=true]{O}{T}\ncput[npos=1.65]{$\vphantom{\Big(} E_8[\theta] $}
  \ncline[nodesep=0pt,doubleline=true]{O}{U}\ncput[npos=1.65]{$\vphantom{\Big(} E_8[\theta^2] $}
      \ncline[nodesep=0pt,doubleline=true]{N}{N1}\ncput[npos=1.65]{$\vphantom{\Big(} E_8[\zeta^4] $}
  \ncline[nodesep=0pt,doubleline=true]{N}{N2}\ncput[npos=1.65]{$\vphantom{\Big(} E_8[\zeta] $}
    \ncline[nodesep=0pt,doubleline=true]{B}{X1}\ncput[npos=1.85]{$\phantom{\Big(e} E_6[\theta]_\varepsilon $}
      \ncline[nodesep=0pt,doubleline=true]{B}{Y1}\ncput[npos=1.85]{$\phantom{\Big(i} E_6[\theta^2]_\varepsilon $}
      \ncline[nodesep=0pt,doubleline=true]{X1}{X2}\ncput[npos=1.9]{$\phantom{\Big(} E_8[\zeta^2]$}
      \ncline[nodesep=0pt,doubleline=true]{Y1}{Y2}\ncput[npos=1.9]{$\phantom{\Big(} E_8[\zeta^3]$}
      \ncline[nodesep=0pt]{X1}{X3}\naput[npos=1.6]{$ E_6[\theta]_{\phi_{2,2}} $}
      \ncline[nodesep=0pt]{Y1}{Y3}\nbput[npos=1.6]{$ E_6[\theta^2]_{\phi_{2,2}}$}
      \ncline[nodesep=0pt]{X3}{X4}\naput[npos=1.6]{$E_6[\theta]_1$}
      \ncline[nodesep=0pt]{Y3}{Y4}\nbput[npos=1.6]{$E_6[\theta^2]_1$}
      
 \psellipticarc[linewidth=1pt]{->}(5.5,3)(0.7,0.7){195}{255}
 \psellipticarc[linewidth=1pt]{->}(5.5,3)(0.7,0.7){285}{345}
  \psellipticarc[linewidth=1pt]{->}(7,3)(0.7,0.7){195}{255}
 \psellipticarc[linewidth=1pt]{->}(7,3)(0.7,0.7){285}{345}
  \psellipticarc[linewidth=1pt]{->}(10,3)(0.7,0.7){195}{245}
 \psellipticarc[linewidth=1pt]{->}(10,3)(0.7,0.7){275}{345}
  \psellipticarc[linewidth=1pt]{->}(10,3)(0.7,0.7){15}{85}
 \psellipticarc[linewidth=1pt]{->}(9.77,4.47)(0.6,0.6){295}{10}
  \psellipticarc[linewidth=1pt]{->}(9.77,4.47)(0.6,0.6){150}{265}
   \psellipticarc[linewidth=1pt]{->}(9.77,1.53)(0.6,0.6){-10}{-295}
  \psellipticarc[linewidth=1pt]{->}(9.77,1.53)(0.6,0.6){-265}{-150}     
\end{pspicture}
\end{center}
\caption{Brauer tree of the principal $\Phi_{15}$-block of $E_8(q)$}
\label{15E8}
\end{figure}

\begin{figure}[h!]
\begin{center}
\begin{pspicture}(21,4)

  \cnode[fillstyle=solid,fillcolor=black](6.5,2){5pt}{F2}
    \cnode(-1,2){5pt}{A}
  \cnode(0.5,2){5pt}{B}
  \cnode(2,2){5pt}{C}
  \cnode(3.5,2){5pt}{D}
  \cnode(5,2){5pt}{E}
  \cnode(6.5,2){8pt}{F}
   \cnode(6.5,0.5){5pt}{Fb} \cnode(6.5,3.5){5pt}{Ft} 
  \cnode(8,2){5pt}{G}
  \cnode(9.5,2){5pt}{H}
     \cnode(9.5,0.5){5pt}{Hb} \cnode(9.5,3.5){5pt}{Ht}
  \cnode(11,2){5pt}{K} 
   \cnode(12.5,2){5pt}{L}  
  \cnode(14,2){5pt}{I}
  \cnode(15.5,2){5pt}{J}
  \cnode(17,2){5pt}{M}
  \cnode(18.5,2){5pt}{N}
  \cnode(20,2){5pt}{O}

  \ncline[nodesep=0pt]{A}{B}\naput[npos=-0.1]{$\vphantom{\Big(}D_{4,\phi_{1,12}'}$}\naput[npos=1.1]{$\vphantom{\Big(}D_{4,\phi_{4,7}'}$}
  \ncline[nodesep=0pt]{B}{C}\naput[npos=1.1]{$\vphantom{\Big(}D_{4,\phi_{6,6}'}$}
  \ncline[nodesep=0pt]{C}{D}\naput[npos=1.1]{$\vphantom{\Big(}  D_{4,\phi_{4,7}''}$}
  \ncline[nodesep=0pt]{D}{E}\naput[npos=1.1]{$\vphantom{\Big(}D_{4,\phi_{1,12}''}$}
  \ncline[nodesep=0pt,doubleline=true]{E}{F}  
  \ncline[nodesep=0pt,doubleline=true]{F}{G}\naput[npos=0.8]{$\vphantom{\Big(} \phi_{8,91}$}
  \ncline[nodesep=0pt]{G}{H}\naput[npos=0.6]{$\vphantom{\Big(} \phi_{35,74}$}
  \ncline[nodesep=0pt]{H}{K}\naput[npos=1.1]{$\vphantom{\Big(}\phi_{300,44}$}
  \ncline[nodesep=0pt]{K}{L}\naput[npos=1.1]{$\vphantom{\Big(}\phi_{840,31}$}
  \ncline[nodesep=0pt]{L}{I}\naput[npos=1.1]{$\vphantom{\Big(}  \phi_{1134,20}$}
  \ncline[nodesep=0pt]{I}{J}\naput[npos=1.1]{$\vphantom{\Big(}\phi_{840,13}$}
  \ncline[nodesep=0pt]{J}{M}\naput[npos=1.1]{$\vphantom{\Big(} \phi_{300,8}$}
  \ncline[nodesep=0pt]{M}{N}\naput[npos=1.1]{$\vphantom{\Big(}\phi_{35,2}$}
  \ncline[nodesep=0pt]{N}{O}\naput[npos=1.1]{$\vphantom{\Big(}\phi_{8,1}$}
  \ncline[nodesep=0pt,doubleline=true]{F}{Fb}\ncput[npos=1.6]{$E_8[\theta^2]$}
  \ncline[nodesep=0pt,doubleline=true]{F}{Ft}\ncput[npos=1.6]{$E_8[\theta]$}
  \ncline[nodesep=0pt,doubleline=true]{H}{Hb}\ncput[npos=1.6]{$E_8[-\theta]$}
  \ncline[nodesep=0pt,doubleline=true]{H}{Ht}\ncput[npos=1.6]{$E_8[-\theta^2]$}

\psellipticarc[linewidth=1pt]{->}(6.5,2)(0.7,0.7){15}{75}
 \psellipticarc[linewidth=1pt]{->}(6.5,2)(0.7,0.7){195}{255}
 \psellipticarc[linewidth=1pt]{->}(6.5,2)(0.7,0.7){285}{345}
  \psellipticarc[linewidth=1pt]{->}(6.5,2)(0.7,0.7){105}{165}
\psellipticarc[linewidth=1pt]{->}(9.5,2)(0.7,0.7){15}{75}
 \psellipticarc[linewidth=1pt]{->}(9.5,2)(0.7,0.7){195}{255}
 \psellipticarc[linewidth=1pt]{->}(9.5,2)(0.7,0.7){285}{345}

\end{pspicture}
\end{center}
\caption{Brauer tree of the $\Phi_{18}$-block of $E_8(q)$ associated to  $({}^2 A_2(q),\phi_{21})$}
\label{18E8}
\end{figure}

\end{landscape}

\pagebreak 

\begin{landscape}
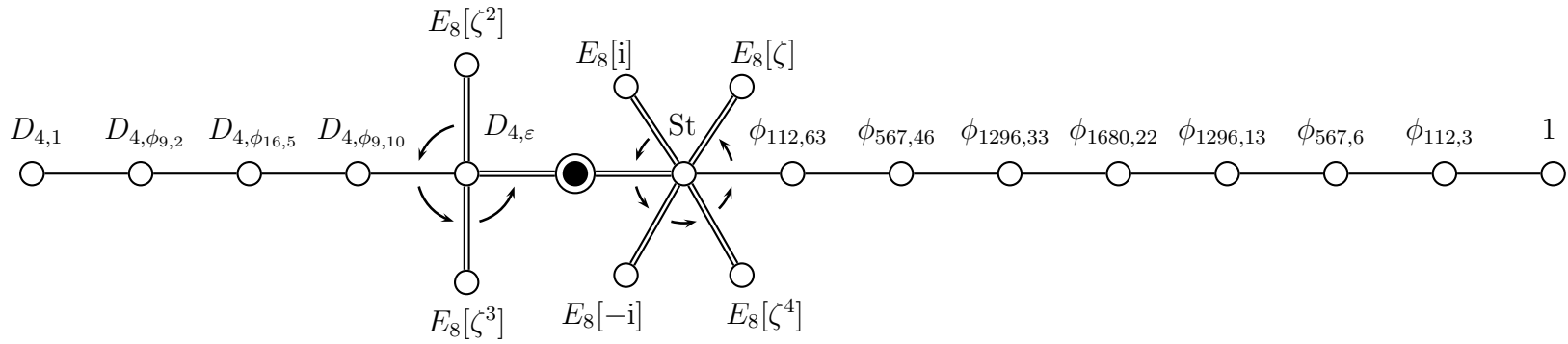
\begin{figure}[h!] \vskip -1.2cm 
\begin{center}
\begin{pspicture}(17,5)

  \cnode[fillstyle=solid,fillcolor=black](5,2){5pt}{A2}
    \cnode(5,2){8pt}{A}
  \cnode(6.5,2){5pt}{B}
  \cnode(8,2){5pt}{C}
  \cnode(9.5,2){5pt}{D}
  \cnode(11,2){5pt}{E}
  \cnode(12.5,2){5pt}{F}
  \cnode(14,2){5pt}{G}
  \cnode(15.5,2){5pt}{H}
  \cnode(17,2){5pt}{K} 
  \cnode(-1,2){5pt}{L}
  \cnode(0.5,2){5pt}{M}
  \cnode(2,2){5pt}{N}
  \cnode(3.5,2){5pt}{O}
  \cnode(-2.5,2){5pt}{P}
  \cnode(18.5,2){5pt}{Q}
    \cnode(7.3,0.6){5pt}{I}
  \cnode(7.3,3.2){5pt}{J}
    \cnode(5.7,0.6){5pt}{R}
  \cnode(5.7,3.2){5pt}{S}
      \cnode(3.5,0.5){5pt}{T}
  \cnode(3.5,3.5){5pt}{U}
  
   \ncline[nodesep=0pt,doubleline=true]{A}{B}\naput[npos=1.15]{$\vphantom{\Big(} \mathrm{St}$}
  \ncline[nodesep=0pt]{B}{C}\naput[npos=1.1]{$\vphantom{\Big(}\phi_{112,63}$}
  \ncline[nodesep=0pt]{C}{D}\naput[npos=1.1]{$\vphantom{\Big(}  \phi_{567,46}$}
  \ncline[nodesep=0pt]{D}{E}\naput[npos=1.1]{$\vphantom{\Big(}\phi_{1296,33}$}
  \ncline[nodesep=0pt]{E}{F}\naput[npos=1.1]{$\vphantom{\Big(} \phi_{1680,22}$}
  \ncline[nodesep=0pt]{F}{G}\naput[npos=1.1]{$\vphantom{\Big(} \phi_{1296,13}$}
  \ncline[nodesep=0pt]{G}{H}\naput[npos=1.1]{$\vphantom{\Big(} \phi_{567,6}$}
  \ncline[nodesep=0pt]{H}{K}\naput[npos=1.1]{$\vphantom{\Big(} \phi_{112,3} $}
  \ncline[nodesep=0pt]{L}{M}\naput[npos=-0.1]{$\vphantom{\Big(} D_{4,\phi_{9,2}} $}  
  \ncline[nodesep=0pt]{M}{N}\naput[npos=-0.1]{$\vphantom{\Big(} D_{4,\phi_{16,5}} $}
  \ncline[nodesep=0pt]{N}{O}\naput[npos=-0.1]{$\vphantom{\Big(}D_{4,\phi_{9,10}} $}
  \ncline[nodesep=0pt,doubleline=true]{O}{A}\naput[npos=0.4]{$\vphantom{\Big(}D_{4,\varepsilon}  $}
  \ncline[nodesep=0pt]{P}{L}\naput[npos=-0.1]{$\vphantom{\Big(}D_{4,1}  $}
  \ncline[nodesep=0pt]{K}{Q}\naput[npos=1.1]{$\vphantom{\Big(} 1 $}
  \ncline[nodesep=0pt,doubleline=true]{B}{I}\ncput[npos=1.65]{$\vphantom{\Big(} E_8[\zeta^4] $}
  \ncline[nodesep=0pt,doubleline=true]{B}{J}\ncput[npos=1.65]{$\vphantom{\Big(} E_8[\zeta] $}
  \ncline[nodesep=0pt,doubleline=true]{B}{R}\ncput[npos=1.65]{$\vphantom{\Big(} E_8[-{\II}] $}
  \ncline[nodesep=0pt,doubleline=true]{B}{S}\ncput[npos=1.65]{$\vphantom{\Big(} E_8[{\II}] $}
    \ncline[nodesep=0pt,doubleline=true]{O}{T}\ncput[npos=1.65]{$\vphantom{\Big(} E_8[\zeta^3] $}
  \ncline[nodesep=0pt,doubleline=true]{O}{U}\ncput[npos=1.65]{$\vphantom{\Big(} E_8[\zeta^2] $}
  
    \psellipticarc[linewidth=1pt]{->}(3.5,2)(0.7,0.7){105}{165}
 \psellipticarc[linewidth=1pt]{->}(3.5,2)(0.7,0.7){195}{255}
 \psellipticarc[linewidth=1pt]{->}(3.5,2)(0.7,0.7){285}{345}
     \psellipticarc[linewidth=1pt]{->}(6.5,2)(0.7,0.7){15}{45}
 \psellipticarc[linewidth=1pt]{->}(6.5,2)(0.7,0.7){135}{165}
 \psellipticarc[linewidth=1pt]{->}(6.5,2)(0.7,0.7){255}{285}
  \psellipticarc[linewidth=1pt]{->}(6.5,2)(0.7,0.7){195}{225}
 \psellipticarc[linewidth=1pt]{->}(6.5,2)(0.7,0.7){315}{345}
 
\end{pspicture}
\end{center}

\caption{Brauer tree of the principal $\Phi_{20}$-block of $E_8(q)$}
\label{20E8}
\end{figure}

\begin{figure}[h!]  
\begin{center}
\begin{pspicture}(17,7.2)

  \cnode[fillstyle=solid,fillcolor=black](5,3.9){5pt}{A2}
      \cnode(5,3.9){8pt}{A}
  \cnode(6.5,3.9){5pt}{B}
  \cnode(8,3.9){5pt}{C}
  \cnode(9.5,3.9){5pt}{D}
  \cnode(11,3.9){5pt}{E}
  \cnode(12.5,3.9){5pt}{F}
  \cnode(14,3.9){5pt}{G}
  \cnode(15.5,3.9){5pt}{H}
  \cnode(17,3.9){5pt}{K}  
  \cnode(7.3,2.7){5pt}{I}
  \cnode(7.3,5.1){5pt}{J}
  \cnode(-1,3.9){5pt}{L}
  \cnode(0.5,3.9){5pt}{M}
  \cnode(2,3.9){5pt}{N}
  \cnode(3.5,3.9){5pt}{O}
  \cnode(-2.5,3.9){5pt}{P}
  \cnode(18.5,3.9){5pt}{Q}
  \cnode(5.8,5.1){5pt}{R}
  \cnode(5.8,2.7){5pt}{S}
  \cnode(3.9,5){5pt}{T}
  \cnode(3.9,2.8){5pt}{U}
  \cnode(2.8,1.7){5pt}{V}
  \cnode(2.8,6.1){5pt}{W}
  \cnode(1.7,0.6){5pt}{X}
  \cnode(1.7,7.2){5pt}{Y}

   \ncline[nodesep=0pt,doubleline=true]{A}{B}\naput[npos=0.8]{$\vphantom{\Big(} \mathrm{St}$}
  \ncline[nodesep=0pt]{B}{C}\naput[npos=1.1]{$\vphantom{\Big(}\phi_{35,74}$}
  \ncline[nodesep=0pt]{C}{D}\naput[npos=1.1]{$\vphantom{\Big(}  \phi_{160,55}$}
  \ncline[nodesep=0pt]{D}{E}\naput[npos=1.1]{$\vphantom{\Big(}\phi_{350,38}$}
  \ncline[nodesep=0pt]{E}{F}\naput[npos=1.1]{$\vphantom{\Big(} \phi_{448,25}$}
  \ncline[nodesep=0pt]{F}{G}\naput[npos=1.1]{$\vphantom{\Big(} \phi_{350,14}$}
  \ncline[nodesep=0pt]{G}{H}\naput[npos=1.1]{$\vphantom{\Big(} \phi_{160,7}$}
  \ncline[nodesep=0pt,doubleline=true]{B}{I}\ncput[npos=1.65]{$\phantom{\Big(mm} E_8[-\theta] $}
  \ncline[nodesep=0pt,doubleline=true]{B}{J}\ncput[npos=1.65]{$\phantom{\Big(mim} E_8[-\theta^2] $}
  \ncline[nodesep=0pt]{H}{K}\naput[npos=1.1]{$\vphantom{\Big(} \phi_{35,2} $}
  \ncline[nodesep=0pt]{L}{M}\naput[npos=-0.1]{$\vphantom{\Big(} D_{4,\phi_{8,3}'} $}  
  \ncline[nodesep=0pt]{M}{N}\naput[npos=-0.1]{$\vphantom{\Big(} D_{4,\phi_{12,4}} $}
  \ncline[nodesep=0pt]{N}{O}\naput[npos=-0.1]{$\vphantom{\Big(}D_{4,\phi_{8,9}''} $}
  \ncline[nodesep=0pt,doubleline=true]{O}{A}\naput[npos=-0.1]{$\vphantom{\Big(}D_{4,\phi_{2,16}''}  $}
  \ncline[nodesep=0pt]{P}{L}\naput[npos=-0.1]{$\vphantom{\Big(}D_{4,\phi_{2,4}'}  $}
  \ncline[nodesep=0pt]{K}{Q}\naput[npos=1.1]{$\vphantom{\Big(} 1 $}
  \ncline[nodesep=0pt,doubleline=true]{R}{A}\ncput[npos=-0.7]{$\phantom{\Big(im} E_8[{\II}] $}
  \ncline[nodesep=0pt,doubleline=true]{S}{A}\ncput[npos=-0.7]{$\phantom{\Big(mm} E_8[-{\II}] $}
  \ncline[nodesep=0pt,doubleline=true]{T}{A}\naput[npos=-0.1]{$ E_6[\theta]_{\phi_{1,3}''} $}
  \ncline[nodesep=0pt,doubleline=true]{U}{A}\nbput[npos=-0.1]{$ E_6[\theta^2]_{\phi_{1,3}''} $}
  \ncline[nodesep=0pt]{V}{U}\nbput[npos=-0.1]{$E_6[\theta^2]_{\phi_{2,2}} $}
  \ncline[nodesep=0pt]{W}{T}\naput[npos=-0.1]{$E_6[\theta]_{\phi_{2,2}}$}
  \ncline[nodesep=0pt]{X}{V}\nbput[npos=-0.1]{$ E_6[\theta^2]_{\phi_{1,3}'} $}
  \ncline[nodesep=0pt]{Y}{W}\naput[npos=0.1]{$E_6[\theta]_{\phi_{1,3}'} $}

    \psellipticarc[linewidth=1pt]{->}(6.5,3.9)(0.7,0.7){195}{285}
 \psellipticarc[linewidth=1pt]{->}(6.5,3.9)(0.7,0.7){315}{345}
 \psellipticarc[linewidth=1pt]{->}(6.5,3.9)(0.7,0.7){15}{45}
     \psellipticarc[linewidth=1pt]{->}(5,3.9)(0.7,0.7){15}{45}
          \psellipticarc[linewidth=1pt]{->}(5,3.9)(0.7,0.7){75}{120}
 \psellipticarc[linewidth=1pt]{->}(5,3.9)(0.7,0.7){240}{295}
 \psellipticarc[linewidth=1pt]{->}(5,3.9)(0.7,0.7){315}{345}

\end{pspicture}
\end{center} 
\caption{Brauer tree of the principal $\Phi_{24}$-block of $E_8(q)$}
\label{24E8}
\end{figure}
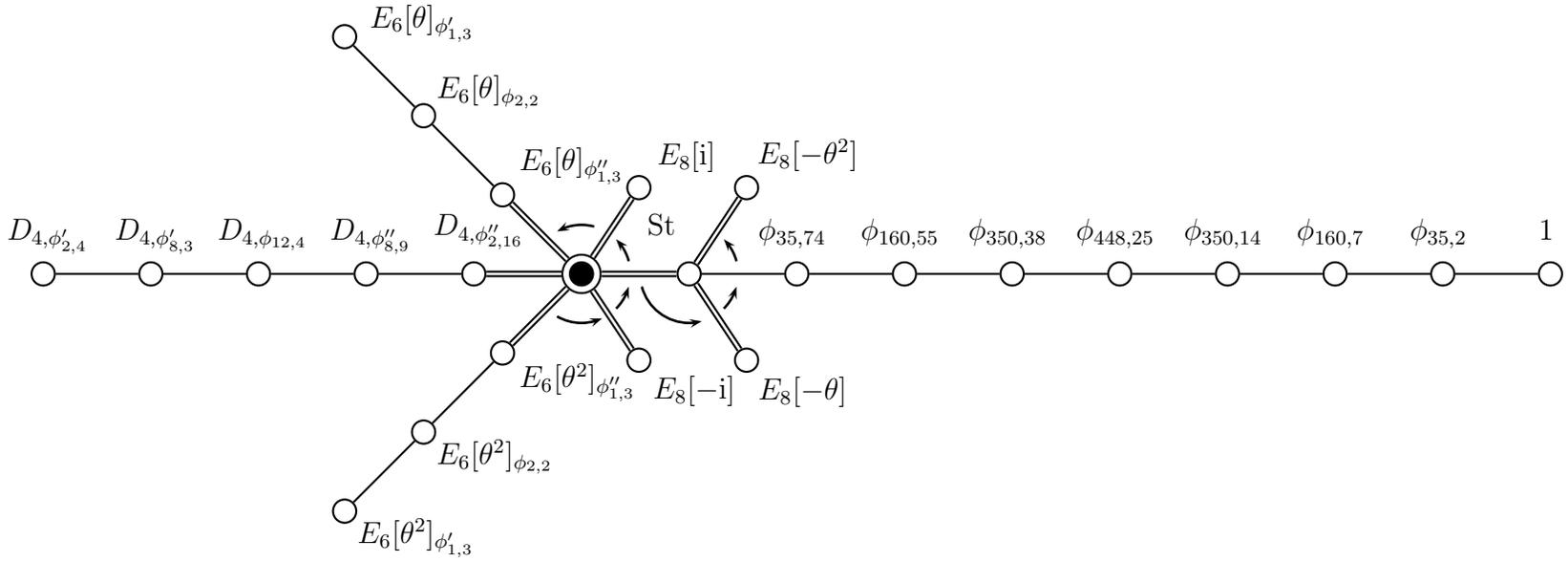

\end{landscape}

\def\cprime{$'$}

\end{document}